\documentclass[10pt]{article}
\usepackage[utf8]{inputenc}

\usepackage{amsfonts,amsmath, amssymb,amsthm,amscd}
\usepackage{tikz}
\usepackage{color}
\usepackage{mathtools}
\usepackage{dsfont}
\usepackage{comment}
\usepackage{geometry}

\newtheorem{theorem}{Theorem}
\newtheorem{proposition}{Proposition}
\newtheorem{corollary}{Corollary}

\newtheorem{lemma}{Lemma}
\newtheorem{definition}{Definition}

\newtheorem{remark}{Remark}

\newcommand{\E}{{\mathbb E }}

\newcommand{\tr}{{\text{Tr} }}
\renewcommand{\Re}{{\text{Re} }}
\renewcommand{\Im}{{\text{Im} }}
\renewcommand{\i}{{\text{i} }}

\allowdisplaybreaks
\begin{document}

 \begin{minipage}{0.85\textwidth}
\end{minipage}
\begin{center}
	\large\bf
Spherical Sherrington-Kirkpatrick   model for deformed Wigner matrix with fast decaying edges
\end{center}

\begin{center} 
	\begin{minipage}{0.3\textwidth}
	\begin{center}
		Ji Oon Lee\\
		\footnotesize 
		{Korea Advanced Institute of Science and Technology}\\
		{\it jioon.lee@kaist.edu}
	\end{center}
\end{minipage} 
		\begin{minipage}{0.3\textwidth}
			\begin{center}
				Yiting Li  \\
				\footnotesize 
				{Korea Advanced Institute of Science and Technology}\\
				{\it yitingli@kaist.ac.kr}
			\end{center}
		\end{minipage} 
\end{center}

\begin{abstract}
We consider the $2$-spin spherical Sherrington--Kirkpatrick model whose disorder is given by a deformed Wigner matrix of the form $W+\lambda V$, where $W$ is a Wigner matrix and $V$ is a random diagonal matrix with i.i.d. entries. Assuming that the density function of the entries of $V$ decays faster than a certain rate near the edges of its spectrum, we prove the sharp phase transition of the limiting free energy and its fluctuation. In the high temperature regime, the fluctuation of $F_N$ converges in distribution to a Gaussian distribution, whereas it converges to a Weibull distribution in the low temperature regime. We also prove several results for deformed Wigner matrices, including a local law for the resolvent entries,  a central limit theorem of the linear spectral statistics, and a theorem on the rigidity of eigenvalues.
\end{abstract}

\tableofcontents

\section{Introduction}

The Sherrington--Kirkpatrick (SK) model and its variants have been intensively studied in statistical physics and probability theory to understand the behavior of spin glass. Its spherical variant, known as the spherical Sherrington--Kirkpatrick (SSK) model, is defined through the mean-field Hamiltonian of the form
\begin{align}\label{eq84}
-\langle J\sigma,\sigma\rangle,
\end{align}
where the disorder $J$ is an $N \times N$ matrix and the spin $\sigma=(\sigma_1,\ldots,\sigma_N) \in S_{N-1}=\{(\sigma_1,\ldots,\sigma_N)\in{\mathbb R }^N : \sum \sigma_i^2=N\}$. The SSK model is widely used in various fields of study including high-dimensional statistics and learning theory.

One of the key features of the SSK model (and the SK model) is the sharp phase transition of the free energy, defined as
\begin{align}\label{eq87}
F_N=F_N(\beta)=\frac{1}{N}\log\Bigg[\int_{S_{N-1}}\exp\Big(\beta\langle\sigma,J\sigma\rangle\Big)d\omega_N(\sigma)\Bigg],
\end{align}
where $\beta$ is the inverse temperature and $\omega_N$ is the normalized uniform measure on $S_{N-1}$. When the disorder $J$ is a real Wigner matrix, it was proved by Crisanti and Sommers \cite{Crisanti+Sommers}, and Talagrand \cite{Talagrand} that as $N\to\infty$ the free energy $F_N$ converges to 
\begin{align}
F_N\to F_W(\beta):=\begin{cases}
  \beta^2\quad&\text{if }0<\beta\le 1/2\\
 2\beta-\frac{1}{2}\log(2\beta)-\frac{3}{4}\quad&\text{if } \beta > 1/2
\end{cases}.
\end{align}
The fluctuation of the free energy is also markedly different in the high temperature case ($\beta < 1/2$) and the low temperature case ($\beta > 1/2$). 
Baik and the first author \cite{Baik+Lee} studied the fluctuation $F_N-F_W(\beta)$ and proved that 
\begin{align}\label{eq85}
\begin{cases}
	N(F_N-F_W(\beta))\rightarrow \text{ a normal distribution}\quad&\text{if }0<\beta<1/2\\
2^{2/3}(\beta-\frac{1}{2})N^{2/3}(F_N-F_W(\beta))\rightarrow \text{ the Tracy--Widom distribution }\quad&\text{if }\beta>1/2
\end{cases}
\end{align}
where the convergence is in distribution.

Heuristically, the fluctuation of the free energy in the high temperature regime is affected by all eigenvalues of $J$ through its linear spectral statistics (LSS), defined by
\[
\sum_{i=1}^N f(\lambda_i),
\]
where $\lambda_1 \geq \lambda_2 \geq \dots \geq \lambda_N$ are the eigenvalues of $J$. On the other hand, in the low temperature regime, the fluctuation of $F_N$ is dominated by that of the largest $\lambda_1$. Since the fluctuations of the LSS and the largest eigenvalue are given by a Gaussian and the Tracy--Widom, respectively, one obtains the phase transition as in \eqref{eq85}. Similar argument also holds for other disorders such as the sample covariance matrix and the orthogonal invariant ensemble \cite{Baik+Lee}.

One natural question about the free energy of the SSK model is whether the heuristic argument above is universal, i.e., the picture of the all eigenvalues versus the largest eigenvalue is valid even when the disorder is not one of the classical random matrix models (Wigner matrix, sample covariance matrix, and invariant ensemble). To test the universality, we consider the case where the disorder is of the form
\begin{align}\label{eq83}
J = W+\lambda V,
\end{align} 
where $W$ is a Wigner matrix and $V$ is a random diagonal matrix. Such a matrix is called a deformed Wigner matrix, and with certain choices of the parameters, it is known that several key assumptions in \cite{Baik+Lee} are not satisfied, most notably the square-root decay at the edge of the spectrum and the Tracy--Widom limit of the largest eigenvalue. 

\subsection{Main contribution} 

Under the assumption that the decay of the spectrum of $J = W+\lambda V$ is convex, we prove that there exists a critical inverse temperature $\beta_c$ such that
\begin{itemize}
\item if $\beta < \beta_c$, the fluctuation of $F_N$ converges in distribution to a Gaussian distribution, and
\item if $\beta > \beta_c$, the fluctuation of $F_N$ converges in distribution to a Weibull distribution
\end{itemize}
with precise formulas for both limiting distributions, where the limiting Weibull distribution is originated from the corresponding (Weibull) distribution of the largest eigenvalue of $J$. This in particular suggests that the dichotomy between the fluctuation given by the LSS and that by the largest eigenvalue holds not only for the classical random matrix models but for more general models. We also prove the limiting free energy $F(\beta)$ for both regimes.

It should be noted that the order of the fluctuation in the low temperature regime is $N^{-1/(b+1)}$ for some $b>1$ but that in the high temperature is $N^{-1/2}$, and hence the fluctuation is larger in the low temperature regime than in the high temperature regime. This was also true for the SSK model with the Wigner disorder in \eqref{eq85}, though the exact orders of the fluctuations ($N^{-2/3}$ in the low temperature regime and $N^{-1}$ in the high temperature regime) do not coincide with those for our model.

The main technical difficulty in the proof of the main result is the lack of several results for $J$, which are crucial in the analysis of the free energy in \cite{Baik+Lee}. In this paper, we prove the following for $J$:
\begin{itemize}
	\item a local law for resolvent entries,
	\item a central limit theorem of linear statistics,
	\item the rigidity of eigenvalues.
\end{itemize}
These results are not only important for the understanding of the free energy but also significant per se in random matrix theory.

\subsection{Related works}

\subsubsection{Spherical Sherrington--Kirkpatrick model} \label{sec:background_of_ssk} 

The SK model was introduced by Sherrington--Kirkpatrick \cite{Sherrington+Kirkpatrick} as a mean-field version of the Edwards--Anderson model \cite{Edwards+Anderson}, which is an Ising-type model of spin glass. The limiting free energy was first predicted by Parisi \cite{Parisi}, which is now known as the Parisi formula, and later proved by Guerra \cite{Guerra} and Talagrand \cite{Talagrand_SK}.

The spherical Sherrington–Kirkpatrick (SSK) model was introduced by Kosterlitz, Thouless, and Jones \cite{Kosterlitz+Thouless+Jones_____}, where the limiting free energy was explicitly computed without a rigorous proof. A formula analogous to the Parisi formula was obtained by Crisanti and Sommers \cite{Crisanti+Sommers} and later proved by Talagrand \cite{Talagrand}. For more recent results on the free energy and its fluctuation for the SSK model, we refer to \cite{Baik+Lee2_____,Baik+Lee3_____,Baik+Lee+Wu_____,Landon+Sosoe_____,Landon+Sosoe2_____,Nguyen+Sosoe_____}.

\subsubsection{Deformed Wigner matrix}

Deformed Wigner matrix of the form \eqref{eq83} was first introduced by Pastur \cite{Pastur}, where it was proved that the empirical distribution of \eqref{eq83} converges to a deterministic probability distribution $\mu_{fc}$ as $N\to\infty$. The $\mu_{fc}$ is known as the free convolution of $\mu$ and the semicircle distribution, and assuming that the empirical distribution of $V$ is bounded and exhibits concave decay at the edge of its spectrum, it is known that $\mu_{fc}$ exhibits square-root decay at the corresponding edge. In this case, several key results for the Wigner matrix, including the local law for the resolvent \cite{Lee+Schnelli2,Lee+Schnelli+Stetler+Yau}, the delocalization of the eigenvectors and the rigidity of the eigenvalues \cite{Lee+Schnelli2}, the bulk universality \cite{Lee+Schnelli+Stetler+Yau}, the edge universality \cite{Lee+Schnelli+Stetler+Yau,Lee+Schnelli_edge_universality}, and the normality of the LSS \cite{Ji+Lee}, hold with natural modification.

On the other hand, much less is known for the case where $\mu_{fc}$ does not exhibit the square-root decay at the edge. It was proved by the first author and Schnelli \cite{Lee+Schnelli2,Lee+Schnelli} that $\mu_{fc}$ decays at the same rates as the empirical distribution of $V$ if it is convex and $\lambda$ in \eqref{eq83} is larger than a certain critical value $\lambda_+$. (See Lemma \ref{lemma:properties_of_mu_fc} for more detail.) In this case, it is also known that the eigenvectors associated to the extreme eigenvalues are partially localized \cite{Lee+Schnelli}.

\subsection{Relation to a signal detection problem}
	
	The SSK model is closely related to the problem of detecting the presence of the rank-one signal in a noisy data matrix. Suppose that the data matrix $M$ is of the form
	\[
	M = \sqrt{\lambda} xx^T + H,
	\]
	where the signal $x \in {\mathbb R }^N$ and the noise $H$ is an $N \times N$ real symmetric random matrix. When the signal-to-noise (SNR) $\lambda$ is not large, in order to detect the signal, it is common to analyze the largest eigenvalue and its associated eigenvalue, which is the principal component analysis (PCA). In the simplest case where $H$ is a Wigner matrix and $\|x\|=1$, the following transition for the largest eigenvalue $\lambda_1$ of $M$ is known; if $\lambda > 1$, $\lambda_1$ is strictly larger than $2$ and separates from the bulk of the spectrum, whereas if $\lambda < 1$, $\lambda_1$ converges to $2$, the edge of the spectrum, and cannot be distinguished from the null model ($\lambda=0$).
	
	If the SNR $\lambda$ is below the threshold and the noise is Gaussian, it is known that no tests can reliably detect the presence of the signal. For this case, it is natural to consider the hypothesis testing between the null hypothesis $\lambda = 0$ and the alternative $\lambda = \omega$ for some positive constant $\omega$, which is also known as the weak detection. By the Neyman--Pearson lemma, the likelihood ratio (LR) test is optimal in the sense that it minimizes the sum of the Type-I error and the Type II-error. For the $(i, j)$-entry of the data matrix with $i\neq j$, the ratio of the densities under the null and the alternative is
	\[
	\frac{\exp \big( N (M_{ij} - \sqrt{\lambda} x_i x_j)^2 \big)}{\exp \big( N M_{ij}^2 \big)}.
	\]
	Assuming that the signal is chosen uniformly from the unit sphere $S^N$ and the noise is GOE, the likelihood ratio is given by
	\begin{equation} \begin{split} \label{eq:LR}
			\frac{d {\mathbb P}_1}{d {\mathbb P}_0} &:= \int_{S^N} \prod_{i<j} \frac{\exp \big( N (M_{ij} - \sqrt{\lambda} x_i x_j)^2 \big)}{\exp \big( N M_{ij}^2 \big)} \prod_k \frac{\exp \big( N (M_{ij} - \sqrt{\lambda} x_i x_j)^2 /2\big)}{\exp \big( N M_{ij}^2 /2\big)} d\omega^N(\sigma) \\
			&= \int_{S^N} \prod_{i \neq j} \exp \left( -N \sqrt{\lambda} M_{ij} x_i x_j + \frac{N}{2} \lambda x_i^2 x_j^2 \right) d\omega^N(\sigma),
	\end{split} \end{equation}
	where $d\omega^N$ is the uniform measure on $S^N$. Note that the logarithm of the LR in \eqref{eq:LR} coincides with the free energy of the SSK model after shifting and rescaling. In the LR test, if the test statistic $\frac{d {\mathbb P}_1}{d {\mathbb P}_0} < 1$ the null hypothesis is accepted, while it is rejected if $\frac{d {\mathbb P}_1}{d {\mathbb P}_0} > 1$. Since the fluctuation of the LR is equal to the fluctuation of the free energy of the SSK model, it is possible to prove the optimal error for the weak detection.
	
	If the rank-$1$ signal $xx^T$ is perturbed by $U=(U_{ij})$, the ratio of the densities is changed to
	\[ \begin{split}
		&\frac{\exp \big( N (M_{ij} - \sqrt{\lambda} U_{ij} - \sqrt{\lambda} x_i x_j)^2 \big)}{\exp \big( N M_{ij}^2 \big)} \\
		&= \exp \left( -2 \sqrt{\lambda} N (M_{ij} - \sqrt{\lambda} U_{ij}) x_i x_j - 2\sqrt{\lambda} N M_{ij} U_{ij} + \lambda N x_i^2 x_j^2 \right).
	\end{split} \]
	Thus, for given $U$ the LR in \eqref{eq:LR} becomes
	\begin{equation} \label{eq:LR_deformed}
		\prod_{i \neq j} \exp \left( -\sqrt{\lambda} N M_{ij} U_{ij} \right) \int_{S^N} \prod_{i \neq j} \exp \left( -N \sqrt{\lambda} (M_{ij} - \sqrt{\lambda} U_{ij} ) x_i x_j + \frac{N}{2} \lambda x_i^2 x_j^2 \right) d\omega^N(\sigma)
	\end{equation}
	for which it is required to consider the free energy of the SSK model with deformed Gaussian interaction. Note that while $U$ is not assumed to be diagonal, we may diagonalize $U$ in the integrand in \eqref{eq:LR_deformed} for the analysis since GOE is orthogonally invariant.

\subsection{Organization of the paper}

The rest of the paper is organized as follows: In Section \ref{sec:model_and_main_theorem}, we precisely define the model and introduce our main results. In Section \ref{sec:Preliminaries}, we list several important results needed in the proof of main results. In Sections \ref{sec:local_law}, \ref{sec:CLT}, and \ref{sec:rigidity}, we prove our results on deformed Wigner matrices - local law for the resolvent entries, CLT for the linear spectral statistics, and the rigidity of the eigenvalues, respectively. In Sections \ref{sec:low_temperature} and \ref{sec:high_temperature}, we prove the main theorems for the low temperature case and the high temperature case, respectively. Some technical details on the results for the steepest descent curve and the proofs of some auxiliary lemmas are collected in Appendices.

\section{Model and main results}\label{sec:model_and_main_theorem}

\subsection{Definition of the model}

Recall that the disorder $J = W + \lambda V$. Here, $W$ is an $N \times N$ real Wigner matrix for which we use the following definition:
\begin{definition} \label{def:Wigner}
An $N \times N$ matrix $W=(W_{ij})_{N\times N}$ is a Wigner matrix if
\begin{itemize}
	\item $\{W_{ij}\vert i\le j\}$ are independent real-valued random variables.
	\item $W_{ij}=W_{ji}$.
	\item $\E[W_{ij}]=0$, $\E[W_{ij}^2]=\frac{1+\delta_{ij}}{N}$.
	\item There exist $\theta>1$ and $\theta'>0$ such that
\begin{align}
{\mathbb P}(\sqrt N\vert W_{ij}\vert >x)\le\theta'\exp(-x^{1/\theta})\quad\forall x\ge0, N\ge1\text{ and } i,j\in\{1,\ldots,N\}
\end{align}
\end{itemize}
\end{definition}
We remark that the subexponential decay condition guarantees the the existence of all (normalized) moments and an overwhelming-probability bound as follows:
\begin{enumerate}
	\item for any $p\in\{1,2,\ldots\}$,
\begin{align}\label{eqn1} 
	\sup\limits_{i,j,N}\E[\vert \sqrt NW_{ij}\vert ^p]<\infty;
\end{align}

\item if $\epsilon'>0$ and $D'>0$,   then for large enough $N$ we have 
\begin{align}\label{eqn2}
{\mathbb P}\Big(\vert W_{ij}\vert \le N^{\epsilon'-\frac{1}{2}},\forall i,j\in\{1,\ldots,N\}\Big)>1-N^{-D'}.
\end{align}
\end{enumerate}

We assume that $V$ is a random diagonal matrix whose entries are i.i.d. with centered Jacobi distribution $\mu$ for which we use the following definitions:
\begin{definition} \label{def:Jacobi}
A probability measure $\mu$ is a Jacobi measure on $[-1,1]$ if its density function is given by
$$\frac{d\mu}{dx}=\frac{d(x)}{Z}(1+x)^a(1-x)^b\mathds{1}_{[-1,1]}(x)$$
where
\begin{itemize}
	\item $a>-1$ and $b>-1$
	\item $d(x)\in C^1([-1,1])$ and $d(x)>0$ on $[-1,1]$.
	\item $Z$ is the normalization constant: $Z=\int_{-1}^1d(x)(1+x)^a(1-x)^bdx$
	\item $\mu$ is centered: $\int_{-1}^1xd\mu(x)=0$
\end{itemize}
\end{definition}

We also assume that $V$ is independent of $W$. For a given constant $\lambda>0$, if we denote by $\lambda\mu$ the law of $\lambda v$ where $v$ is a random variable with law $\mu$, then
the empirical measure of $W+\lambda V$ converges to $\mu_{fc}$ as $N\to\infty$, which is given by
$$\mu_{fc}:=\mu_{sc}\boxplus(\lambda\mu)$$
where $\boxplus$ denotes the additive free convolution and $\mu_{sc}$ denotes the semicircle distribution. It is known that $\mu_{fc}$ has a density function, which we will call $\rho_{fc}$; see Remark 2.5 in \cite{Lee+Schnelli} for the detail. In the following lemma, we collect the results on $\mu_{fc}$,
\begin{lemma}\label{lemma:properties_of_mu_fc}
	Set
	$$\lambda_\pm=\Big(\int_{-1}^1\frac{d\mu(x)}{(1\mp x)^2}\Big)^{1/2},\quad\tau_\pm=\int_{-1}^1\frac{d\mu(x)}{1\mp x}.$$
	There exists $L_-<0<L_+$ such that $\text{supp}(\mu_{fc})=[L_-,L_+]$. If $b>1$ and $\lambda>\lambda_+$, then
	\begin{enumerate}
		\item $L_+=\lambda+\frac{\tau_+}{\lambda}$,
		\item $L_++\int \frac{\rho_{fc}(x)dx}{x-L_+}=\lambda$, and
		\item there exists $C_0\ge1$ such that
		\begin{align}\label{eq12}
			\frac{x^b}{C_0}\le \rho_{fc}(L_+-x)\le C_0x^b\text{ for }x\in[0,L_+]
		\end{align}
	\end{enumerate}
	If $a>1$ and $\lambda>\lambda_-$, the statements above hold for $L_-$, with $L_+$ and $\tau_+$ replaced by $L_-$ and $\tau_-$ respectively. In particular,
		\begin{align}\label{eq13}
			\frac{x^a}{C_0}\le \rho_{fc}(L_-+x)\le C_0x^a\text{ for }x\in[0,\vert L_-\vert ]
		\end{align}
\end{lemma}
For the proof of Lemma \ref{lemma:properties_of_mu_fc}, see Lemma 2.3 and Remark 2.6 \cite{Lee+Schnelli}

\subsection{Main results}

Recall that the free energy of the SSK model at inverse temperature $\beta>0$ is defined by
$$F_N=F_N(\beta)=\frac{1}{N}\log\Bigg[\int_{S_{N-1}}\exp\Big(\beta\langle\sigma,(W+\lambda V)\sigma\rangle\Big)d\omega_N(\sigma)\Bigg]$$
where $S_{N-1}=\{(x_1,\ldots,x_N):\sum_{i=1}^N x_i^2=N\}$ and $\omega_N$ is the (normalized) uniform measure on $S_{N-1}$. We will prove that the constant
\begin{align}\label{critical temperature}
\beta_c=\frac{1}{2}\int \frac{\rho_{fc}(t)}{L_+-t}dt.
\end{align}
is the critical inverse temperature of the SSK model, i.e., we study the fluctuation of $F_N$ in two cases: $0<\beta<\beta_c$ (high temperature regime) and $\beta>\beta_c$ (low temperature regime). We remark that $\beta_c$ is well defined when $b>1$ and $\lambda>\lambda_+$ (see Equation \eqref{eq12}).

Our first main result is the following theorem for the free energy in the low temperature regime:
\begin{theorem}[Main theorem: low temperature]\label{thm:low_temperature}
Suppose  $\beta>\beta_c$, $\lambda>\max(\lambda_-,\lambda_+)$, $b>11$ and $1<a<\frac{b^2-6b-7}{4}$. Then the fluctuation of $F_N$ converges in distribution to a Weibull distribution. More precisely, 
$$\lim\limits_{N\to\infty}{\mathbb P}\Big(N^{\frac{1}{b+1}}\Big[\frac{F_N+\frac{1}{2}\log(2e\beta)+\frac{1}{2}\int\log(L_+-t)d\mu_{fc}(t)-\beta L_+}{\beta-\beta_c}\Big]\le s\Big)=\exp\big(-\frac{C_\mu(-s)^{b+1}}{b+1}\big)\quad\forall s\le0$$
where $C_\mu=\Big(\frac{\lambda}{\lambda^2-\lambda_+^2}\Big)^{b+1}\cdot d(1)\cdot2^a\cdot Z^{-1}$.
\end{theorem}

Our second main result is for the high temperature regime.
\begin{theorem}[Main theorem: high temperature] \label{thm:high_temperature}
	Suppose  $0<\beta<\beta_c$, $\lambda>\max(\lambda_-,\lambda_+)$, $a>1$ and $b>37/3$. Suppose $\hat\gamma$ is the unique point on $(L_+,+\infty)$ such that  $\int\frac{1}{\hat\gamma-t}d\mu_{fc}(t)=2\beta$. Then 
	$$2\sqrt N\Big(F_N+\frac{1}{2}\log(2\beta e)-\beta\hat\gamma+\frac{1}{2}\int\log(\hat\gamma-t)d\mu_{fc}(t)\Big)$$
converges in distribution to a centered Gaussian distribution whose variance is
$$
\frac{1}{4\pi^2}\Big(\oint_{\mathcal C} (1+m_{fc}'(\xi))m_{fc}(\xi)\log(\hat\gamma-\xi)d\xi\Big)^2-\frac{1}{4\pi^2}\int_{-1}^1\Big(\oint_{\mathcal C}\dfrac{(1+m_{fc}'(\xi))\log(\hat\gamma-\xi)}{(\lambda t-\xi-m_{fc}(\xi))}d\xi\Big)^2 d\mu(t)$$
where $m_{fc}(\cdot)$ is the Stieltjes transform of $\mu_{fc}$ and $\mathcal C$ is a counterclockwise path which encloses $[L_-,L_+]$ but does not enclose $\hat\gamma$. Here we take the analytic branch of $\log(\cdot)$ on ${\mathbb C }\backslash(-\infty,0]$ such that $\Im\log(\cdot)\in(-\pi,\pi)$.
\end{theorem}

From Theorems \ref{thm:low_temperature} and \ref{thm:high_temperature}, we immediately obtain the following corollary on the limiting free energy:
\begin{corollary}\label{coro:limit_of_free_energy}
Suppose $\lambda>\max(\lambda_-,\lambda_+)$, $b>37/3$ and $1<a<\frac{b^2-6b-7}{4}$.  As $N\to\infty$ we have
\begin{align*}
F_N\to F(\beta)=\begin{cases}
-\frac{1}{2}\log(2e\beta)-\frac{1}{2}\int\log(L_+-t)d\mu_{fc}(t)+\beta L_+\quad&\text{if }\beta>\beta_c\\
-\frac{1}{2}\log(2e\beta)-\frac{1}{2}\int\log(\hat\gamma-t)d\mu_{fc}(t)+\beta \hat\gamma\quad&\text{if }0<\beta<\beta_c
\end{cases}
\end{align*}
in distribution.
\end{corollary}

From the definitions of $\beta_c$ and $\hat\gamma$, we see that $\lim\limits_{\beta\to\beta_c-}\hat\gamma= L_+$ and that $\lim\limits_{\beta\to\beta_c-}F(\beta)=\lim\limits_{\beta\to\beta_c+}F(\beta)$.

\subsection{Outline of the proof}

In this paper, we study the fluctuation of $F_N$ by following the idea introduced in \cite{Baik+Lee}. In the {\bf low} temperature case (i.e., $\beta>\beta_c$), we will show that the leading term of $F_N$ is a linear function of $\lambda_1$. Since the fluctuation of $\lambda_1$ has size $O(N^{-\frac{1}{1+b}})$ and converges to a Weibull distribution, so does the fluctuation of $F_N$, as in Theorem \ref{thm:low_temperature}. In the {\bf high} temperature case (i.e., $0<\beta<\beta_c$), the leading term of $F_N$ is a linear function of the quantity
\begin{align}\label{eq86} 
	\frac{1}{N}\sum_{i=1}^N f(\lambda_i).
\end{align}
for some $N$-independent deterministic function $f$. Thus, by the central limit theorem (see Theorem \ref{thm:CLT}), the fluctuation of $F_N$ has size $O(N^{-1/2})$ and converges to a Gaussian distribution, as in Theorem \ref{thm:high_temperature}.

For the actual proof, in addition to the known results, we need the local law for resolvent entries, the central limit theorem for linear statistics, and the rigidity of eigenvalues. While we prove these results in the current paper, some of them are not strong enough to directly follow the analysis in \cite{Baik+Lee}. To overcome the difficulty, we introduce several changes in the detail of the proof. Most notably, (1) for the low temperature case, instead of proving a lemma analogous to Lemma 6.4 of \cite{Baik+Lee} that is required to control the integral of an exponential function along the curve of the steepest descent in Lemma \ref{lemma:analogue_of_Lemma_6.3}, we prove a refined result for the curve in Lemma \ref{lemma:steepest_descent_curve}, and (2) for the high temperature case, instead of controlling the difference $\vert \gamma-\hat\gamma\vert $ by applying the rigidity of eigenvalues, we use the local law to control it as in Lemma \ref{lemma:distance_between_gamma_and_hat_gamma}.

In what follows, we list our new results on the deformed Wigner matrices:

\begin{definition}\label{definition:D_delta(M)}
\begin{itemize}
	\item 
For any $M>0$ and $\delta>0$, define
	$$D_\delta(M)=\big\{x+\i y\big\vert \vert x\vert \le M, N^{-\delta}<\vert y\vert \le3\big\}$$
\item For any $z\in{\mathbb C }\backslash{\mathbb R }$, define $$m_N(z)=\frac{1}{N}\sum_{i=1}^N\frac{1}{\lambda_i-z},\quad G(z)=\frac{1}{W+\lambda V-z}$$
where $\lambda_1\ge\cdots\ge\lambda_N$ are eigenvalues of $W+\lambda V$.
\end{itemize}
\end{definition} 

\begin{theorem}[local law for resolvent entries] \label{thm:local_law_for_resolvent_entries} Suppose $M>0$ and $0<\delta\le\frac{1}{4}$. For any $\epsilon'>0$ and $D'>0$, we have for large enough $N$ that
%	$$\P\Big(\max_{i\ne j}|G_{ij}(z)|\le N^{\epsilon'-\frac{1}{2}} |\Im z|^{-3} ,\forall z\in D_\delta(M)\Big)>1-N^{-D'}$$ 
	$${\mathbb P}\Big(\max_{i,j}\vert G_{ij}-\delta_{ij}\cdot\frac{1}{\lambda v_i-z-m_N(z)}\vert \le N^{\epsilon'-\frac{1}{2}} \vert \Im z\vert ^{-3},\forall z\in D_\delta(M)\Big)>1-N^{-D'}$$
\end{theorem}

We remark that the local law for the trace of the resolvent was proved by the first author and Schnelli. See \cite{Lee+Schnelli} and also Section \ref{sec:local_law_for_trace} of the current paper.

\begin{theorem}[CLT for linear statistics]\label{thm:CLT}
	Let $f(x)$ be a function which is analytic on a neighborhood of $[L_-,L_+]$. Suppose $a>1$, $b>37/3$ and $\lambda>\max(\lambda_+,\lambda_-)$. Then 
$$\frac{1}{\sqrt N}\Big(\sum_if(\lambda_i)-N\int f(t)\rho_{fc}(t)dt\Big)$$
	 converges in distribution to a centered Gaussian distribution whose variance is
	$$\frac{1}{4\pi^2}\Big(\oint_{\mathcal C} f(\xi)(1+m_{fc}'(\xi))m_{fc}(\xi)d\xi\Big)^2-\frac{1}{4\pi^2}\int_{-1}^1\Big(\oint_{\mathcal C}\dfrac{f(\xi)(1+m_{fc}'(\xi))}{\lambda t-\xi-m_{fc}(\xi)}d\xi\Big)^2 d\mu(t)$$
	where $\mathcal C$ is a counterclockwise path enclosing $[L_-,L_+]$ such that $f$ is analytic on a neighborhood of the region bounded by $\mathcal C$. 
\end{theorem}
\begin{definition}\label{definition:classic_position}
	Define the deterministic number $\gamma_x=\gamma_x(N)$ and $\hat\gamma_y=\hat\gamma_y(N)$ by
	$$\mu_{fc}([\gamma_x,+\infty])=\frac{x-\frac{1}{2}}{N}\quad\forall x\in[1,N]$$
	$$\mu_{fc}([\hat\gamma_y,+\infty])=\frac{y}{N}\quad\forall y\in(0,N)$$
	with the convention that $\hat\gamma_N=L_-$ and $\hat\gamma_0=L_+$. Here $x$ and $y$ are not necessarily  integers.
\end{definition}
\begin{theorem}[Rigidity of eigenvalues]\label{thm:rigidity_concise_version}
Suppose $a>1$, $b>3$ and $\lambda>\max(\lambda_-,\lambda_+)$. Suppose $\epsilon\in(\frac{1}{b+1},\frac{1}{4})$. There exists an event $E_N(\epsilon)$ such that
\begin{align}\label{eq115}
{\mathbb P}(E_N(\epsilon))\ge1-\kappa_0(\log N)^{1+2b}N^{-\epsilon}
\end{align}
when $N$ is large enough. Moreover, if $E_N(\epsilon)$ holds, then:
\begin{enumerate}
	\item for any $\zeta\in(0,\frac{\frac{1}{4}-\epsilon}{b+1})$ we have
\begin{align}\label{eq103}
\vert \lambda_i-\gamma_i\vert \le N^{-\frac{1}{4}+\epsilon+\zeta b}\quad\text{when  }N\text{ is large enough and }i\in\mathbb Z\cap[\kappa'N^{1-\zeta(b+1)},\frac{N}{2}] 
\end{align}
	\item for any $\zeta'\in(0,\frac{\frac{1}{4}-\epsilon}{a+1})$ we have
\begin{align}\label{eq104}
\vert \lambda_i-\gamma_i\vert \le N^{-\frac{1}{4}+\epsilon+\zeta' a}\quad\text{when  }N\text{ is large enough and }i\in\mathbb Z\cap[\frac{N}{2},N-\kappa'N^{1-\zeta'(a+1)}].
\end{align}
\end{enumerate}
Here $\kappa_0>0$ and $\kappa'>0$ are constants independent of $\zeta$ and $\zeta'$.
\end{theorem}

\subsection{Remarks}

As discussed in Introduction, we expect the existence of the dichotomy between the fluctuation given by the LSS in the high temperature regime and the fluctuation dominated by the largest eigenvalue in the low temperature regime, regardless of the choice of various parameters in the deformed Wigner matrix. The main technical issue is the non-optimality of the local law; if the local law can be improved, the rigidity result will also be improved and it will be possible to relax the condition on $a$ and $b$. It is even expected that the fluctuation of $F_N$ would converge to a Gaussian distribution when $\lambda<\lambda_+$, since the fluctuation of $\lambda_1$ converges to a Gaussian distribution in this case. However, we do not attempt to prove the claim in the current paper.

%he decay rates $a$ and $b$. If the local law can be improved, then the rigidity of eigenvalues can also be improved, therefore one will be able to weaken the conditions on $a$ and $b$. However we do not pursue the optimal conditions on $a$ and $b$ in this paper.

\section{Preliminaries}\label{sec:Preliminaries}

\begin{definition}
Suppose $\omega$ is a measure on ${\mathbb R }$. Define its Stieltjes transform by
\begin{align}\label{StieltjesTransform}
\int \frac{d\omega(t)}{t-z},\quad\forall z\in {\mathbb C }\backslash\text{supp}(\omega).
\end{align}
\end{definition}

\subsection{Fluctuation of the largest eigenvalue}
Recall that $\lambda_1\ge\lambda_2\ge\cdots\ge\lambda_N$ are eigenvalues of $W+\lambda V$. The following theorem can be found in \cite{Lee+Schnelli}.
\begin{theorem}\label{thm:fluctuation_of_lambda_1}
If $b>1$ and $\lambda>\lambda_+$, then
$$\lim_{N\to\infty}{\mathbb P}\Big(N^{\frac{1}{1+b}}(L_+-\lambda_1)\le s\Big)=1-\exp\Big(-\frac{C_\mu s^{1+b}}{1+b}\Big),\quad\forall s\ge0$$
where $C_\mu=\Big(\frac{\lambda}{\lambda^2-\lambda_+^2}\Big)^{b+1}\cdot d(1)\cdot2^a\cdot Z^{-1}$ as defined in Theorem \ref{thm:low_temperature}.

If $a>1$ and $\lambda>\lambda_-$, then
$$\lim_{N\to\infty}{\mathbb P}\Big(N^{\frac{1}{1+a}}(\lambda_N-L_-)\le s\Big)=1-\exp\Big(-\frac{C_\mu' s^{1+a}}{1+a}\Big),\quad\forall s\ge0$$
where $C_\mu'=\Big(\frac{\lambda}{\lambda^2-\lambda_-^2}\Big)^{a+1}\cdot d(-1)\cdot2^b\cdot Z^{-1}$.
\end{theorem}
\begin{remark}
For the second conclusion of Theorem \ref{thm:fluctuation_of_lambda_1}, see the sentence above section 2.4.1 in \cite{Lee+Schnelli}. It can also be proved by replacing $W+\lambda V$ by $-W+\lambda (-V)$.
\end{remark}

The next lemma is a direct corollary of (3.22) in \cite{Lee+Schnelli}.
\begin{lemma}\label{lemma:all_eigenvalues_in_2+lambda+r}
For any constant $r>0$ we have that
$$\lim_{N\to\infty}{\mathbb P}\Big(\max\limits_{1\le k\le N}\vert \lambda_k\vert \le 2+\lambda+r\Big)=1.$$
Therefore, 
\begin{align}\label{eqn22}
	[L_-,L_+]\subset[-2-\lambda,2+\lambda].
\end{align}
\end{lemma}

\subsection{Local law for the trace of the resolvent}\label{sec:local_law_for_trace}
In this subsection we introduce the local law for the trace of the resolvent obtained in \cite{Lee+Schnelli}.
\begin{definition}\label{definition:Stieltjes_transforms}
Suppose $\mu_N$ is the empirical measure of $W+\lambda V$: $\mu_N=\frac{1}{N}\sum_{i=1}^N\delta_{\lambda_i}$.
	Let
\begin{itemize}
	\item $m_{fc}(z)$  be the Stieltjes transform of $\mu_{fc}$: $m_{fc}(z)=\int \frac{\rho_{fc}(t)}{t-z}dt$ (as mentioned in Theorem \ref{thm:high_temperature});
	\item $\hat m_{fc}(z)$  be the Stieltjes transform of $(\frac{1}{N}\sum_{i=1}^N\delta_{\lambda v_i})\boxplus\mu_{sc}$.
\end{itemize}
\end{definition}

\begin{definition}\label{definition_g_and_g_i}
	For $z\in{\mathbb C }\backslash{\mathbb R }$, let
	$$g_i(z)=\frac{1}{\lambda v_i-z-m_{fc}(z)},\quad\hat g_i(z)=\frac{1}{\lambda v_i-z-\hat m_{fc}(z)}.$$
\end{definition}
\begin{lemma}\label{lemma:self_consistent_equation}
For $z\in{\mathbb C }\backslash{\mathbb R }$,
$$m_N(z)=\frac{1}{N}\tr G(z),\quad\hat m_{fc}(z)=\frac{1}{N}\sum_{i=1}^N\hat g_i(z),\quad m_{fc}(z)=\E[g_i(z)]=\int\frac{1}{\lambda t-z-m_{fc}(z)}d\mu(t)$$
\begin{align}\label{eqn21}
(1+\hat m_{fc}'(z))\Big(1-\frac{1}{N}\sum \hat g_i^2(z)\Big)=1
\end{align}
\begin{align}\label{eqn48}
	(1+m_{fc}'(z))\Big(1-\int\frac{d\mu(t)}{(\lambda t-z-m_{fc}(z))^2}\Big)=1.
\end{align}
\begin{align}\label{eq102}
\vert g_i(z)\vert \le\frac{1}{\vert \Im z\vert }\quad\text{and}\quad \vert \hat g_i(z)\vert \le\frac{1}{\vert \Im z\vert }
\end{align}
\end{lemma}
\begin{proof}
The first conclusion is trivial. The second and third  conclusions are direct corollaries of (2.3) of \cite{Lee+Schnelli2}. The fourth conclusion can be proved by taking derivatives on both sides of the second conclusion:
\begin{align}\label{eq94}
\hat m_{fc}'(z)=\frac{1}{N}\sum_{i=1}^N\Big(\frac{1}{\lambda v_i-z-\hat m_{fc}(z)}\Big)'=\frac{1}{N}\sum\frac{1+\hat m_{fc}'(z)}{(\lambda v_i-z-\hat m_{fc}(z))^2}.
\end{align}
The fifth conclusion can be proved by similarly taking derivatives on both sides of the third conclusion. The last conclusion is because both $\Im m_{fc}(z)$ and $\Im \hat m_{fc}(z)$ have the same sign as $\Im z$.
\end{proof}

\begin{definition}\label{definition:D_epsilon_andD'_epsilon}
Suppose  $\epsilon\in (0,\frac{11b-9}{2b+2})$. Let 
 $\tilde v_i$ be the $i$-th largest one of $\{v_1,\ldots,v_N\}$. We define the regions $\mathcal D_\epsilon$, $\mathcal D_\epsilon'$ and the events $\tilde\Omega(\epsilon)$, $\Omega_*(\epsilon)$ and $\Omega_0(\epsilon,c_1,c_2)$ by the following.
\begin{itemize}
	\item 
$\mathcal D_\epsilon=\{x+\i y\vert -3-\lambda\le x\le3+\lambda, N^{-\frac{1}{2}-\epsilon}\le y\le N^{-\frac{1}{1+b}+\epsilon}\}$
\item $\mathcal D_\epsilon'=\{z\in\mathcal D_\epsilon\vert \vert \lambda \tilde v_i-z-m_{fc}(z)\vert >\frac{1}{2}N^{-\frac{1}{1+b}-\epsilon},\forall i\in[20,N]\}$
	\item $\tilde\Omega(\epsilon)=\{\vert m_N(z)-\hat m_{fc}(z)\vert \le N^{2\epsilon-\frac{1}{2}}\text{ for all }z\in\mathcal D_\epsilon'\}$ 
\item $\Omega_*(\epsilon)=\{\Im m_N(z)\le N^{2\epsilon-\frac{1}{2}},\forall z\in\mathcal D_\epsilon'\}.$ 
\item $\Omega_0(\epsilon,c_1,c_2)$  is the event on which the following conditions are satisfied for any $k\in\{1,\ldots,19\}$. 
\begin{itemize} 
	\item If $j\in\{1,\ldots,N\}\backslash\{k\}$ then $N^{-\epsilon-\frac{1}{1+b}}<\vert \tilde v_j-\tilde v_k\vert <(\log N) N^{-\frac{1}{1+b}}$.
	Moreover
	$$N^{-\epsilon-\frac{1}{1+b}}<\vert 1-\tilde v_1\vert <(\log N) N^{-\frac{1}{1+b}}.$$
	\item If $z\in \mathcal D_\epsilon$ and 
	$\vert \Re(z+m_{fc}(z)-\lambda \tilde v_k)\vert =\min\limits_{1\le i\le N}\vert \Re(z+m_{fc}(z)-\lambda \tilde v_i)\vert $
	then
	$$\frac{1}{N}\sum_{i\in\{1,\ldots,N\}\backslash\{k\}}\frac{1}{\vert \lambda \tilde v_i-z-m_{fc}(z)\vert ^2}<c_1.$$
	\item If $z\in\mathcal D_\epsilon$ then
	$\Big\vert \frac{1}{N}\sum\limits_{i=1}^N\frac{1}{\lambda\tilde v_i-z-m_{fc}(z)}-\int\frac{d\mu(t)}{\lambda t-z-m_{fc}(z)}\Big\vert \le c_2N^{\frac{3\epsilon}{2}-\frac{1}{2}}$.
\end{itemize}
Here $c_1\in(0,1)$ and $c_2>0$ are constants. 
\end{itemize}
\end{definition}

\begin{remark}
	\begin{itemize}
		\item 
		Notice that $\mathcal D_\epsilon'$ is random but is independent of $W$. 
		\item 
		We defined $\Omega_0(\epsilon,c_1,c_2)$ in the same way as Definition 3.5 in \cite{Lee+Schnelli}. The condition $\epsilon\in (0,\frac{11b-9}{2b+2})$ comes from (3.20) of \cite{Lee+Schnelli}.  Definition 3.5 in \cite{Lee+Schnelli} involves a constant $n_0$ and we let $n_0=20$ in the current paper.
		\item \cite{Lee+Schnelli} requires  the entries of the diagonal matrix to be ordered along the diagonal, so in order to used results in \cite{Lee+Schnelli}, we  use $\tilde v_i$ instead of $v_i$ in the definitions of $\mathcal D_\epsilon'$ and $\Omega_0(\epsilon,c_1,c_2)$. 
	\end{itemize} 
\end{remark}

\begin{proposition}\label{proposition:extreme_eigenvalue}Suppose $b>1$, $\lambda>\lambda_+$ and $\epsilon\in (0,\frac{11b-9}{2b+2})$.
There exist  constants $c_1\in(0,1)$, $c_2>0$, $\nu_0>0$, $\nu_1>0$ and $N_0>0$ such that:
\begin{enumerate}
	\item ${\mathbb P}(\Omega_0(\epsilon,c_1,c_2))\ge1-\nu_0(\log N)^{1+2b}N^{-\epsilon}$ for all $N$;
	\item ${\mathbb P}(\Omega_0(\epsilon,c_1,c_2)\backslash\tilde \Omega(\epsilon))\le \exp(-\nu_1(\log N)^{10\log\log N})$ if $N$ is large enough;
	\item $\Omega_0(\epsilon,c_1,c_2)$ is measurable with respect to the sigma algebra generated by the entries of $V$.
\end{enumerate}
\end{proposition}
\begin{proof}
The first two conclusions of Proposition \ref{proposition:extreme_eigenvalue} are proved in \cite{Lee+Schnelli}. See  (3.30) and Proposition 5.1 there. The last conclusion is from the definition of $\Omega_0(\epsilon,c_1,c_2)$. 
\end{proof}

\begin{definition}\label{definition:Omega_V}
	Let $\Omega_V(\epsilon)$ be the  $\Omega_0(\epsilon,c_1,c_2)$ with $c_1=c_1(\epsilon)$ and $c_2=c_2(\epsilon)$ properly chosen such that the conclusions of Proposition \ref{proposition:extreme_eigenvalue} hold. 
\end{definition}
The next lemmas are Lemma 5.5 and Lemma 3.7 in \cite{Lee+Schnelli}.  
\begin{lemma}\label{lemma:Omega_*}
Suppose $b>1$, $\lambda>\lambda_+$ and $\epsilon\in (0,\frac{11b-9}{2b+2})$. There exists a constant $\nu_2>0$   such that if $N$ is large enough then
$${\mathbb P}(\Omega_V(\epsilon)\backslash\Omega_*(\epsilon))\le\exp(-\nu_2(\log N)^{10\log\log N}).$$
\end{lemma}
\begin{lemma}\label{lemma:hatm_fc_close_tom_fc}
 Suppose $b>1$, $\lambda>\lambda_+$ and $\epsilon\in (0,\frac{11b-9}{2b+2})$. If $\Omega_V(\epsilon)$ holds and $z\in\mathcal D_\epsilon'$, then 
$$\vert \hat m_{fc}(z)-m_{fc}(z)\vert \le N^{2\epsilon-\frac{1}{2}}.$$
\end{lemma}

\subsection{Integral representation of the partition function of the SSK model}
The following lemma comes from   Lemma 1.3 and (5.25) of \cite{Baik+Lee}.
\begin{lemma}\label{lemma:integral_formula}
Suppose $M$ is an $N\times N$ real symmetric matrix with eigenvalues $\lambda_1(M)\ge\cdots\ge\lambda_N(M)$. Suppose $\beta>0$. Then
$$\int_{S_{N-1}}e^{\beta\langle\sigma,M\sigma\rangle}d\omega_N(\sigma)=C_N\int_{a_0-\i \infty}^{a_0+\i\infty}e^{\frac{N}{2}R_M(z)}dz$$ 
where
\begin{itemize}
	\item $a_0$ is an arbitrary constant satisfying $a_0>\lambda_1(M)$;
	\item the integration contour is the vertical line from $a_0-\i\infty$ to $a_0+\i\infty$;
	\item $R_M(z)=2\beta z-\frac{1}{N}\sum_i\log(z-\lambda_i(M))$ where we take the analytic branch of the $\log$ function such that $\Im\log(z-\lambda_i(M))\in(-\pi,\pi)$ for all $z$ on the integration contour;
	\item $C_N=\frac{\Gamma(N/2)}{2\pi\i(N\beta)^{\frac{N}{2}-1}}$ where $\Gamma(z)$ denotes the Gamma function. Moreover,
	$$C_N=\frac{\sqrt N \beta}{\i\sqrt\pi(2\beta e)^{N/2}}(1+O(N^{-1})).$$
\end{itemize}
\end{lemma}

\subsection{Helffer-Sj\"ostrand formula}
The next  lemma can be found in Section 11.2 of \cite{Erdos+Yau}.
\begin{lemma}\label{lemma:Helffer-Sjostrand formula}
Suppose $\chi:{\mathbb R }\to[0,1]$ is $C^\infty$ such that $\chi(x)=1$ when $x\in[-1,1]$ and $\chi(x)=0$ when $x\not\in[-2,2]$. If $f\in C_c^2({\mathbb R })$, then
$$f(t)=\frac{1}{2\pi}\int_{{\mathbb R }^2}\frac{\i yf''(x)\chi(y)+\i\chi'(y)(f(x)+\i yf'(x))}{t-x-\i y}dxdy,\quad\forall t\in{\mathbb R }.$$
\end{lemma}
 
\subsection{Cumulant expansion}
The next lemma is Lemma 3.2 of \cite{Lee+Schnelli3}.
\begin{lemma}\label{lemma:cumulant_expansion}
Suppose $l\in\{1,2,\ldots\}$ and $F\in C^{l+1}({\mathbb R },{\mathbb C })$. Let $Y$ be a real-valued centered random variable with finite moments up to the order $l+2$. Let $\mathcal G$ be a sigma algebra independent of $Y$. Then
\begin{align*}
\E[YF(Y)\vert \mathcal G]=\sum_{r=1}^l\frac{\kappa^{(r+1)}(Y)}{r!}\E[F^{(r)}(Y)\vert \mathcal G]+\mathcal E(Y)
\end{align*}
where $\kappa^{(r+1)}(Y)$ denotes the $(r+1)$-cumulant of $Y$. The error term $\mathcal E(Y)$ satisfies:
\begin{align}\label{eqn56}
\vert \mathcal E(Y)\vert \le C_l\E[\vert Y\vert ^{l+2}]\sup_{\vert t\vert \le Q}\vert F^{(l+1)}(t)\vert +C_l\E[\vert Y\vert ^{l+2}\mathds1_{\vert Y\vert >Q}]\sup_{t\in{\mathbb R }}\vert F^{(l+1)}(t)\vert 
\end{align}
where $Q>0$ is an arbitrary cutoff and $C_l$ satisfies $C_l\le \frac{(p_0\cdot l)^l}{l!}$ for some absolute constant $p_0>0$.
\end{lemma}

\subsection{Marcinkiewicz-Zygmund inequality}
The next lemma is copied from Lemma D.1, Lemma D.2 and Lemma D.3 of \cite{Benaych-Georges+Knowles}. It is a version of the Marcinkiewicz-Zygmund inequality.
\begin{lemma}\label{lemma:Marcinkiewicz-Zygmund}
	Let $X_1,\ldots,X_N,Y_1,\ldots,Y_N$ be independent centered random variables such that for each $p\in\{1,2,\ldots\}$ there exists a constant $\mu_p>0$ satisfying
	$$\E[\vert X_i\vert ^p]^{1/p}\le\mu_p,\quad\E[\vert Y_i\vert ^p]^{1/p}\le\mu_p\quad(1\le i\le N).$$
	Then for   deterministic families $(a_{ij})$ and $(b_i)$ we have
	$$\E[\vert \sum_i b_iX_i\vert ^p]^{1/p}\le C_p\cdot\mu_p\cdot(\sum_i\vert b_i\vert ^2)^{1/2},\quad\forall p\ge1$$
	$$\E[\vert \sum_{ij}a_{ij}X_iY_j\vert ^p]^{1/p}\le C_p\cdot\mu_p^2\cdot(\sum_{ij}\vert a_{ij}\vert ^2)^{1/2},\quad\forall p\ge1$$
	$$\E[\vert \sum_{i\ne j}a_{ij}X_iX_j\vert ^p]^{1/p}\le C_p\cdot\mu_p^2\cdot(\sum_{i\ne j}\vert a_{ij}\vert ^2)^{1/2},\quad\forall p\ge1$$
	where $C_p$ is a constant depending only on $p$.
\end{lemma}

\subsection{Some  results for symmetric matrices}

Suppose $\mathtt M$  is an $N\times N$ real symmetric matrix. Suppose $T$ is a subset of $\{1,\ldots,N\}$.

\begin{definition}
	\begin{itemize}
		\item We use $\mathtt M^{(T)}$ to denote the $(N-\vert T\vert )\times(N-\vert T\vert )$ matrix:
		$$(\mathtt M_{ij})_{i,j\in\{1,\ldots,N\}\backslash T}$$
		\item 
		For $z\in{\mathbb C }\backslash{\mathbb R }$ let
		$$\mathtt R(z)=(\mathtt M-z)^{-1},\quad \mathtt R^{(T)}(z)=(\mathtt M^{(T)}-z)^{-1}$$ 
		\item 
		We also set
		$$\sum_i^{(T)}=\sum_{i:i\not\in T},\quad\sum_{i,j}^{(T)}=\sum_{i:i\not\in T}\sum_{j:j\not\in T}$$
	\end{itemize}
\end{definition}
\begin{remark}
	\begin{enumerate}
		\item When $T=\{i\}$, we use $(i)$ instead of $(\{i\})$ in the above definitions. Similarly, we write $(ij)$ instead of $(\{i,j\})$. We use $(Ti)$ to denote $(T\cup\{i\})$.
		\item In $\mathtt M^{(T)}$ and $\mathtt R^{(T)}$ we use the original values of matrix indices. For example, the indices for the rows and columns of $\mathtt M^{(2)}$ are $1,3,4,\ldots,N$.
		\item It is easy to see that $\mathtt R(z)$ is a symmetric matrix.
	\end{enumerate}
\end{remark}

\begin{lemma}\label{lemma:well_known_facts_for_resolvent}
	Suppose $\Im z_1\ne0$ and $\Im z_2\ne0$. Then
	\begin{itemize}
		\item $\frac{d}{dz}\mathtt R(z_1)=\mathtt R^2(z_1)$
		\item  $\sum_{j}\vert \mathtt R_{ij}(z_1)\vert ^2=\dfrac{\Im \mathtt R_{ii}(z_1)}{\Im z_1}, \quad\forall i$
		\item 
		$\big\vert \big(\mathtt R^{k_1}(z_1)(\mathtt R')^{k_2}(z_2)\big)_{ij}\big\vert \le\frac{1}{\vert \Im z_1\vert ^{k_1}\vert \Im z_2\vert ^{2k_2}},\quad$ for any $i,j\in\{1,\ldots,N\}$,  $k_1,k_2\in\{0,1,2,\ldots\}$. 
	\end{itemize} 
\end{lemma}
\begin{proof}
	The first can be proved by directly taking the derivative. The second conclusion is the Ward identity, see (3.6) of \cite{Benaych-Georges+Knowles}. For the last conclusion, suppose 
	$$\mathtt M=O\text{diag}(\lambda_1(\mathtt M),\ldots,\lambda_N(\mathtt M))O^T$$
	where $\lambda_1(\mathtt M),\ldots,\lambda_N(\mathtt M)$ are eigenvalues of $\mathtt M$ and $O$ is an orthogonal matrix. Then
\begin{multline*}
\mathtt R(z_1)=O\text{diag}((\lambda_1(\mathtt M)-z_1)^{-1},\ldots,(\lambda_N(\mathtt M)-z_1)^{-1})O^T\quad\text{and}\\\mathtt R'(z_2)=O\text{diag}((\lambda_1(\mathtt M)-z_2)^{-2},\ldots,(\lambda_N(\mathtt M)-z_2)^{-2})O^T\quad\text{(by the first conclusion of this lemma)}.
\end{multline*}
So
\begin{multline*}
\big\vert \big(\mathtt R^{k_1}(z_1)(\mathtt R')^{k_2}(z_2)\big)_{ij}\big\vert =\Big\vert \sum_{r=1}^N\frac{1}{(\lambda_r(\mathtt M)-z_1)^{k_1}(\lambda_r(\mathtt M)-z_2)^{2k_2}}O_{ir}O_{jr}\Big\vert \\
\le\frac{1}{\vert \Im z_1\vert ^{k_1}\vert \Im z_2\vert ^{2k_2}}\sum_{r=1}^N\vert O_{ir}O_{jr}\vert \le\frac{1}{\vert \Im z_1\vert ^{k_1}\vert \Im z_2\vert ^{2k_2}}\sum_{r=1}^N\frac{O_{ir}^2+O_{jr}^2}{2}= \frac{1}{\vert \Im z_1\vert ^{k_1}\vert \Im z_2\vert ^{2k_2}}.
\end{multline*}
\end{proof}

\begin{lemma}[Resolvent identities]\label{lemma:resolvent+identities}
	\begin{enumerate}
		\item If $i,j, k\not\in T$ and $i,j\ne k$, then 
		\begin{align}\label{eqn3}
			\mathtt R^{(T)}_{ij}=\mathtt R_{ij}^{(Tk)}+\dfrac{\mathtt R^{(T)}_{ik}\mathtt R^{(T)}_{jk}}{\mathtt R^{(T)}_{kk}}.
		\end{align}
		\item if $i\ne j$ then 
		\begin{align}\label{eqn4}
			\mathtt R_{ij}=-\mathtt R_{ii}\mathtt R_{jj}^{(i)}\Big(\mathtt M_{ij}-\sum_{k,l}^{(ij)}\mathtt M_{ik}\mathtt R_{kl}^{(ij)}\mathtt M_{lj}\Big)
		\end{align}
		\item \begin{align}\label{eqn15}
			\mathtt R_{ii}^{-1}=\mathtt M_{ii}-z-\sum_{k,l}^{(i)}\mathtt M_{ik}\mathtt R_{kl}^{(i)}M_{li}
		\end{align}
		\item \begin{align}\label{eqn27}
			\frac{\partial\mathtt R_{kl}}{\partial\mathtt M_{ij}}=\frac{-1}{1+\delta_{ij}}(\mathtt R_{ki}\mathtt R_{lj}+\mathtt R_{kj}\mathtt R_{li})
		\end{align}
	\end{enumerate}
\end{lemma}
\begin{proof} 
 The first conclusion can be found in (3.4) of \cite{Benaych-Georges+Knowles}. 	The second conclusion is (5.9) of \cite{Benaych-Georges+Knowles}. The third conclusion is (5.1) of \cite{Benaych-Georges+Knowles}. The fourth conclusion can be proved by definition.
\end{proof}

\begin{lemma}\label{lemma:Lemma5.2_from_lecture_notes}
	If $\mathtt M$ is an $N\times N$ real symmetric random matrix such that $\{\mathtt M_{ij}\vert i\le j\}$ are independent and $\E[\mathtt M_{ij}^2]=\frac{1}{N}$ for $i\ne j$, then
	$$\frac{1}{\mathtt R_{ii}(z)}=-z-\frac{1}{N}\tr\mathtt  R(z)+\mathtt M_{ii}+\frac{1}{N}\sum_k\frac{(\mathtt R_{ki}(z))^2}{\mathtt R_{ii}(z)}-\sum_{k,l}^{(i)}\mathtt M_{ik}\mathtt R_{kl}^{(i)}\mathtt M_{li}+\frac{1}{N}\sum_k^{(i)}\mathtt R_{kk}^{(i)}.$$
\end{lemma}
\begin{proof}
	This lemma is from Lemma 5.2 of \cite{Benaych-Georges+Knowles}
	and the fact that $\sum_{k,l}^{(i)}\E\Big[\mathtt M_{ik}\mathtt R_{kl}^{(i)}\mathtt M_{li}\Big\vert \mathtt M^{(i)}\Big]=\frac{1}{N}\sum_k^{(i)}\mathtt R_{kk}^{(i)}$ (since $\mathtt M_{ik}\mathtt M_{li}$ is independent of the sigma algebra generated by the entries of  $\mathtt M^{(i)}$).
\end{proof}

\section{Local law for resolvent entries: proof of Theorem \ref{thm:local_law_for_resolvent_entries}}\label{sec:local_law}

In this section we follow the idea introduced in \cite{Benaych-Georges+Knowles} to prove Theorem \ref{thm:local_law_for_resolvent_entries}. Recall that $G(z)$ is defined in Definition \ref{definition:D_delta(M)}.

\begin{lemma}\label{lemma:stochastic_domination}
Suppose $i\ne j$. For any $\epsilon'>0$, $D'>0$, there exists $N_0=N_0(\epsilon',D')>0$ such that if $N>N_0$ and $z\in{\mathbb C }\backslash{\mathbb R }$ then
$${\mathbb P}\Big(\Big\vert \sum_{k,l}^{(ij)}W_{ik}G_{kl}^{(ij)}W_{lj}\Big\vert \le N^{\epsilon'}\sqrt{\frac{1}{N^2}\sum_{k,l}^{(ij)}\vert G_{kl}^{(ij)}\vert ^2}\Big)>1-N^{-D'}$$
$${\mathbb P}\Big(\Big\vert \sum_{k,l}^{(i)}W_{ik}G_{kl}^{(i)}W_{li}\Big\vert \le N^{\epsilon'}\sqrt{\frac{1}{N^2}\sum_{k\ne l}^{(i)}\vert G_{kl}^{(i)}\vert ^2}+N^{\epsilon'-1}\sum_k^{(i)}\vert G_{kk}^{(i)}\vert \Big)>1-N^{-D'}$$
$${\mathbb P}\Big(\Big\vert \sum_{k,l}^{(ij)}W_{jk} G_{kl}^{(ij)}W_{lj} \Big\vert \le N^{\epsilon'}\sqrt{\frac{1}{N^2}\sum_{k\ne l}^{(ij)}\vert G_{kl}^{(ij)}\vert ^2}+N^{\epsilon'-1}\sum_k^{(ij)}\vert G_{kk}^{(ij)}\vert \Big)>1-N^{-D'}$$
$${\mathbb P}\Big(\Big\vert \sum_{k}^{(i)}(W_{ik}^2-\frac{1}{N})G_{kk}^{(i)} \Big\vert \le N^{\epsilon'}\sqrt{\frac{1}{N^2}\sum_{k}^{(i)}\vert G_{kk}^{(i)}\vert ^2}\Big)>1-N^{-D'}$$
where each of $G_{kl}^{(ij)}$ and $G_{kl}^{(i)}$ takes value at $z$.
\end{lemma}
\begin{proof}
\begin{enumerate}
	\item For the first conclusion, suppose $\mathcal G_1$ is the sigma algebra generated by entries of $(W+\lambda V)^{(ij)}$. Then $W_{ik}$ and $W_{lj}$ are independent of $\mathcal G_1$. Let
$$B_{kl}=\frac{G_{kl}^{(ij)}}{\sqrt{\sum_{k,l}^{(ij)}\vert G_{kl}^{(ij)}\vert ^2}}.$$

For any natural number $p$ and any sample point $\omega$ in the probability space, we have
\begin{multline*}
\E\Big[\Big\vert \sum_{k,l}^{(ij)}W_{ik}B_{kl}W_{lj}\Big\vert ^{2p}\Big\vert \mathcal G_1\Big](\omega)
=\E\Big[\Big(\sum_{k,l}^{(ij)}W_{ik}B_{kl}W_{lj}\Big)^p\Big(\sum_{k,l}^{(ij)}W_{ik}\bar B_{kl}W_{lj}\Big)^p\Big\vert \mathcal G_1\Big](\omega)\\
=\E\Big[\Big(\sum_{k,l}^{(ij)}W_{ik}B_{kl}(\omega)W_{lj}\Big)^p\Big(\sum_{k,l}^{(ij)}W_{ik}\bar B_{kl}(\omega)W_{lj}\Big)^p\Big]=\E\Big[\Big\vert \sum_{k,l}^{(ij)}W_{ik}B_{kl}(\omega)W_{lj}\Big\vert ^{2p}\Big]\\
\le C\Big(\sum_{k,l}^{(ij)}\Big\vert \frac{B_{kl}(\omega)}{N}\Big\vert ^2\Big)^p=CN^{-2p}
\end{multline*}
where $C>0$ depends only on $p$. We used  \eqref{eqn1} and the second conclusion of Lemma \ref{lemma:Marcinkiewicz-Zygmund} in the inequality.
So,
\begin{multline*}
{\mathbb P}\Big(\Big\vert \sum_{k,l}^{(ij)}W_{ik}G_{kl}^{(ij)}W_{lj}\Big\vert >N^{\epsilon'}\sqrt{\frac{1}{N^2}\sum_{k,l}^{(ij)}\vert G_{kl}^{(ij)}\vert ^2}\Big)={\mathbb P}\Big(\Big\vert N\sum_{k,l}^{(ij)}W_{ik}B_{kl}W_{lj}\Big\vert >N^{\epsilon'}\Big)
\\\le N^{-2p\epsilon'}\E\Big[\Big\vert N\sum_{k,l}^{(ij)}W_{ik}B_{kl}W_{lj}\Big\vert ^{2p}\Big]\le CN^{-2p\epsilon'}.
\end{multline*}
Choosing $p$ large enough such that $2p\epsilon'>D'$ we complete the proof of the first conclusion.

\item For the second conclusion, we use the same argument as above except that the $\mathcal G_1$ is replaced by the sigma algebra generated by entries of $(W+\lambda V)^{(i)}$, the summation is replaced by $\sum_{k\ne l}^{(i)}$ and $B_{kl}$ is replaced by 
$$\frac{G_{kl}^{(i)}}{\sqrt{\sum_{k\ne l}^{(i)}\vert G_{kl}^{(i)}\vert ^2}}.$$
Then using the third conclusion of Lemma \ref{lemma:Marcinkiewicz-Zygmund}  we have for $N>N_0=N_0(\epsilon',D')$:
\begin{align}\label{eqn7}
	{\mathbb P}\Big(\Big\vert \sum_{k\ne l}^{(i)}W_{ik}G_{kl}^{(i)}W_{li}\Big\vert >N^{\epsilon'}\sqrt{\frac{1}{N^2}\sum_{k\ne l}^{(i)}\vert G_{kl}^{(i)}\vert ^2}\Big)\le CN^{-2p\epsilon'}\le N^{-D'}.
\end{align}
This together with \eqref{eqn2} and the fact that
$\sum_{k,l}^{(i)}W_{ik}G_{kl}^{(i)}W_{li}=\sum_{k\ne l}^{(i)}W_{ik}G_{kl}^{(i)}W_{li}+\sum_k^{(i)}W_{ik}^2G_{kk}^{(i)}$
complete the proof of the second conclusion.

\item The third conclusion can be proved in the same way as the second conclusion.
\item For the last conclusion, let $X_k=NW_{ik}^2-1$ and $Q_k=G_{kk}^{(i)}/\sqrt{\sum_{k}^{(i)}\vert G_{kk}^{(i)}\vert ^2}$. Then, similarly as above, with $\mathcal G_2$ be the sigma algebra generated by entries of $(W+\lambda V)^{(i)}$, for any natural number $p$ and any sample point $\omega$, 
$$\E[\vert \sum_{k}^{(i)}X_kQ_k\vert ^{2p}\vert \mathcal G_2](\omega)=\E[\vert \sum_{k}^{(i)}X_kQ_k(\omega)\vert ^{2p}]\le C$$
where $C$ depends only on $p$. We used  \eqref{eqn1} and the first conclusion of Lemma \ref{lemma:Marcinkiewicz-Zygmund} in the last inequality. So for any $\epsilon'>0$, $D'>0$, if $N>N_0=N_0(\epsilon',D')$, then
\begin{align*}
{\mathbb P}\Big(\vert \sum_{k}^{(i)}X_kG_{kk}^{(i)}\vert >N^{\epsilon'}\sqrt{\sum_{k}^{(i)}\vert G_{kk}^{(i)}\vert ^2}\Big)\le N^{-2p\epsilon'}\E[\vert \sum_{k}^{(i)}X_kQ_k\vert ^{2p}]\le CN^{-2p\epsilon'}
\end{align*}
Choosing $p$ large enough such that $2p\epsilon'>D'$, we complete the proof.
\end{enumerate}
\end{proof}

\begin{corollary}\label{coro:bound_for_1/G_ii}
For any $M>0$, $\epsilon'>0$, $D'>0$, if   $\vert \Re z\vert \le M$ and $0<\vert \Im z\vert \le3$ then	
$${\mathbb P}\Big(\Big\vert \frac{1}{G_{ii}}\Big\vert \le \frac{3N^{\epsilon'}}{\vert \Im z\vert },\forall i\in[1,N]\Big)>1-N^{-D'}$$
$${\mathbb P}\Big(\Big\vert \frac{1}{G_{jj}^{(i)}}\Big\vert \le \frac{3N^{\epsilon'}}{\vert \Im z\vert },\forall i\ne j\Big)>1-N^{-D'}$$
%$$\P\Big(|G_{kj}^{(i)}|\le \Lambda_*+\Lambda_*^2\cdot\frac{3N^{\epsilon'}}{|\Im z|},\forall i,j,k \text{ distinct}\Big)>1-N^{-D'}$$ 
for large enough $N$.
\end{corollary}
\begin{proof}
By Lemma \ref{lemma:resolvent+identities},
\begin{align*}
\frac{1}{G_{ii}}=\lambda v_i+W_{ii}-z-\sum_{k,l}^{(i)}W_{ik}G_{kl}^{(i)}W_{li},\quad
	\frac{1}{G_{jj}^{(i)}}=\lambda v_j+W_{jj}-z-\sum_{k,l}^{(ij)}W_{jk}G_{kl}^{(ij)}W_{lj}.
\end{align*}
This together with \eqref{eqn2}, Lemma \ref{lemma:stochastic_domination} and the facts that
$$\vert \lambda v_i\vert \le\lambda,\quad \vert G_{kl}^{(i)}\vert \le\frac{1}{\vert \Im z\vert },\quad \vert G_{kl}^{(ij)}\vert \le\frac{1}{\vert \Im z\vert },\quad  \vert z\vert \le M+3\quad  (\text{provided }\vert \Re z\vert \le M, \vert \Im z\vert \in(0,3])$$
yield the conclusions. 
%For the last conclusion, by Lemma \ref{lemma:resolvent+identities},
%\begin{align*}
%\Big|G_{kj}^{(i)}\Big|=\Big|G_{kj}-\frac{G_{ki}G_{ji}}{G_{ii}}\Big|\le\Lambda_*+\Lambda_*^2\Big|\frac{1}{G_{ii}}\Big|\quad\text{provided that $i,j,k$ are distinct}.
%\end{align*}
%This together with the first conclusion proves the last conclusion.
\end{proof}

\begin{proof}[Proof of Theorem \ref{thm:local_law_for_resolvent_entries}]
	
Suppose   $i\ne j$ and $z\in{\mathbb C }\backslash{\mathbb R }$. By Lemma \ref{lemma:well_known_facts_for_resolvent} and \eqref{eqn3}
\begin{multline}\label{eqn8}
\vert \sum_k^{(ij)} G_{kk}^{(ij)}\vert =\vert \sum_k^{(ij)}\Big( G_{kk}-\frac{G_{ki}G_{ki}}{G_{ii}}-\frac{G_{kj}^{(i)}G_{kj}^{(i)}}{G_{jj}^{(i)}}\Big)\vert \\
\le \frac{N}{\vert \Im z\vert }+\frac{\vert (G^2)_{ii}-(G_{ii})^2-(G_{ij})^2\vert }{\vert G_{ii}\vert }+\frac{\vert (G^{(i)})^2_{jj}-(G_{jj}^{(i)})^2\vert }{\vert G_{jj}^{(i)}\vert }\le \frac{N}{\vert \Im z\vert }+\frac{3}{\vert \Im z\vert ^2\vert G_{ii}\vert }+\frac{2}{\vert \Im z\vert ^2\vert G_{jj}^{(i)}\vert }
\end{multline}
and
\begin{align*}
\Big(\frac{1}{N^2}\sum_{k,l}^{(ij)}\vert G_{kl}^{(ij)}\vert ^2\Big)^{1/2}=\Big(\sum_{k}^{(ij)}\frac{\Im G_{kk}^{(ij)}}{N^2\Im z}\Big)^{1/2}
\le \Big(\frac{1}{N\vert \Im z\vert ^2}+\frac{3}{N^2\vert \Im z\vert ^3\vert G_{ii}\vert }+\frac{2}{N^2\vert \Im z\vert ^3\vert G_{jj}^{(i)}\vert }\Big)^{1/2}
\end{align*}
By Corollary \ref{coro:bound_for_1/G_ii}, for any $\epsilon'>0$, $D'>0$, if $N$ is large enough and  $i\ne j$, then
\begin{align}\label{eqn5}
{\mathbb P}\Big(\Big(\frac{1}{N^2}\sum_{k,l}^{(ij)}\vert G_{kl}^{(ij)}\vert ^2\Big)^{1/2}\le \sqrt{\frac{1}{N\vert \Im z\vert ^2}+\frac{15N^{\epsilon'}}{N^2\vert \Im z\vert ^4}}\Big)>1-N^{-D'}\quad\text{provided }\vert \Re z\vert \le M, \vert \Im z\vert \in(0,3]
\end{align}

Now using \eqref{eqn4}, \eqref{eqn2}, the first conclusion of Lemma \ref{lemma:stochastic_domination} and \eqref{eqn5}, we have that for any $\epsilon'>0$ and $D'>0$, if $N$ is large enough then
$${\mathbb P}\Big(\max_{i\ne j}\vert G_{ij}\vert \le \frac{N^{\epsilon'}}{\sqrt N\vert \Im z\vert ^3}+\frac{N^{\epsilon'}}{ N\vert \Im z\vert ^4}\Big)>1-N^{-D'}\quad\text{provided }\vert \Re z\vert \le M \text{ and } \vert \Im z \vert \in(0,3].$$
If $\delta\le\frac{1}{4}$, then $\frac{N^{\epsilon'}}{\sqrt N\vert \Im z\vert ^3}\ge\frac{N^{\epsilon'}}{ N\vert \Im z\vert ^4}$ for all $z\in D_\delta(M)$. This together with a classic ``lattice" argument  proved that for any $\epsilon'>0$ and $D'>0$, if $N$ is large enough then 
\begin{align}\label{eq88}
{\mathbb P}\Big(\max_{i\ne j}\vert G_{ij}\vert \le \frac{N^{\epsilon'}}{\sqrt N\vert \Im z\vert ^3}, \forall z\in D_\delta(M)\Big)>1-N^{-D'}.
\end{align}

 By Lemma \ref{lemma:Lemma5.2_from_lecture_notes},
\begin{multline}\label{eqn11}
\vert G_{ii}-\frac{1}{\lambda v_i-z-m_N(z)}\vert =\frac{\vert G_{ii}\vert }{\vert \lambda v_i-z-m_N(z)\vert }\vert \frac{1}{G_{ii}}-(\lambda v_i-z-m_N(z))\vert \\
=\frac{\vert G_{ii}\vert }{\vert \lambda v_i-z-m_N(z)\vert }\Big\vert W_{ii}+\frac{1}{N}\sum_k\frac{ (G_{ki}(z))^2 }{ G_{ii}(z)}-\sum_{k,l}^{(i)} W_{ik} G_{kl}^{(i)} W_{li}+\frac{1}{N}\sum_k^{(i)}G_{kk}^{(i)}\Big\vert \\
\le \vert \Im z\vert ^{-2}\Big(\vert W_{ii}\vert +\frac{\vert (G^2)_{ii}\vert }{N\vert G_{ii}\vert }+\vert \sum_{k\ne l}^{(i)}W_{ik}G_{kl}^{(i)}W_{li}\vert +\vert \sum_k^{(i)}(W_{ik}^2-\frac{1}{N})G_{kk}^{(i)}\vert \Big)
\end{multline}
where we used Lemma    \ref{lemma:well_known_facts_for_resolvent} and the fact that $\Im m_N$ has the same sign as $\Im z$ in the last step.

By Lemma \ref{lemma:well_known_facts_for_resolvent} and \eqref{eqn3}, we have 
\begin{align}\label{eqn9}
	\sqrt{\frac{1}{N^2}\sum_{k}^{(i)}\vert G_{kk}^{(i)}\vert ^2}\le\frac{1}{\vert \Im z\vert \sqrt N}
\end{align}
and
\begin{align}\label{eqn10}
	\sqrt{\frac{1}{N^2}\sum_{k\ne l}^{(i)}\vert G_{kl}^{(i)}\vert ^2}\le \sqrt{\frac{1}{N^2}\sum_{k, l}^{(i)}\vert G_{kl}^{(i)}\vert ^2}=\Big(\sum_k^{(i)}\frac{\Im G_{kk}^{(i)}}{N^2\Im z}\Big)^{1/2}\le\frac{1}{\vert \Im z\vert \sqrt N}
\end{align}

If $\epsilon'>0$, $D'>0$,  $\vert \Re z\vert \le M$ and $\vert \Im z\vert \in(0,3]$, then by \eqref{eqn11}, \eqref{eqn2},  Corollary \ref{coro:bound_for_1/G_ii}, \eqref{eqn7} and the last conclusion of Lemma \ref{lemma:stochastic_domination} we have for large enough $N$:
$${\mathbb P}\Big(\vert G_{ii}-\frac{1}{\lambda v_i-z-m_N(z)}\vert \le \frac{N^{\epsilon'}}{\vert \Im z\vert ^2}\Big[\frac{1}{\sqrt N}+\frac{3N^{\epsilon'}}{N\vert \Im z\vert ^3}+\sqrt{\frac{1}{N^2}\sum_{k\ne l}^{(i)}\vert G_{kl}^{(i)}\vert ^2}+\sqrt{\frac{1}{N^2}\sum_{k}^{(i)}\vert G_{kk}^{(i)}\vert ^2}\Big]\Big)>1-N^{-D'}$$
thus by \eqref{eqn9} and \eqref{eqn10}, 
$${\mathbb P}\Big(\vert G_{ii}-\frac{1}{\lambda v_i-z-m_N(z)}\vert \le \frac{3N^{2\epsilon'}}{N\vert \Im z\vert ^5}+\frac{5N^{\epsilon'}}{\sqrt N\vert \Im z\vert ^3}\Big)>1-N^{-D'}$$
If $\delta\le\frac{1}{4}$, then $\frac{N^{\epsilon'}}{\sqrt N\vert \Im z\vert ^3}\ge\frac{N^{\epsilon'}}{ N\vert \Im z\vert ^5}$ for all $z\in D_\delta(M)$. This together with a classic ``lattice" argument yields that for any $\epsilon'>0$ and $D'>0$, if $N$ is large enough then 
$${\mathbb P}\Big( \vert G_{ii}-\frac{1}{\lambda v_i-z-m_N(z)}\vert \le \frac{N^{\epsilon'}}{\sqrt N\vert \Im z\vert ^3}, \text{ for any } z\in D_\delta(M)\text{ and }1\le i\le N\Big)>1-N^{-D'}.$$
This together with \eqref{eq88} completes the proof.
\end{proof}

\section{Central limit theorem for linear statistics: proof of Theorem \ref{thm:CLT}}\label{sec:CLT}  

Suppose $f$ is a fixed function satisfying the condition in Theorem \ref{thm:CLT}. In this section we  prove Theorem \ref{thm:CLT}, i.e., the fact that   the linear statistics
\begin{align}\label{eqn42}
\frac{1}{\sqrt N}\Big(\sum_if(\lambda_i)-N\int f(t)\rho_{fc}(t)dt\Big)
\end{align}
converges in distribution to a Gaussian variable. We use the method introduced in \cite{Ji+Lee}, but we prove Lemma \ref{lemma:convergence_of_P_1} in a different way. The method we prove Lemma \ref{lemma:convergence_of_P_1} is similar as that in \cite{Li+Schnelli+Xu}.

Throughout this section, we assume that the conditions of Theorem \ref{thm:CLT} are  satisfied.
\begin{definition}\label{definition:constants}
\begin{itemize}
	\item Suppose $d\in(0,\frac{1}{2})$ is a constant which is small enough such that $f$ is analytic on a neighborhood of the rectangular region whose vertices are $L_++d\pm d\i$ and $L_--d\pm d\i$. 
	\item 
	Use $\Gamma$ to denote the boundary of the above rectangular region with counterclockwise orientation.
	\item Let $$\varpi\in\Big(\frac{2}{3(b+1)},\min\Big(\frac{1}{14},\frac{1}{10}-\frac{2}{5(b+1)},\frac{1}{8}-\frac{1}{b+1},\frac{1}{9}-\frac{2}{3(b+1)}\Big)\Big),$$
	$$\varsigma\in[\frac{1}{b+1},\frac{1}{8}-\varpi)\quad\quad\text{and}\quad\varsigma'\in\Big(0,\min\Big(\frac{1}{3}-4\varpi,\frac{1}{2}-7\varpi,\frac{1}{2}-5\varpi-2\varsigma\Big)\Big).$$
	\item Let $\Gamma_+=\{z\in\Gamma\vert \vert \Im z\vert \ge N^{-\varpi}\}$. The orientation of $\Gamma_+$ is induced from $\Gamma$.
\end{itemize}
\end{definition}
\begin{remark}
Since $b>37/3$, it is easy to check that the constants $\varpi$, $\varsigma$, $\varsigma'$   exist.
\end{remark}

\begin{definition}
\begin{itemize}
	\item Let $\sigma(V)$ be the sigma algebra generated by $V$:
$$\sigma(V)=\sigma(v_1,\ldots,v_N).$$
   \item Use $\E_N[\cdot]$ to denote the conditional expectation $\E[\cdot\vert \sigma(V)]$.
\end{itemize}
\end{definition}

\begin{lemma}\label{lemma:convergence_of_P_1}As $N\to\infty$,
$$\frac{1}{\sqrt N}\int_{\Gamma_+}f(\xi)\Big[\tr G(\xi)-\E_N[\tr G(\xi)]\Big]d\xi\to 0\quad\text{in distribution}.$$
\end{lemma}
We prove Lemma \ref{lemma:convergence_of_P_1} in Section \ref{sec:proof_of_convergence_of_P_1}.

\begin{lemma}\label{lemma:convergence_of_P_2}As $N\to\infty$,
	$$\frac{1}{\sqrt N}\int_{\Gamma_+}f(\xi)\Big[N \hat m_{fc}(\xi)-\E_N[\tr G(\xi)]\Big]d\xi\to 0\quad\text{in distribution}.$$
\end{lemma}

We prove Lemma \ref{lemma:convergence_of_P_2} in Section \ref{sec:proof_of_convergence_of_P_2}.

\subsection{Some auxiliary lemmas}
Recall that $g_i$ and $\hat g_i$ are defined in Definition \ref{definition_g_and_g_i}.
\begin{lemma}\label{lemma:lambdav_i-z-m_fc_large_on_Gamma}
	
	There is a constant $C_d>0$ depending on $d$ such that $$\min_i\vert \lambda v_i-z-m_{fc}(z)\vert \ge C_d\quad\forall z\in\Gamma$$
	Moreover, $g_i(z)$ is analytic on ${\mathbb C }\backslash[L_-,L_+]$ for all $i\in\{1,\ldots,N\}$.
\end{lemma}
\begin{proof}See Appendix \ref{appendix}.
\end{proof}
\begin{definition}
Define $M_d=\max(\vert L_++d+2\vert ,\vert L_--d-2\vert )$ and
\begin{align}\label{good_event}
	B_N=\tilde\Omega(\varsigma)\cap\{\max_{i, j}\vert G_{ij}(z)-\frac{\delta_{ij}}{\lambda v_i-z-m_N(z)}\vert \le N^{\varsigma'-\frac{1}{2}} \vert \Im z\vert ^{-3},\forall z\in D_\varpi(M_d)\}
\end{align}
where $\tilde\Omega(\varsigma)$ is defined in Definition  \ref{definition:Omega_V} and $D_\varpi(M_d)$ is defined by Definition \ref{definition:D_delta(M)}. The parameters $\varpi$, $\varsigma$ and $\varsigma'$ are defined in Definition \ref{definition:constants}.
\end{definition}
\begin{lemma}\label{lemma:properties_of_good_event}
	\begin{enumerate}
		\item 
		there exists $N_0>0$ such that if $N>N_0$ then $\Gamma_+\subset\mathcal D_\varsigma'$
		\item for any $D'>0$ we have that if $N$ is large enough then
		\begin{align}\label{eqn43}
			{\mathbb P}(\Omega_V(\varsigma)\backslash B_N)<N^{-D'}.
		\end{align}
		\item if $N$ is large enough and $B_N\cap\Omega_V(\varsigma)$ holds, then the following holds for each $\xi\in\Gamma_+$:
		$$	\vert G_{ii}(\xi)-\hat g_i(\xi)\vert \le  N^{\varsigma'-\frac{1}{2}}\cdot\vert \Im \xi \vert ^{-3}+N^{2\varsigma-\frac{1}{2}}\cdot\vert \Im \xi \vert ^{-2},\quad \vert G_{ii}(\xi)\vert 
		\ge W'\vert \Im \xi\vert $$
		$$	
		\vert \hat m_{fc}(\xi)-\frac{1}{N}\tr G^{(i)}(\xi)\vert \le N^{2\varsigma-\frac{1}{2}} +\frac{3}{W'N\vert \Im\xi\vert ^3}$$
		$$\vert G_{ii}(\xi)-g_i(\xi)\vert \le  N^{\varsigma'-\frac{1}{2}}\cdot\vert \Im \xi \vert ^{-3}+2N^{2\varsigma-\frac{1}{2}}\cdot\vert \Im \xi \vert ^{-2}$$ 
	\end{enumerate}
where $W'$ is a constant in $(0,1)$.
	
\end{lemma}
\begin{proof} See Appendix \ref{appendix}.
\end{proof}
\subsection{Proof of Theorem \ref{thm:CLT}}

\begin{proof}[Proof of Theorem \ref{thm:CLT}]
Let $\varpi'$ be a positive constant in $[\frac{1}{b+1}-\varpi,\varpi/2)$ and 
$$R_N=\Big\{\lambda_i\in[L_--\frac{d}{10},L_++\frac{d}{10}],\forall i\Big\}\cap\Omega_V(\varpi')\cap\tilde\Omega(\varpi').$$
See Section \ref{sec:local_law_for_trace} for the notations. By Definition \ref{definition:D_epsilon_andD'_epsilon}, Definition  \ref{definition:Omega_V}, \eqref{eqn22}, Lemma \ref{lemma:lambdav_i-z-m_fc_large_on_Gamma} and the fact that $d<1/2$, we know that if $N$ is large enough then
\begin{align}\label{eqn44}
\Gamma\cap\{z\vert N^{-\frac{1}{2}-\varpi'}\le \Im z\le N^{-\varpi}\}\subset\mathcal D_{\varpi'}'.
\end{align}
By Theorem \ref{thm:fluctuation_of_lambda_1} and Proposition \ref{proposition:extreme_eigenvalue}, 
\begin{align}\label{eqn45}
{\mathbb P}(R_N)\to1\quad\text{as }N\to\infty.
\end{align}
On $R_N$ we have
$$f(\lambda_i)=\frac{1}{2\pi\i}\oint_\Gamma\frac{f(\xi)}{\xi-\lambda_i}d\xi\quad\text{and}\quad f(t)=\frac{1}{2\pi\i}\oint_\Gamma\frac{f(\xi)}{\xi-t}d\xi\quad\forall t\in[L_-,L_+]$$
and then
\begin{multline}\label{eq91}
\frac{1}{\sqrt N}[\sum f(\lambda_i)-N\int f(t)\rho_{fc}(t)dt]=\frac{1}{\sqrt N}\frac{1}{2\pi\i}\Big[\oint_\Gamma f(\xi)\sum\frac{1}{\xi-\lambda_i}d\xi-N\int\oint_\Gamma\frac{f(\xi)}{\xi-t}d\xi\cdot\rho_{fc}(t)dt\Big]\\
=\frac{1}{\sqrt N2\pi\i}\oint_\Gamma f(\xi)(Nm_{fc}(\xi)-\tr G(\xi))d\xi=\\
\frac{1}{\sqrt N2\pi\i}\int_{\Gamma\backslash\Gamma_+} f(\xi)(Nm_{fc}(\xi)-\tr G(\xi))d\xi
+\frac{1}{\sqrt N2\pi\i}\int_{\Gamma_+} f(\xi)(\E_N[\tr G(\xi)]-\tr G(\xi))d\xi\\
+\frac{1}{\sqrt N2\pi\i}\int_{\Gamma_+} f(\xi)(N\hat m_{fc}(\xi)-\E_N[\tr G(\xi)])d\xi+\frac{1}{\sqrt N2\pi\i}\int_{\Gamma_+} f(\xi)(Nm_{fc}(\xi)-N\hat m_{fc}(\xi))d\xi\\
:=P_0+P_1+P_2+P_3.
\end{multline}
\begin{lemma}\label{lemma:hat_m_fc-m_fc}
	If $\xi\in{\mathbb C }\backslash{\mathbb R }$ then
	\begin{multline*}
		\sqrt N(\hat m_{fc}-m_{fc})=\\
		(1+m_{fc}')\Big(\frac{1}{\sqrt N}\sum\Big( g_i(\xi)-\E[g_i(\xi)]\Big)+\frac{(\hat m_{fc}-m_{fc})^2}{\sqrt N}\sum \hat g_ig_i^2
		+\frac{1 }{\sqrt N}(\hat m_{fc}-m_{fc})\sum (g_i^2-\E [g_i^2])\Big).
	\end{multline*}
\end{lemma}
\begin{proof}See Appendix \ref{appendix}.
\end{proof}
Using Lemma \ref{lemma:hat_m_fc-m_fc} for $P_3$  (i.e., the last term  in \eqref{eq91}), we have that if $R_N$ holds then
\begin{align}\label{eq92}
\frac{1}{\sqrt N}[\sum f(\lambda_i)-N\int f(t)\rho_{fc}(t)dt]=P_0+P_1+P_2+P_{31}+P_{32}+P_{33}
\end{align}
where $P_0$, $P_1$, $P_2$ are defined in \eqref{eq91} and
$$P_{31}:=\frac{-1}{2\pi\i\sqrt N}\sum_i\int_{\Gamma_+}f(\xi)(1+m_{fc}')(g_i-\E g_i)d\xi,$$
$$P_{32}:=\frac{-1}{2\pi\i\sqrt N}\int_{\Gamma_+}f(\xi)(1+m_{fc}')(\hat m_{fc}-m_{fc})\sum_i(g_i^2-\E [g_i^2])d\xi,$$
$$P_{33}:=\frac{-1}{2\pi\i\sqrt N}\int_{\Gamma_+}f(\xi)(1+m_{fc}')(\hat m_{fc}-m_{fc})^2\sum_i \hat g_ig_i^2d\xi.$$

{\bf $\bullet$ Asymptotic behavior of $P_0$.}
When $R_N$ holds, we have:
\begin{multline}\label{eqn46}
\vert P_0\vert \le \frac{1}{\sqrt N2\pi}\int_{\Gamma\cap\{\vert \Im z\vert \le N^{-(1+\varpi')/2}\}} \vert f(\xi)\vert \vert Nm_{fc}(\xi)-\tr G(\xi)\vert d\xi\\
+\frac{1}{\sqrt N2\pi}\int_{\Gamma\cap\{ N^{-(1+\varpi')/2}\le\vert \Im z\vert \le N^{-\varpi}\}} \vert f(\xi)\vert \vert Nm_{fc}(\xi)-\tr G(\xi)\vert d\xi\\
\le \frac{1}{\sqrt N 2\pi}4N^{-\frac{1}{2}-\frac{\varpi'}{2}}\sup_{z\in\Gamma}\vert f(z)\vert \cdot4N/d\quad\text{(since $\vert m_{fc}(\xi)\vert \le 2/d$ and $\vert G_{ii}(z)\vert \le2/d$)}\\
+\frac{1}{\sqrt N 2\pi}4N^{-\varpi}\sup_{z\in\Gamma}\vert f(z)\vert \cdot N\cdot 2N^{2\varpi'-\frac{1}{2}}\quad\text{(by \eqref{eqn44}, Lemma \ref{lemma:hatm_fc_close_tom_fc} and Definition \ref{definition:D_epsilon_andD'_epsilon})}\\
=(\frac{8}{\pi d}+\frac{4}{\pi})\sup_{z\in\Gamma}\vert f(z)\vert (N^{-\varpi'/2}+N^{2\varpi'-\varpi})
=o(1)\quad\text{(since $\varpi'<\varpi/2$)}
\end{multline}
 
{\bf $\bullet$ Asymptotic behavior of $P_{33}$.}
Let
$$\varpi''\in(\frac{1}{1+b},\frac{11b-9}{2b+2})\quad\text{such that}\quad4\varpi''+\varpi<\frac{1}{2}.$$
By the condition on $\varpi$, such $\varpi''$ exists. When $N$ is large enough, we have that $\Gamma_+\subset\mathcal D_{\varpi''}'$ and that if $\Omega_V(\varpi'')$ holds then by Lemma \ref{lemma:hatm_fc_close_tom_fc},
\begin{align}\label{eq90}
\vert \hat m_{fc}(\xi)-m_{fc}(\xi)\vert \le N^{-\frac{1}{2}+2\varpi''}\quad\forall \xi\in\Gamma_+
\end{align} $$\vert P_{33}\mathds1_{\Omega_V(\varpi'')}\vert \le \frac{1}{2\pi\sqrt N}\vert \Gamma\vert \sup_{z\in\Gamma}\vert f(z)\vert (1+\frac{1}{d^2})N^{-1+4\varpi''}N\cdot N^\varpi \frac{1}{C_d^2}$$
where $C_d$ is defined in Lemma \ref{lemma:lambdav_i-z-m_fc_large_on_Gamma}. Here we used the fact that $\vert m_{fc}'(\xi)\vert $ is bounded by $\frac{1}{d^2}$ for $\xi\in\Gamma$.
The last inequality together with the facts that $4\varpi''+\varpi<\frac{1}{2}$ and that ${\mathbb P}(\Omega_V(\varpi''))\to1$ (by Proposition \ref{proposition:extreme_eigenvalue}) yield:
\begin{align}\label{eqn52} 
	P_{33}\to0\quad\text{in distribution}.
\end{align}

{\bf $\bullet$ Asymptotic behavior of $P_{32}$.} Let
$$W_N=\{\frac{1}{N^{\frac{1}{2}+\varpi''}}\vert \sum_{i=1}^N(g_i^2(\xi)-\E[g_i^2(\xi)])\vert \le1,\quad\forall\xi\in\Gamma\}.$$
\begin{lemma}\label{coro:convergence_as_process}
	Suppose $a_1>0$, $a_2>0$ are constants. Then
	$${\mathbb P}\Big(\Big\vert \frac{1}{N^{\frac{1}{2}+a_1}}\sum_{i=1}^N(g_i^2(\xi)-\E[g_i^2(\xi)])\Big\vert \le a_2,\quad\forall\xi\in\Gamma\Big)\to1\quad \text{as } N\to\infty$$
	$${\mathbb P}\Big(\Big\vert \frac{1}{N^{\frac{1}{2}+a_1}}\sum_{i=1}^N(g_i(\xi)-\E[g_i(\xi)])\Big\vert \le a_2,\quad\forall\xi\in\Gamma\Big)\to1\quad \text{as } N\to\infty$$
\end{lemma}
\begin{proof}See Appendix \ref{appendix}.
\end{proof}
By Lemma \ref{coro:convergence_as_process} and Proposition \ref{proposition:extreme_eigenvalue},
${\mathbb P}(\Omega_V(\varpi'')\cap W_N)\to1$ as $N\to\infty$.
This together with the fact that
\begin{multline*}
\vert \mathds1_{\Omega_V(\varpi'')\cap W_N}P_{32}\vert \le \frac{1}{2\pi}\int_{\Gamma_+}\sup_{z\in\Gamma}\vert f(z)\vert (1+\frac{1}{d^2})\vert \hat m_{fc}-m_{fc}\vert N^{\varpi''}\frac{1}{N^{\frac{1}{2}+\varpi''}}\big\vert \sum_i(g_i^2-\E [g_i^2])\big\vert d\xi\\
\le \frac{1}{2\pi}\vert \Gamma\vert \sup_{z\in\Gamma}\vert f(z)\vert (1+\frac{1}{d^2})N^{3\varpi''-\frac{1}{2}}=o(1)\quad\text{(by \eqref{eq90} and the condition on $\varpi''$)}
\end{multline*}
yield: 
\begin{align}\label{eqn53}
	P_{32}\to0\quad\text{in distribution}.
\end{align}

{\bf $\bullet$ Asymptotic behavior of $P_{31}$.} Let
$$U_N=\{\frac{1}{N^{\frac{1}{2}+\frac{\varpi}{2}}}\vert \sum_{i=1}^N(g_i(\xi)-\E[g_i(\xi)])\vert \le1,\quad\forall\xi\in\Gamma\}.$$
By Lemma \ref{coro:convergence_as_process}, ${\mathbb P}( U_N)\to1$ as $N\to\infty$. This together with the fact that
\begin{multline*}
\mathds1_{ U_N}\Big\vert \frac{-1}{2\pi\i\sqrt N}\sum_i\int_{\Gamma\backslash\Gamma_+}f(\xi)(1+m_{fc}')(g_i-\E g_i)d\xi\Big\vert \\
\le \mathds1_{U_N}\cdot\frac{1}{2\pi}\int_{\Gamma\backslash\Gamma_+}\sup_{z\in\Gamma}\vert f(z)\vert (1+\frac{1}{d^2})N^{\varpi/2}\cdot\Big(\frac{1}{N^{\frac{1}{2}+\frac{\varpi}{2}}}\vert \sum_{i=1}^N(g_i(\xi)-\E[g_i(\xi)])\vert \Big)d\xi\\
\le \frac{1}{2\pi} \sup_{z\in\Gamma}\vert f(z)\vert (1+\frac{1}{d^2})\cdot4N^{-\varpi/2}=o(1)\quad\text{(since $\vert \Gamma\backslash\Gamma_+\vert =4N^{-\varpi}$)}
\end{multline*}
yield: 
$$\frac{-1}{2\pi\i\sqrt N}\sum_i\int_{\Gamma\backslash\Gamma_+}f(\xi)(1+m_{fc}')(g_i-\E g_i)d\xi\to0\quad\text{in distribution}. $$  
So $P_{31}$ has the same limit in distribution as
\begin{align}\label{eqn54} 
	\frac{-1}{2\pi\i\sqrt N}\sum_i\oint_\Gamma f(\xi)(1+m_{fc}')(g_i-\E g_i)d\xi.
\end{align}
By central limit theorem, \eqref{eqn54} converges in distribution to a centered Gaussian distribution whose variance is
\begin{multline}\label{eq93} 
\text{Var}\Big(	\frac{-1}{2\pi\i}\oint_\Gamma f(\xi)(1+m_{fc}'(\xi))g_i(\xi)d\xi\Big)=\text{Var}\Big(	\frac{-1}{2\pi\i}\oint_\Gamma f(\xi)\frac{1+m_{fc}'(\xi)}{\lambda v_1-\xi-m_{fc}(\xi)}d\xi\Big)\\
=\E\Big[\Big(\frac{-1}{2\pi\i}\oint_\Gamma f(\xi)\frac{1+m_{fc}'(\xi)}{\lambda v_1-\xi-m_{fc}(\xi)}d\xi\Big)^2\Big]-\Big(\E\Big[\frac{-1}{2\pi\i}\oint_\Gamma f(\xi)\frac{1+m_{fc}'(\xi)}{\lambda v_1-\xi-m_{fc}(\xi)}d\xi\Big]\Big)^2\\
=\frac{1}{4\pi^2}\Big(\oint_\Gamma f(\xi)(1+m_{fc}'(\xi))m_{fc}(\xi)d\xi\Big)^2-\frac{1}{4\pi^2}\oint_\Gamma\oint_\Gamma\int_{-1}^1\dfrac{f(\xi_1)f(\xi_2)(1+m_{fc}'(\xi_1))(1+m_{fc}'(\xi_2))d\mu(t)}{(\lambda t-\xi_1-m_{fc}(\xi_1))(\lambda t-\xi_2-m_{fc}(\xi_2))}d\xi_1d\xi_2\\
=\frac{1}{4\pi^2}\Big(\oint_\Gamma f(\xi)(1+m_{fc}'(\xi))m_{fc}(\xi)d\xi\Big)^2-\frac{1}{4\pi^2}\int_{-1}^1\Big(\oint_\Gamma\dfrac{f(\xi)(1+m_{fc}'(\xi))}{\lambda t-\xi-m_{fc}(\xi)}d\xi\Big)^2 d\mu(t)
\end{multline}

{\bf $\bullet$ Conclusion.} The asymptotic behaviors of $P_0$, $P_{31}$, $P_{32}$, $P_{33}$  together with  \eqref{eqn45}, \eqref{eq92}, Lemma \ref{lemma:convergence_of_P_1} and Lemma  \ref{lemma:convergence_of_P_2} complete the proof of Theorem \ref{thm:CLT}. We remark that the variance of the imaginary part of    $\frac{-1}{2\pi\i}\oint_\Gamma f(\xi)(1+m_{fc}'(\xi))g_i(\xi)d\xi$, i.e., the left hand side of \eqref{eq93}, must be 0. This is because the above argument show that \eqref{eqn54} has the same limit in distribution as the real-valued random variable \eqref{eqn42}. 
\end{proof}

\subsection{Proof of Lemma \ref{lemma:convergence_of_P_1}}\label{sec:proof_of_convergence_of_P_1}

\begin{proof}[Proof of Lemma \ref{lemma:convergence_of_P_1}]
  According to Proposition \ref{proposition:extreme_eigenvalue}, it suffices to prove that
\begin{align*} 
X_N:=\frac{1}{\sqrt N}\int_{\Gamma_+}f(\xi)\Big[\tr G(\xi)-\E_N\tr G(\xi)\Big]\cdot\mathds1_{\Omega_V(\varsigma)}d\xi
\end{align*} 
converges in distribution to zero. Fix  $t\in {\mathbb R }$. We only need to show that
\begin{align}\label{eqn41}
\E[\exp(\i t X_N)]\to 1\quad\text{as }N\to\infty.
\end{align}

Notice that
\begin{multline}\label{eqn40}
\frac{d}{dt}\E[\exp(\i t X_N)]=\frac{d}{dt}\E[\E_N[\exp(\i t X_N)]]\\
=\frac{\i}{\sqrt N}\E\Big[\E_N\Big[\exp(\i t X_N)\int_{\Gamma_+}f(\xi)\Big[\tr G(\xi)-\E_N\tr G(\xi)\Big]\cdot\mathds1_{\Omega_V(\varsigma)}d\xi\Big]\Big]\\
=\frac{\i}{\sqrt N}\int_{\Gamma_+}f(\xi)\cdot\E\Big[\mathds1_{\Omega_V(\varsigma)}\E_N\Big[\exp(\i t X_N)[\tr G(\xi)-\E_N\tr G(\xi)]\Big] \Big]d\xi
\end{multline}

\begin{lemma}\label{lemma:derivative_of_exp(itX_N)}
Suppose the conditions of Lemma \ref{lemma:convergence_of_P_1} are satisfied. Then we have
\begin{align}\label{eqn31}
	\frac{\partial e^{\i tX_N}}{\partial W_{ij}}
	=e^{\i tX_N}\frac{-2\i t}{\sqrt N(1+\delta_{ij})}\int_{\Gamma_+}f(\xi)G_{ij}'(\xi)d\xi\cdot\mathds1_{\Omega_V(\varsigma)}
\end{align}
Moreover,	there exist constants $r_0>0$ and $N_0>0$  such that if $N>N_0$  then
\begin{align}\label{eq98}
\vert \frac{\partial e^{\i tX_N}}{\partial W_{ij}}\vert \le r_0\cdot N^{\varpi-\frac{1}{2}},\quad\vert \frac{\partial^2 (e^{\i tX_N})}{\partial W_{ij}^2}\vert \le r_0\cdot N^{2\varpi-\frac{1}{2}},\quad\vert \frac{\partial^3 (e^{\i tX_N})}{\partial W_{ij}^3}\vert \le r_0\cdot N^{3\varpi-\frac{1}{2}}.
\end{align}
\end{lemma}
\begin{proof}
By Lemma \ref{lemma:well_known_facts_for_resolvent},
\begin{multline*}
	\frac{\partial e^{\i tX_N}}{\partial W_{ij}}=e^{\i tX_N}\frac{\i t}{\sqrt N}\int_{\Gamma_+}f(\xi)\frac{-2}{1+\delta_{ij}}\sum_kG_{ik}(\xi)G_{kj}(\xi)d\xi\cdot\mathds1_{\Omega_V(\varsigma)}\\
	=e^{\i tX_N}\frac{-2\i t}{\sqrt N(1+\delta_{ij})}\int_{\Gamma_+}f(\xi)G_{ij}'(\xi)d\xi\cdot\mathds1_{\Omega_V(\varsigma)}
\end{multline*}
Noticing $\vert G'_{ij}(\xi)\vert \le\vert \Im \xi\vert ^{-2}$ we complete the proof of the first inequality in\eqref{eq98}. The other two inequalities in \eqref{eq98} can be proved similarly by directly taking more derivatives of $e^{\i tX_N}$ with respect to $W_{ij}$.
\end{proof}

For any $\xi\in{\mathbb C }\backslash{\mathbb R }$, by the definition 
$$G(\xi)=\frac{1}{\lambda V+W-\xi}$$
we have $(\xi-\lambda v_i)G_{ii}=-1+(WG)_{ii}=-1+\sum_jW_{ij}G_{ij}$. Then by    \eqref{eqn27},
\begin{multline*}
(\xi-\lambda v_i)\E_N[e^{\i tX_N}(G_{ii}(\xi)-\E_N G_{ii}(\xi))]=\sum_j\Big(\E_N[e^{\i tX_N}W_{ij}G_{ij}(\xi)]-\E_N[e^{\i tX_N}]\E_N[W_{ij}G_{ij}(\xi)]\Big).
\end{multline*}
To use cumulant expansion to study $(\xi-\lambda v_i)\E_N[e^{\i tX_N}(G_{ii}(\xi)-\E_N G_{ii}(\xi))]$, we  write:
\begin{multline}\label{eqn30}
(\xi-\lambda v_i)\E_N[e^{\i tX_N}(G_{ii}(\xi)-\E_N G_{ii}(\xi))]\\
=\sum_j \frac{1+\delta_{ij}}{N}\Big(\E_N[\frac{\partial e^{\i tX_N}}{\partial W_{ij}}G_{ij}(\xi)+e^{\i tX_N}\frac{-1}{1+\delta_{ij}}(G_{ii}(\xi)G_{jj}(\xi)+(G_{ij}(\xi))^2)]\\
-\E_N[e^{\i tX_N}]\E_N[\frac{-1}{1+\delta_{ij}}(G_{ii}(\xi)G_{jj}(\xi)+(G_{ij}(\xi))^2)]\Big)\\
+\sum_j\dfrac{\E[\vert \sqrt NW_{ij}\vert ^3]}{2N^{3/2}}\Big(\E_N[G_{ij}(\xi)\frac{\partial^2(e^{\i tX_N})}{\partial W_{ij}^2}]+2\E_N[\frac{\partial(e^{\i tX_N})}{\partial W_{ij}}\frac{\partial G_{ij}(\xi)}{\partial W_{ij}}]+\E_N[(e^{\i tX_N}-\E_Ne^{\i tX_N})\frac{\partial^2 G_{ij}(\xi)}{\partial W_{ij}^2}]\Big)
\\
+\mathcal E_1^{(i)}(\xi),\quad\forall\xi\in{\mathbb C }\backslash{\mathbb R }.
\end{multline}
According to Lemma \ref{lemma:cumulant_expansion}, there are constants $r_1>0$  and $N_0>0$ such that if $N>N_0$ then
\begin{align}\label{eqn35}
\vert \mathcal E_1^{(i)}(\xi)\vert \le r_1\cdot \frac{1}{N}\cdot\big(N^{3\varpi-\frac{1}{2}}\frac{1}{\vert \Im \xi\vert }+\frac{1}{\vert \Im \xi\vert ^4}\big)\le\frac{r_1}{N\vert \Im\xi\vert ^4}\quad\text{for any }\xi\in\Gamma_+ 
\end{align}
 Here we  used \eqref{eqn1}, Lemma \ref{lemma:derivative_of_exp(itX_N)} and the condition $\varpi<\frac{1}{14}$ to control $\mathcal E_1^{(i)}(\xi)$. For convenience we let $\mathcal E_2^{(i)}(\xi)$  be the second summation on the right hand side of \eqref{eqn30}:
$$\mathcal E_2^{(i)}(\xi):=\sum_j\dfrac{\E[\vert \sqrt NW_{ij}\vert ^3]}{2N^{3/2}}\Big(\E_N[G_{ij}(\xi)\frac{\partial^2(e^{\i tX_N})}{\partial W_{ij}^2}]+2\E_N[\frac{\partial(e^{\i tX_N})}{\partial W_{ij}}\frac{\partial G_{ij}(\xi)}{\partial W_{ij}}]+\E_N[(e^{\i tX_N}-\E_Ne^{\i tX_N})\frac{\partial^2 G_{ij}(\xi)}{\partial W_{ij}^2}]\Big).$$
Moreover we set:
\begin{multline*}
	\mathcal E_3(\xi):=\frac{1}{N}\sum_ig_i(\xi)\E_N[e^{\i tX_N}(G_{ii}'(\xi)-\E_NG_{ii}'(\xi))]\\
	+\frac{2\i t}{N^{3/2}}\int_{\Gamma_+}f(\xi')\sum_ig_i(\xi)\E_N[e^{\i tX_N}(G(\xi)G'(\xi'))_{ii}]d\xi'\mathds1_{\Omega_V(\varsigma)}-\sum_ig_i(\xi)\Big(\mathcal E_1^{(i)}(\xi)+\mathcal E_2^{(i)}(\xi)\Big).
\end{multline*}
\begin{lemma}\label{lemma:some_computation}
For   any $\xi\in{\mathbb C }\backslash{\mathbb R }$ we have:
\begin{multline}\label{eq99}
(1-\frac{1}{N}\sum_ig_i^2(\xi))	\E_N[e^{\i t X_N}(\tr G(\xi)-\E_N\tr G(\xi))]
\\=
-\sum_ig_i(\xi)\E_N[e^{\i tX_N}(G_{ii}(\xi)-g_i(\xi))(m_{fc}(\xi)-\frac{1}{N}\tr G(\xi))]\\
-\E_N[e^{\i tX_N}]\sum_i g_i(\xi)\E_N[(G_{ii}(\xi)-g_i(\xi))(\frac{1}{N}\tr G(\xi)-  m_{fc}(\xi))]+\mathcal E_3(\xi)
\end{multline}
\end{lemma}
\begin{proof}
See Appendix \ref{appendix}.
\end{proof}

Using \eqref{eqn27} and Lemma \ref{lemma:well_known_facts_for_resolvent}, for any $z\in{\mathbb C }\backslash{\mathbb R }$
\begin{multline}\label{eqn57}
\vert \frac{\partial^2 G_{ij}}{\partial W_{ij}^2}\vert =\frac{\vert 6G_{ij}G_{ii}G_{jj}+2(G_{ij})^3\vert }{(1+\delta_{ij})^2}\le
\begin{cases}
\frac{1}{\vert \Im z\vert ^3}\quad&\text{if }i=j\\
\frac{6}{\vert \Im z\vert ^2}\max_{i\ne j}\vert G_{ij}\vert +2\max_{i\ne j}\vert G_{ij}\vert ^3\quad&\text{if }i\ne j
\end{cases}\\
\le \begin{cases}
	\frac{1}{\vert \Im z\vert ^3}\quad&\text{if }i=j\\
	\frac{8}{\vert \Im z\vert ^2}\max_{i\ne j}\vert G_{ij}\vert \quad&\text{if }i\ne j
\end{cases}
\end{multline}
By \eqref{eqn1}, \eqref{eqn27}, \eqref{eqn57}, Lemma \ref{lemma:well_known_facts_for_resolvent} and Lemma \ref{lemma:derivative_of_exp(itX_N)}, if $\xi\in\Gamma_+$ and $N>N_0$ then
\begin{align}\label{eqn36}
\vert \mathcal E_2^{(i)}(\xi)\vert \le r_2\Big( \frac{N^{2\varpi}}{N\vert \Im \xi\vert }+\frac{1}{\sqrt N\vert \Im \xi\vert ^2}\E_N[\max_{j\ne i}\vert G_{ij}(\xi)\vert ]\Big)
\end{align}
where $r_2>0$  and $N_0>0$ are constants.

By Lemma \ref{lemma:well_known_facts_for_resolvent},
\begin{align}\label{eqn34}
\vert (G(\xi)G'(\xi'))_{ii}\vert \le \vert \Im \xi\vert ^{-1}\vert \Im \xi'\vert ^{-2}\quad\forall \xi,\xi'\in{\mathbb C }\backslash{\mathbb R }.		
\end{align}
Plugging \eqref{eqn34} into the definition of $\mathcal E_3(\xi)$,   by \eqref{eqn35}, \eqref{eqn36} and the fact that $\vert g_i(\xi)\vert \le\vert \Im\xi\vert ^{-1}$, we have

\begin{align}\label{eqn39}
	\vert \mathcal E_3(\xi)\vert \le r_3\big(\vert \Im\xi\vert ^{-5}+N^{2\varpi}\cdot\vert \Im \xi\vert ^{-2}+\frac{\sqrt N}{\vert \Im \xi\vert ^3}\E_N[\max_{j\ne i}\vert G_{ij}(\xi)\vert ]\big),\quad\text{if }\xi\in\Gamma_+\text{ and }N>N_0
\end{align}
where $r_3>0$ and $N_0>0$ are constants.

\begin{lemma}\label{lemma:11111}
	Suppose  $\Omega_V(\varsigma)$ holds. If $N$ is large enough, then:
	$$\Big\vert 1-\frac{1}{N}\sum_ig_i^2(\xi)\Big\vert \ge\frac{1}{3}\vert \Im\xi\vert ^2\quad\forall \xi\in\Gamma_+.$$
\end{lemma}
\begin{proof}
	See Appendix \ref{appendix}.
\end{proof}

By \eqref{eq99}, \eqref{eqn39}, Lemma \ref{lemma:properties_of_good_event} and Lemma \ref{lemma:11111}, if $N$ is large enough and $\xi\in\Gamma_+$  then  
\begin{multline}\label{eqn58}
\vert \E_N[e^{\i t X_N}(\tr G(\xi)-\E_N\tr G(\xi))]\mathds1_{\Omega_V(\varsigma)}\vert \\
\le6N^{1+3\varpi}\Big(\E_N[\vert G_{ii}-g_i\vert \vert m_{fc}-\frac{1}{N}\tr G\vert \mathds1_{\Omega_V(\varsigma)\backslash B_N}]+\E_N[\max_{i\ne j}\vert G_{ij}(\xi)\vert \mathds1_{\Omega_V(\varsigma)\backslash B_N}]\Big)+p_N(\xi)\\
\le 30N^{1+5\varpi}\E_N[\mathds1_{\Omega_V(\varsigma)\backslash B_N}]+p_N(\xi)
\end{multline}
where $p_N(\xi)=\frac{8N^{2\varsigma+\varsigma'}}{ \vert \Im \xi\vert ^6}+\frac{16N^{4\varsigma}}{ \vert \Im \xi\vert ^5}+\frac{2r_3}{\vert \Im\xi\vert ^7}+\frac{2r_3N^{2\varpi}}{\vert \Im\xi\vert ^4}+\frac{2r_3N^{\varsigma'}}{\vert \Im\xi\vert ^8}$.

According to \eqref{eqn58}, \eqref{eqn40} and \eqref{eqn43}, there exist constants $r_4>0$ and $N_0>0$ such that if $N>N_0$ then
\begin{multline*}
\Big\vert \frac{d}{dt}\E[\exp(\i t X_N)]\Big\vert \le\frac{1}{\sqrt N}\int_{\Gamma_+}\vert f(\xi)\vert \Big(30N^{1+5\varpi}{\mathbb P}(\Omega_V(\varsigma)\backslash B_N)+p_N(\xi)\Big)d\xi\\
\le r_4(N^{5\varpi+2\varsigma+\varsigma'-\frac{1}{2}}+N^{4\varpi+4\varsigma-\frac{1}{2}}+N^{7\varpi+\varsigma'-\frac{1}{2}})
\end{multline*}
and  therefore
\begin{align}\label{eqn61}
\Big\vert \E[\exp(\i t X_N)]-1\Big\vert \le t\cdot r_4(N^{5\varpi+2\varsigma+\varsigma'-\frac{1}{2}}+N^{4\varpi+4\varsigma-\frac{1}{2}}+N^{7\varpi+\varsigma'-\frac{1}{2}}).
\end{align}
By the conditions on $\varpi$, $\varsigma$ and $\varsigma'$ in Definition \ref{definition:constants}, the exponent for each term on the right hand side of \eqref{eqn61} must be negative. So  \eqref{eqn41} is true and we complete the proof of Lemma \ref{lemma:convergence_of_P_1}.
\end{proof}

\subsection{Proof of Lemma \ref{lemma:convergence_of_P_2}}\label{sec:proof_of_convergence_of_P_2}
\begin{proof}[Proof of Lemma \ref{lemma:convergence_of_P_2}] 

According to Proposition \ref{proposition:extreme_eigenvalue}, it suffices to prove that
\begin{align}\label{eqn14}
\frac{1}{\sqrt N}\int_{\Gamma_+}f(\xi)\Big[N \hat m_{fc}(\xi)-\E_N\tr G(\xi)\Big]\cdot\mathds1_{\Omega_V(\varsigma)}d\xi\to0\quad\text{in distribution}.
\end{align} 
Let
$$Q_i=-\lambda v_i-W_{ii}+\sum_{p,q}^{(i)}W_{ip}G_{pq}^{(i)}W_{qi}.$$
\begin{lemma}\label{lemma:computation}
For any $\xi\in{\mathbb C }\backslash{\mathbb R }$, $\E_N[\tr G(\xi)-N\hat m_{fc}(\xi)]$ equals:
\begin{multline}\label{eqn23}
(1+\hat m_{fc}'(\xi))\Big(-\frac{1}{N}\sum_{i=1}^N\hat g_i^2(\xi)\E_N[G_{ii}(\xi)+\frac{1}{G_{ii}(\xi)}\sum_p^{(i)}(G_{ip}(\xi))^2]+\E_N[\sum_{i=1}^N\frac{(G_{ii}(\xi)-\hat g_i(\xi))^3}{(G_{ii}(\xi))^2}]\Big)
\\+(1+\hat m_{fc}'(\xi))\sum_{i=1}^N\hat g_i^3(\xi)\Big(\frac{2}{N}+\frac{2}{N^2}\E_N[(\tr G^{(i)}(\xi))']+\sum_p^{(i)}\E_N[(G_{pp}^{(i)}(\xi))^2](\E[W_{ip}^4]-\frac{3}{N^2})+\E_N[(\hat m_{fc}(\xi)-\frac{1}{N}\tr G^{(i)}(\xi))^2]\Big)
\end{multline}
\end{lemma}
\begin{proof}See Appendix \ref{appendix}.
\end{proof}

There exists a constant $u_1\in(0,1)$ such that
if $\xi\in \Gamma\backslash{\mathbb R }$, then 
$$\vert 1/\hat g_i\vert =\vert \lambda v_i-\xi-\hat m_{fc}\vert \le \lambda+\vert \xi\vert +\frac{1}{\vert \Im\xi\vert }\le\frac{1}{u_1\vert \Im\xi\vert }$$
and thus by \eqref{eq102}
\begin{align}\label{eqn59}
u_1\cdot\vert \Im\xi\vert \le\vert \hat g_i(\xi)\vert \le \frac{1}{\vert \Im\xi\vert }.
\end{align}
By the definition of Stieltjes transform, $\vert \hat m_{fc}'\vert \le\frac{1}{\vert \Im\xi\vert ^2}$. This together with \eqref{eqn1}, Lemma \ref{lemma:well_known_facts_for_resolvent} and \eqref{eqn59} yield:   
\begin{align}\label{eqn64}
	\vert (1+\hat m_{fc}')\frac{1}{N}\sum_{i=1}^N\hat g_i^2(\xi)\E_N[G_{ii}(\xi)]\vert \le2\vert \Im\xi\vert ^{-5},\quad\forall \xi\in \Gamma\backslash{\mathbb R }
\end{align}
\begin{multline}\label{eqn67}
	\vert (1+\hat m_{fc}')\sum_{i=1}^N\hat g_i^3(\xi)\Big(\frac{2}{N}+\frac{2}{N^2}\E_N[(\tr G^{(i)})']+\sum_p^{(i)}\E_N[(G_{pp}^{(i)})^2](\E[W_{ip}^4]-\frac{3}{N^2})\Big)\vert \\
	\le 2\vert \Im\xi\vert ^{-5}\Big(2+2\vert \Im\xi\vert ^{-2}+\vert \Im\xi\vert ^{-2}(\max_{a,b}\E[(\sqrt NW_{ab})^4]+3)\Big)\le u_2\vert \Im\xi\vert ^{-7},\quad\forall \xi\in \Gamma\backslash{\mathbb R }
\end{multline}
where $u_2>0$ is a constant. According to Lemma \ref{lemma:properties_of_good_event}, there are constants $u_3>0$ and $N_0>0$ such that if   $N>N_0$ and $\xi\in\Gamma_+$, then 
\begin{multline}\label{eqn65}
\vert (1+\hat m_{fc}')\frac{1}{N}\sum_{i=1}^N\hat g_i^2(\xi)\E_N[\frac{1}{G_{ii}(\xi)}\sum_p^{(i)}(G_{ip}(\xi))^2\mathds1_{\Omega_V(\varsigma)\cap B_N}]\vert 
\le\frac{u_3}{\vert \Im\xi\vert ^5}\E_N[\vert \sum_p^{(i)}(G_{ip}(\xi))^2\vert \mathds1_{\Omega_V(\varsigma)\cap B_N}]\\
=\frac{u_3}{\vert \Im\xi\vert ^5}\E_N[\vert (G^2(\xi))_{ii}-(G_{ii}(\xi))^2\vert \mathds1_{\Omega_V(\varsigma)\cap B_N}]\le u_3\vert \Im\xi\vert ^{-7}\quad\text{(by Lemma \ref{lemma:well_known_facts_for_resolvent})}
\end{multline}
\begin{align}\label{eqn66}
\vert (1+\hat m_{fc}')\E_N[\sum_{i=1}^N\frac{(G_{ii}(\xi)-\hat g_i(\xi))^3}{(G_{ii}(\xi))^2}\mathds1_{\Omega_V(\varsigma)\cap B_N}]\vert \le u_3(\frac{N^{3\varsigma'}}{\sqrt N\vert \Im \xi\vert ^{13}}+\frac{N^{6\varsigma}}{\sqrt N\vert \Im\xi\vert ^{10}})
\end{align}
\begin{align}\label{eqn68}
	\vert (1+\hat m_{fc}')\sum_{i=1}^N\hat g_i^3(\xi)\E_N[(\hat m_{fc}-\frac{1}{N}\tr G^{(i)})^2\mathds1_{\Omega_V(\varsigma)\cap B_N}]\vert \le u_3N^{4\varsigma}\vert \Im\xi\vert ^{-5}.
\end{align}
 By Lemma \ref{lemma:computation} and the conditions in Definition \ref{definition:constants}, the terms on the left hand side of \eqref{eqn64}, \eqref{eqn67}, \eqref{eqn65}, \eqref{eqn66}, \eqref{eqn68} all make $o(1)$ contribution to the quantity in \eqref{eqn14}. So to prove \eqref{eqn14} it suffices to show that
\begin{align}\label{eqn24}
\frac{1}{N^{3/2}}\int_{\Gamma_+}f(\xi)(1+\hat m_{fc}')\sum_{i=1}^N\hat g_i^2(\xi)\E_N[\frac{1}{G_{ii}(\xi)}\sum_p^{(i)}(G_{ip}(\xi))^2\mathds1_{\Omega_V(\varsigma)\backslash B_N}]d\xi\to0\quad\text{in distribution}
\end{align}
\begin{align}\label{eqn25} 
\frac{1}{\sqrt N}\int_{\Gamma_+}f(\xi)(1+\hat m_{fc}')\E_N[\sum_{i=1}^N\frac{(G_{ii}(\xi)-\hat g_i(\xi))^3}{(G_{ii}(\xi))^2}\mathds1_{\Omega_V(\varsigma)\backslash B_N}]d\xi\to0\quad\text{in distribution}
\end{align}
\begin{align}\label{eqn26}
\frac{1}{\sqrt N}\int_{\Gamma_+}f(\xi)(1+\hat m_{fc}')\sum_{i=1}^N\hat g_i^3(\xi)\E_N[(\hat m_{fc}-\frac{1}{N}\tr G^{(i)})^2\mathds1_{\Omega_V(\varsigma)\backslash B_N}]d\xi\to0\quad\text{in distribution}
\end{align}
\begin{itemize}
	\item  Notice that $\vert \hat m_{fc}\vert $ and $\vert \frac{1}{N}\tr G^{(i)}\vert $ are bounded by $\frac{1}{\vert \Im \xi\vert }$. By \eqref{eqn43},
if  $N$ is large enough and $\xi\in\Gamma_+$ then
\begin{multline*}
{\mathbb P}\Big(\Big\vert \E_N[(\hat m_{fc}(\xi)-\frac{1}{N}\tr G^{(i)}(\xi))^2\mathds1_{\Omega_V(\varsigma)\backslash B_N}]\Big\vert >N^{-5}\Big)\le N^5\E\Big[\Big\vert \E_N[(\hat m_{fc}(\xi)-\frac{1}{N}\tr G^{(i)}(\xi))^2\mathds1_{\Omega_V(\varsigma)\backslash B_N}]\Big\vert \Big]\\
\le N^5\E\Big[\Big\vert  (\hat m_{fc}(\xi)-\frac{1}{N}\tr G^{(i)}(\xi))^2\mathds1_{\Omega_V(\varsigma)\backslash B_N}\Big\vert \Big]\le4N^{5+2\varpi}{\mathbb P}(\Omega_V(\varsigma)\backslash B_N)<N^{-100}
\end{multline*}
which together with a classic ``lattice" argument yields
$${\mathbb P}\Big(\Big\vert \E_N[(\hat m_{fc}(\xi)-\frac{1}{N}\tr G^{(i)}(\xi))^2\mathds1_{\Omega_V(\varsigma)\backslash B_N}]\Big\vert \le2N^{-5},\quad\forall \xi\in\Gamma_+\Big)\ge1-N^{-20}.$$
This and the facts that $\vert \hat m_{fc}'(\xi)\vert \le N^{2\varpi}$ and $\vert \hat g_i(\xi)\vert \le N^\varpi$ on $\Gamma_+$ complete the proof of \eqref{eqn26}.

\item According to  \eqref{eqn43} and  \eqref{eqn15}, if $N$ is large enough and $\xi\in\Gamma_+$, then
\begin{multline*}
	{\mathbb P}\Big(\Big\vert \E_N[\frac{1}{G_{ii}(\xi)}\sum_p^{(i)}(G_{ip}(\xi))^2\mathds1_{\Omega_V(\varsigma)\backslash B_N}]\Big\vert >N^{-5}\Big)\le N^5\E\Big[\Big\vert \E_N[\frac{1}{G_{ii}(\xi)}\sum_p^{(i)}(G_{ip}(\xi))^2\mathds1_{\Omega_V(\varsigma)\backslash B_N}]\Big\vert \Big]\\
	\le N^5\E\Big[\Big\vert  \frac{1}{G_{ii}(\xi)}\sum_p^{(i)}(G_{ip}(\xi))^2\Big\vert \mathds1_{\Omega_V(\varsigma)\backslash B_N}\Big]\le N^{6+2\varpi}\E\Big[\Big\vert  \frac{1}{G_{ii}(\xi)}\Big\vert \mathds1_{\Omega_V(\varsigma)\backslash B_N}\Big]\\
	\le N^{6+2\varpi}\sqrt{\E\Big[ \frac{1}{\vert G_{ii}(\xi)\vert ^2}\Big]}\sqrt{{\mathbb P}(\Omega_V(\varsigma)\backslash B_N)}\\
	=N^{6+2\varpi}\sqrt{\E\Big[ \vert \lambda v_i+W_{ii}-\xi+\sum_{k,l}^{(i)}W_{ik}G_{kl}^{(i)}(
		\xi)W_{li}\vert ^2
		\Big]}\sqrt{{\mathbb P}(\Omega_V(\varsigma)\backslash B_N)}
	\le N^{8+3\varpi}\sqrt{{\mathbb P}(\Omega_V(\varsigma)\backslash B_N)}\le N^{-100}
\end{multline*}
which together with \eqref{eqn15} and a classic  ``lattice" argument yields
$${\mathbb P}\Big(\Big\vert \E_N[\frac{1}{G_{ii}(\xi)}\sum_p^{(i)}(G_{ip}(\xi))^2\mathds1_{\Omega_V(\varsigma)\backslash B_N}]\Big\vert \le2N^{-5},\quad\forall \xi\in\Gamma_+\Big)\ge1-N^{-20}.$$
This and the facts that $\vert \hat m_{fc}'(\xi)\vert \le N^{2\varpi}$ and $\vert \hat g_i(\xi)\vert \le N^\varpi$ on $\Gamma_+$ complete the proof of \eqref{eqn24}.

\item Similarly, according to \eqref{eqn43} and  \eqref{eqn15}, if $N$ is large enough and  $\xi\in\Gamma_+$, then
\begin{multline*}
	{\mathbb P}\Big(\Big\vert \E_N[\sum_{i=1}^N\frac{(G_{ii}(\xi)-\hat g_i(\xi))^3}{(G_{ii}(\xi))^2}\mathds1_{\Omega_V(\varsigma)\backslash B_N}]\Big\vert >N^{-5}\Big)\le N^5\E\Big[\Big\vert \E_N[\sum_{i=1}^N\frac{(G_{ii}(\xi)-\hat g_i(\xi))^3}{(G_{ii}(\xi))^2}\mathds1_{\Omega_V(\varsigma)\backslash B_N}]\Big\vert \Big]\\
	\le 8N^{5+3\varpi}\sum_i\sqrt{\E[\frac{1}{\vert G_{ii}(\xi)\vert ^4}]{\mathbb P}(\Omega_V(\varsigma)\backslash B_N)}<N^{-100}
\end{multline*}
which together with \eqref{eqn15} and a classic ``lattice" argument yields
$${\mathbb P}\Big(\Big\vert \E_N[\sum_{i=1}^N\frac{(G_{ii}(\xi)-\hat g_i(\xi))^3}{(G_{ii}(\xi))^2}\mathds1_{\Omega_V(\varsigma)\backslash B_N}]\Big\vert \le2N^{-5},\quad\forall \xi\in\Gamma_+\Big)\ge1-N^{-20}.$$
This and the facts that $\vert \hat m_{fc}'(\xi)\vert \le N^{2\varpi}$ and $\vert \hat g_i(\xi)\vert \le N^\varpi$ on $\Gamma_+$ complete the proof of \eqref{eqn25}.
\end{itemize}  
\end{proof}

\section{Rigidity of eigenvalues: proof of Theorem \ref{thm:rigidity_concise_version}} \label{sec:rigidity}

In this section we prove the rigidity of eigenvalues in the bulk of the spectrum. For an eigenvalue $\lambda_i$ in the bulk, we show that it is very close to the deterministic number $\gamma_i$ with high probability. Roughly speaking, the distance between $\lambda_i$ and $\gamma_i$ is no more than $N^{-\frac{1}{4}+\frac{1}{1+b}}$  with high probability.

By the definition of $\gamma_i$ and $\hat\gamma_i$ (see Definition \ref{definition:classic_position}) we have that
\begin{align}\label{eq18}
	\int_{\hat\gamma_{i+1}}^{\hat\gamma_i}\frac{1}{L_+-t}d\mu_{fc}(t)\le \frac{1}{N}\frac{1}{L_+-\gamma_i}\le \int_{\hat\gamma_{i-1}}^{\hat\gamma_{i-2}}\frac{1}{L_+-t}d\mu_{fc}(t),\quad\forall i\in[2,N-1].
\end{align}

\begin{lemma}\label{lemma:distance_between_gamma_i_and_edge}
	There exists a constant $C_*\ge1$ such that 
	\begin{itemize}
		\item if $b>1$ and $ \lambda>\lambda_+$ then $C_*^{-1}\Big(\frac{i}{N}\Big)^{\frac{1}{1+b}}\le\vert L_+-\gamma_i\vert \le C_*\Big(\frac{i}{N}\Big)^{\frac{1}{1+b}}$;
		\item if $a>1$ and $ \lambda>\lambda_-$  then $C_*^{-1}\Big(\frac{N-i}{N}\Big)^{\frac{1}{1+a}}\le\vert L_--\gamma_i\vert \le C_*\Big(\frac{N-i}{N}\Big)^{\frac{1}{1+a}}$.
	\end{itemize}
\end{lemma}
\begin{proof}
	Suppose $b>1$, $ \lambda>\lambda_+$. According to Lemma \ref{lemma:properties_of_mu_fc}, there exists $C_*'>1$ such that
	$$\frac{(L_+-x)^b}{C_*'}\le \rho_{fc}(x)\le C_*'(L_+-x)^b$$
	for $x\in[\gamma_{0.99N},L_+]$. Therefore if $i\le 0.99N$, then
	$$\frac{(L_+-\gamma_i)^{b+1}}{C_*'(b+1)}=\int_{\gamma_i}^{L_+}\frac{(L_+-x)^b}{C_*'}dx\le\int_{\gamma_i}^{L_+}\rho_{fc}(x)dx=\frac{i-\frac{1}{2}}{N}\le \int_{\gamma_i}^{L_+}C_*'(L_+-x)^bdx=\frac{C_*'(L_+-\gamma_i)^{b+1}}{(b+1)}.$$
	If $i>0.99N$, then both $i/N$ and $\vert L_+-\gamma_i\vert $ are of order 1, so the inequality also holds.
	This proves the first conclusion. The second conclusion can be proved in the same way.
\end{proof}

\begin{definition}\label{definition:A_N}
For $\epsilon>0$, set
$$A_N(\epsilon):=\Omega_V(\epsilon)\cap\tilde\Omega(\epsilon)\cap\Omega_*(\epsilon)$$
where $\Omega_V(\epsilon)$, $\tilde\Omega(\epsilon)$ and $\Omega_*(\epsilon)$ are defined in Definition \ref{definition:D_epsilon_andD'_epsilon} and Definition \ref{definition:Omega_V}.
\end{definition}

\begin{lemma}\label{lemma:basic}
	Suppose $b>1$ and $\epsilon\in(\frac{1}{1+b},\frac{11b-9}{2b+2})$. Suppose $A_N(\epsilon)$ holds and $z_0$ is in
	\begin{align}\label{eq2}
		\Big\{x+\i y\Big\vert \vert x\vert \le3+\lambda,y\in[\frac{1}{2}N^{-\frac{1}{1+b}-\epsilon},N^{-\frac{1}{1+b}+\epsilon}]\Big\}.
	\end{align}
If $N$ is large enough, then we have $z_0\in\mathcal D_\epsilon'$ and 
	\begin{align}\label{eq1}
		\vert m_N(z_0)-m_{fc}(z_0)\vert \le 2N^{2\epsilon-\frac{1}{2}}.
	\end{align}
\end{lemma}
\begin{remark}
	In the condition $\epsilon\in(\frac{1}{1+b},\frac{11b-9}{2b+2})$, the interval is not empty since $b>1$.
\end{remark}
\begin{proof}
	By the condition $b>1$ it's easy to see that \eqref{eq2} is contained in $\mathcal D_\epsilon$. We notice that $\Im z_0$ and $\Im m_{fc}(z_0)$ have the same sign, so $\vert \lambda v_i-z_0-m_{fc}(z_0)\vert >\vert \Im z_0\vert \ge \frac{1}{2}N^{-\frac{1}{1+b}-\epsilon}$. Thus $z_0\in\mathcal D_\epsilon'$. Finally \eqref{eq1} is from Lemma \ref{lemma:hatm_fc_close_tom_fc} and the definition of $\tilde\Omega(\epsilon)$.
\end{proof}

Suppose $\epsilon>0$ and
\begin{itemize}
	\item $I$ is an interval contained in $(-2.99-\lambda,2.99+\lambda)$ and it may depend on $N$,
	\item  $\eta_0=N^{-\frac{1}{4}+\epsilon}$ and $\eta_1=N^{-\frac{1}{1+b}-\epsilon}$;
	\item $\chi:{\mathbb R }\to[0,1]$ is a $C^\infty$ function supported on $[-2,2]$ such that $\chi(x)=1$ when $x\in[-1,1]$;
	\item $f:{\mathbb R }\to[0,1]$ is a smooth $N$-depending function such that
	\begin{align} \label{eq3}
		f(x)=
		\begin{cases}
			1\quad&\text{if }x\in I\\
			0\quad&\text{if }\text{dist}(x,I)\ge \eta_0
		\end{cases}
	\end{align}
and $\|f'\|_\infty\le C_f\cdot\eta_0^{-1}$, $\|f''\|_\infty\le C_f\cdot\eta_0^{-2}$ for some  absolute constant $C_f>0$.
\end{itemize}
\begin{remark}\label{remark:properties_of_f}
It is easy to see that	$f$ satisfies the following properties.
	\begin{enumerate}
		\item If $\epsilon<\frac{1}{4}$ then $\text{supp}f\subset(-2.995-\lambda,2.995+\lambda)$ when $N>N_0=N_0(\epsilon)$.
		\item $\vert \text{supp}f'\vert \le2\eta_0$, $\vert \text{supp}f''\vert \le2\eta_0$.
	\end{enumerate}
\end{remark}
Recall that $\mu_N$ is defined in Definition \ref{definition:Stieltjes_transforms}.
\begin{lemma}\label{lemma:difference_of_integral_for_f}
 Suppose $b>3$ and $\epsilon\in(\frac{1}{1+b},\frac{1}{4})$. 	Suppose $A_N(\epsilon)$ hold. Then
	$$\vert \int f(t)d\mu_N(t)-\int f(t)d\mu_{fc}(t)\vert \le \alpha_1(\eta_0^2+\eta_1^2\eta_0)$$
for large enough $N$.	Here $\alpha_1>0$ is a constant depending only on $C_f$. 
\end{lemma}
\begin{remark}
	The condition $\epsilon\in(\frac{1}{1+b},\frac{1}{4})$ is stronger than the condition $\epsilon\in(\frac{1}{1+b},\frac{11b-9}{2b+2})$ in Lemma \ref{lemma:basic} because the condition $b>3$ ensures that $\frac{1}{4}<\frac{11b-9}{2b+2}$.
\end{remark}
\begin{proof}
	According to the Helffer-Sj\"ostrand formula (see Lemma \ref{lemma:Helffer-Sjostrand formula}),
	\begin{align*}
		&\int f(t)d\mu_N(t)-\int f(t)d\mu_{fc}(t)\\
		=&\frac{1}{2\pi}\int\Big(\int_{{\mathbb R }^2}\frac{\i yf''(x)\chi(y)+\i(f(x)+\i yf'(x))\chi'(y)}{t-x-\i y}dxdy\Big)(d\mu_N(t)-d\mu_{fc}(t))\\
		=&\frac{1}{2\pi}\int_{{\mathbb R }^2}\Big[\i yf''(x)\chi(y)+\i(f(x)+\i yf'(x))\chi'(y)\Big](m_N(x+\i y)-m_{fc}(x+\i y))dxdy\\
		=&\frac{1}{2\pi}\Re\Bigg(\int_{{\mathbb R }^2}\Big[\i yf''(x)\chi(y)+\i(f(x)+\i yf'(x))\chi'(y)\Big](m_N(x+\i y)-m_{fc}(x+\i y))dxdy\Bigg)\\
		=&K_1-K_2-K_3
	\end{align*}
	where
	$$K_1=\Re\Big(\frac{\i}{2\pi}\int_{{\mathbb R }^2}\chi'(y)(f(x)+\i yf'(x))(m_N(x+\i y)-m_{fc}(x+\i y))dxdy\Big)$$
	$$K_2=\frac{1}{2\pi}\int_{{\mathbb R }}\int_{\vert y\vert \ge\eta_1}yf''(x)\chi(y)\Im(m_N(x+\i y)-m_{fc}(x+\i y))dydx$$ 
	$$K_3=\frac{1}{2\pi}\int_{{\mathbb R }}\int_{\vert y\vert <\eta_1}yf''(x)\Im(m_N(x+\i y)-m_{fc}(x+\i y))dydx$$
	
	Recall $\epsilon>\frac{1}{1+b}$ and that  $A_N(\epsilon)$ holds. For simplicity let $\tilde m(z)=m_N(z)-m_{fc}(z)$. 
	
	\begin{enumerate}
		\item[{\bf 1.}] We first estimate $K_1$. By  \eqref{eq1}, Remark \ref{remark:properties_of_f} and the definition of $\chi$, if $N$ is large enough then:
		\begin{align*}
			\vert K_1\vert \le2N^{2\epsilon-\frac{1}{2}}\cdot\frac{1}{2\pi} \cdot2\|\chi'\|_\infty \int \vert f(x)\vert +2\vert f'(x)\vert dx\le CN^{2\epsilon-\frac{1}{2}}=C\eta_0^2
		\end{align*}
		where $C>0$ is a constant depending only on $C_f$.
		\item[{\bf 2.}] Then we estimate $K_2$. Suppose $x+\i y$ is in the support of $yf''(x)\chi(y)\mathds1_{\vert y\vert \ge\eta_1}(x+\i y)$. Let $L$ be the counterclockwise circle centered at $x+\i y$ with radius $\min(\frac{y}{2},0.005)$. If $N>N_0=N_0(\epsilon)$, then $L$ is contained in \eqref{eq2} and we have by Cauchy's Theorem and \eqref{eq1} that
		$$\vert \partial_z\tilde m(x+\i y)\vert =\Big\vert \frac{1}{2\pi\i}\oint_L\frac{\tilde m(r)}{(r-x-\i y)^2}dr\Big\vert \le4N^{2\epsilon-\frac{1}{2}}\max(\frac{1}{y},100)$$
		and therefore
		\begin{multline*}
			\vert K_2\vert \le \frac{1}{2\pi}\Big\vert \int_{\vert y\vert \in[\eta_1,2]}\Im\Big[-\int_{{\mathbb R }} f'(x)\partial_z\tilde m(x+\i y)dx\Big]y\chi(y)dy\Big\vert  \quad\text{(integral by parts for $x$)}\\
			\le\frac{1}{2\pi}\cdot 4N^{2\epsilon-\frac{1}{2}}\cdot 2C_f\cdot2\int_{\eta_1}^2y\chi(y)\max(\frac{1}{y},100)dy\le \frac{3200C_f}{\pi}N^{2\epsilon-\frac{1}{2}}=\frac{3200C_f\cdot\eta_0^2}{\pi}
		\end{multline*}
		\item[{\bf 3.}] Finally we estimate $K_3$. Notice that both $y\Im m_N(x+\i y)$ and $y\Im m_{fc}(x+\i y)$ are nonnegative and increase when $y\in[0,\eta_1]$. So,
		\begin{align*}
			\vert K_3\vert =&\frac{1}{\pi}\Big\vert \int_{{\mathbb R }}f''(x)\int_0^{\eta_1} (y\Im m_N(x+\i y)-y\Im m_{fc}(x+\i y))dydx\Big\vert \\
			\le& \frac{1}{\pi}\int_{{\mathbb R }}\vert f''(x)\vert \int_0^{\eta_1} (y\Im m_N(x+\i y)+y\Im m_{fc}(x+\i y))dydx\\
			\le&\frac{1}{\pi}\int_{{\mathbb R }}\vert f''(x)\vert \int_0^{\eta_1} (\eta_1\Im m_N(x+\i \eta_1)+\eta_1\Im m_{fc}(x+\i \eta_1))dydx\\
			=&\frac{\eta_1^2}{\pi} \int_{{\mathbb R }}\vert f''(x)\vert (\Im m_N(x+\i \eta_1)+\Im m_{fc}(x+\i \eta_1))dx
		\end{align*}

		If $x\in\text{supp}f$, then $x+\i \eta_1$ is in \eqref{eq2} and thus also in $\mathcal D_\epsilon'$, provided $N$ is large enough. So we know by \eqref{eq1} and the definition of $\Omega_*(\epsilon)$ that if $N$ is large enough then
		$$\Im m_N(x+\i \eta_1)+\Im m_{fc}(x+\i \eta_1)\le 4N^{2\epsilon-\frac{1}{2}}$$
		and thus
		$$\vert K_3\vert \le \frac{4\eta_1^2}{\pi}N^{2\epsilon-\frac{1}{2}}\int_{{\mathbb R }}\vert f''(x)\vert dx\le \frac{8C_f\cdot\eta_1^2}{\pi\eta_0}N^{2\epsilon-\frac{1}{2}}=\frac{8C_f\cdot\eta_1^2\cdot\eta_0}{\pi}$$	
	\end{enumerate}
	The estimates for $K_1$, $K_2$ and $K_3$ together complete the proof of the lemma.
\end{proof}

\begin{lemma}\label{lemma:mu_N_and_mu_fc_are_close}
 Suppose $b>3$ and $\epsilon\in(\frac{1}{1+b},\frac{1}{4})$.	Suppose  $A_N(\epsilon)$ holds. If $N$ is large enough, then for any (possibly $N$-depending) interval $J\subset{\mathbb R }$,
	$$\vert \mu_N(J)-\mu_{fc}(J)\vert \le 4\eta_0(\|\rho_{fc}(x)\|_\infty+\alpha_2)$$
	where $\alpha_2>0$ is a constant depending only on $C_f$.
\end{lemma} 

\begin{proof}
	First, suppose $J\subset(-2.99-\lambda,2.99+\lambda)$. Define $h:{\mathbb R }\to[0,1]$ to be a smooth function such that
	\begin{align}  \label{eq14}
		h(x)=
		\begin{cases}
			1\quad&\text{if }x\in J\\
			0\quad&\text{if }\text{dist}(x,J)\ge \eta_0
		\end{cases}
	\end{align} From Lemma \ref{lemma:difference_of_integral_for_f}, if $N$ is large enough then
	\begin{align}\label{eq5}
		\mu_N(J)\le \int h(x)d\mu_N(x)\le \int h(x)d\mu_{fc}(x)+\alpha_1(\eta_0^2+\eta_1^2\eta_0)\le\mu_{fc}(J)+2\eta_0\|\rho_{fc}(x)\|_\infty+\alpha_1(\eta_0^2+\eta_1^2\eta_0)
	\end{align} 
	where $\alpha_1>0$ is defined in Lemma \ref{lemma:difference_of_integral_for_f}. On the other hand, let 
	$$J'=\{x\in J\vert \text{the distance between $x$ and the edges of $J$ is no less than }\eta_0\}$$
	and define $\tilde h$ in the same way as $h$, except that $J$ in\eqref{eq14} is replaced by $J'$. So by Lemma \ref{lemma:difference_of_integral_for_f}, if $N$ is large enough then
	\begin{align}\label{eq4}
		\mu_N(J)\ge\int \tilde h(x)d\mu_N(x)\ge\int\tilde h(x)d\mu_{fc}(x)-\alpha_1(\eta_0^2+\eta_1^2\eta_0)\ge\mu_{fc}(J)-2\eta_0\|\rho_{fc}(x)\|_\infty-\alpha_1(\eta_0^2+\eta_1^2\eta_0).
	\end{align}
	We remark that if $\vert J\vert <2\eta_0$, then $J'=\emptyset$, but  \eqref{eq4} is trivial in this case. By \eqref{eq4} and \eqref{eq5},
	\begin{align}\label{eq6}
		\vert \mu_N(J)-\mu_{fc}(J)\vert \le \eta_0(2\|\rho_{fc}(x)\|_\infty+\alpha_1(\eta_0+\eta_1^2))
	\end{align}
	so we complete the proof in the case that $J\in(-2.99-\lambda,2.99+\lambda)$.
	
	Then, suppose $J$  is not necessarily contained in $(-2.99-\lambda,2.99+\lambda)$. By Lemma \ref{lemma:all_eigenvalues_in_2+lambda+r}, 
	$$\mu_{fc}([-2-\lambda,2+\lambda])=1.$$
	So by \eqref{eq6} we have
	\begin{multline*}
		\mu_N((-2.99-\lambda,2.99+\lambda)^c)=1-\mu_N((-2.99-\lambda,2.99+\lambda))\\\le 1-\Big(\mu_{fc}\Big((-2.99-\lambda,2.99+\lambda)\Big)-\eta_0(2\|\rho_{fc}(x)\|_\infty+\alpha_1(\eta_0+\eta_1^2))\Big)=\eta_0(2\|\rho_{fc}(x)\|_\infty+\alpha_1(\eta_0+\eta_1^2))
	\end{multline*}
	Let $J_1=J\cap (-2.99-\lambda,2.99+\lambda)$ and $J_2=J\backslash J_1$. So by \eqref{eq6} and the above inequality,
	\begin{multline*}
		\vert \mu_N(J)-\mu_{fc}(J)\vert \le \vert \mu_N(J_1)-\mu_{fc}(J_1)\vert +\vert \mu_N(J_2)-\mu_{fc}(J_2)\vert \\
		\le \eta_0(2\|\rho_{fc}(x)\|_\infty+\alpha_1(\eta_0+\eta_1^2))+\mu_N(J_2)\le2\eta_0(2\|\rho_{fc}(x)\|_\infty+\alpha_1(\eta_0+\eta_1^2)).
	\end{multline*}
\end{proof}

\begin{proof}[Proof of Theorem \ref{thm:rigidity_concise_version}]
	By Proposition \ref{proposition:extreme_eigenvalue} and Lemma \ref{lemma:Omega_*}, there exist constants $N_0>0$ and $\nu_0>0$ such that if $N>N_0$ then $${\mathbb P}(A_N(\epsilon))\ge 1-2\nu_0(\log N)^{1+2b}N^{-\epsilon}.$$
Suppose $A_N(\epsilon)$ holds. Now it suffices	to prove   \eqref{eq103} and \eqref{eq104}. 
	
	Suppose $L_0$ is the unique point in $[L_-,L_+]$ such that $\mu_{fc}([L_0,L_+])=2/3$. If $N$ is large enough, then by Lemma \ref{lemma:mu_N_and_mu_fc_are_close},
	$$\mu_N([L_0,L_+])\ge\frac{1}{2}$$
	thus 
\begin{align}\label{eq105}
\lambda_i\ge L_0,\quad\forall i\in[1,\frac{N}{2}].
\end{align}

	Define $g(x)$ by
	$$g(x)=\mu_{fc}([x,+\infty)).$$
	According to Lemma \ref{lemma:mu_N_and_mu_fc_are_close}, if $N$ is large enough, then
	\begin{multline}\label{eq7} 
		\vert g(\lambda_i)-g(\gamma_i)\vert \le\vert g(\lambda_i)-\frac{i}{N}\vert +\vert \frac{i}{N}-g(\gamma_i)\vert \le\vert \mu_{fc}([\lambda_i,+\infty))-\mu_N([\lambda_i,+\infty))\vert +\frac{1}{2N}\\
		\le  5\eta_0(\|\rho_{fc}(x)\|_\infty+\alpha_2),\quad\forall i\in[1,N]
	\end{multline}
	where $\alpha_2>0$ is defined in Lemma \ref{lemma:mu_N_and_mu_fc_are_close}.
	
By \eqref{eq12} there is a constant $C>1$ such that 
	\begin{align}\label{eq9}
	\frac{(L_+-x)^b}{C}\le\rho_{fc}(x)\le C(L_+-x)^b,\quad\forall x\in[L_0,L_+]
\end{align}
and therefore
$$\frac{i}{2N}\le\frac{i-\frac{1}{2}}{N}=\int_{\gamma_i}^{	L_+}\rho_{fc}(x)dx\le C\int_{\gamma_i}^{L_+}(L_+-x)^bdx=\frac{C}{b+1}\vert L_+-\gamma_i\vert ^{b+1}\quad\forall 1\le i\le \frac{N}{2}.$$
Then we have
\begin{align}\label{eq106}
\vert \gamma_i-L_+\vert \ge \big(\frac{i}{2N}\frac{b+1}{C}\big)^{\frac{1}{1+b}}\ge N^{-\zeta},\quad \forall i\in[\frac{2C}{1+b}N^{1-\zeta(b+1)},\frac{N}{2}].
\end{align}

	We  control $\vert \gamma_i-\lambda_i\vert $ in two cases.
	\begin{enumerate}
		\item[{\bf Case 1.}] Suppose $i\in[\frac{2C}{1+b}N^{1-\zeta(b+1)},\frac{N}{2}]$ and $\lambda_i\le\gamma_i$. By \eqref{eq105}, \eqref{eq7}, \eqref{eq9} and \eqref{eq106}, when $N$ is large enough, there exists  $ s\in(\lambda_i,\gamma_i)$ such that
	\begin{align}\label{eq107}
\vert \gamma_i-\lambda_i\vert =\frac{\vert g(\gamma_i)-g(\lambda_i)\vert }{\vert g'(s)\vert }\le\frac{5\eta_0(\|\rho_{fc}(x)\|_\infty+\alpha_2)}{\rho_{fc}(s)}\le \frac{5\eta_0(\|\rho_{fc}(x)\|_\infty+\alpha_2)}{(L_+-s)^b/C}\le U_1 N^{-\frac{1}{4}+\epsilon+\zeta b}
	\end{align}
		where  $U_1>0$ is a constant.
		\item[{\bf Case 2.}] Suppose $i\in[\frac{2C}{1+b}N^{1-\zeta(b+1)},\frac{N}{2}]$ and $\lambda_i>\gamma_i$. By \eqref{eq7},  \eqref{eq9}, \eqref{eq106} and the definition of $\zeta$, if $N$ is large enough then
		\begin{multline}\label{eq11}
			g(\lambda_i)\ge g(\gamma_i)-5\eta_0(\|\rho_{fc}(x)\|_\infty+\alpha_2)\ge \int_{\gamma_i}^{L_+}C^{-1}(L_+-x)^bdx-5\eta_0(\|\rho_{fc}(x)\|_\infty+\alpha_2)\\
			=\frac{C^{-1} (L_+-\gamma_i)^{b+1}}{b+1}-5\eta_0(\|\rho_{fc}(x)\|_\infty+\alpha_2)\ge \frac{C^{-1}}{2(b+1)}N^{-\zeta(b+1)}
		\end{multline}
		which implies $\lambda_i<L_+$ (otherwise $g(\lambda_i)=0$) and thus
		\begin{align}\label{eq10}
			g(\lambda_i)\le\int_{\lambda_i}^{L_+}C(L_+-x)^bdx\le \frac{C}{1+b}(L_+-\lambda_i)^{b+1}.
		\end{align}
		By \eqref{eq11} and \eqref{eq10}, if $N$ is large enough then
		$$L_+-\lambda_i\ge(2C^2)^{\frac{-1}{1+b}}N^{-\zeta},$$
		so
\begin{align}\label{eq108}
\vert \gamma_i-\lambda_i\vert =\frac{\vert g(\gamma_i)-g(\lambda_i)\vert }{\vert g'(t)\vert }\le\frac{5\eta_0(\|\rho_{fc}(x)\|_\infty+\alpha_2)}{\rho_{fc}(t)}\le \frac{5\eta_0(\|\rho_{fc}(x)\|_\infty+\alpha_2)}{(L_+-\lambda_i)^b/C}\le U_2 N^{-\frac{1}{4}+\epsilon+\zeta b}.
\end{align}
		 Here $t\in(\gamma_i,\lambda_i)$ and $U_2>0$ is a constant.
	\end{enumerate}
\eqref{eq107}, \eqref{eq108} and the fact that $\zeta$ can be arbitrarily small complete the proof of  \eqref{eq103}. \eqref{eq104} can be proved in the same way.
\end{proof}
 
\section{SSK model in low temperature: proof of Theorem \ref{thm:low_temperature}}\label{sec:low_temperature}
In this section we follow the idea introduced in \cite{Baik+Lee} to prove Theorem \ref{thm:low_temperature}. Because of the results in Lemma \ref{lemma:steepest_descent_curve}, we know that if a particle is moving along the curve of steepest-descent  defined in Definition \ref{definition:steepest_descent_curve}, then its $y$-coordinate is monotone, therefore we do not need a lemma like Lemma 6.4 in \cite{Baik+Lee}.

Throughout this section we suppose the conditions of Theorem \ref{thm:low_temperature} hold.
\begin{definition}
 Suppose $\epsilon_0>0$ is a constant. Let $s_0=s_0(\epsilon_0)>0$ be a constant  such that 
\begin{align}\label{eq110}
	{\mathbb P}\big(\vert \lambda_1-L_+\vert < s_0 N^{-\frac{1}{1+b}}\text{ and }\vert \lambda_N-L_-\vert  <s_0 N^{-\frac{1}{1+a}}\big)>1-\epsilon_0.
\end{align}
for large enough $N$. Set	$$\Omega_N^0(\epsilon_0)=\big\{\vert \lambda_1-L_+\vert < s_0N^{-\frac{1}{1+b}}\text{ and }\vert \lambda_N-L_-\vert  <s_0 N^{-\frac{1}{1+a}}\big\}.$$
\end{definition}
\begin{remark}
By Theorem \ref{thm:fluctuation_of_lambda_1} the constant $s_0$ exists. 	
\end{remark}

\begin{definition}\label{definition:R(z)}
	 Let $R(z)$ be an analytic function defined on ${\mathbb C }\backslash(-\infty,\lambda_1]$   by
	$$R(z)=2\beta z-\frac{1}{N}\sum_{i=1}^N\log(z-\lambda_i).$$
	Here we take the analytic branch of the $\log$ function such that $\Im \log(z-\lambda_i)\in(-\pi,\pi)$ for all $z\in{\mathbb C }\backslash(-\infty,\lambda_1]$.
	 Let $\gamma$  denote the unique number in $(\lambda_1,+\infty)$ such that $R'(\gamma)=0$. Equivalently, $\gamma$ is the unique number on $(\lambda_1,+\infty)$ satisfying $2\beta=\frac{1}{N}\sum_{i=1}^N\frac{1}{\gamma-\lambda_i}$.
\end{definition}

\begin{lemma}\label{lemma:position_of_gamma}
 Suppose	\begin{itemize}
	\item $\epsilon$ is a constant in $\Big(\frac{1}{1+b},\min\Big(\frac{1}{4}-\frac{2}{b+1},\frac{1}{4}-\frac{1}{b+1}-\frac{a}{(b+1)^2}\Big)\Big)$;
	\item $\tau$ is a constant in $\Big(\frac{1}{(b+1)^2},\min\Big(\frac{\frac{1}{4}-\epsilon}{a+1},\frac{\frac{1}{4}-\epsilon-\frac{1}{b+1}}{a},\frac{\frac{1}{4}-\epsilon-\frac{b+2}{(b+1)^2}}{b}\Big)\Big)$;
	\item $\tau_1<\tau_0$ are two constants both in $(1-\tau(b+1),1-\frac{1}{b+1})$.
\end{itemize} 
Suppose $E_N(\epsilon)\cap\Omega_N^0(\epsilon_0)$ holds. There exists a constant $N_0>0$ such that if $N>N_0$, then 
	$$\lambda_1+\frac{1}{3\beta N}<\gamma<\lambda_1+N^{-1+\tau_0}.$$
\end{lemma}
\begin{remark}
According to the conditions $b>11$ and $1<a<\frac{b^2-6b-7}{4}$, it is easy to check that the constants $\epsilon$, $\tau$, $\tau_0$ and $\tau_1$ exist. The event $E_N(\epsilon)$ is defined in Theorem \ref{thm:rigidity_concise_version}.
\end{remark}
\begin{proof}
	Notice that $R'(x)=2\beta-\frac{1}{N}\sum_{i=1}^N\frac{1}{x-\lambda_i}$ is increasing on $(\lambda_1,+\infty)$. Since $$R'(\lambda_1+\frac{1}{3\beta N})<2\beta-\frac{1}{N}\frac{1}{(\lambda_1+\frac{1}{3\beta N})-\lambda_1}<0=R'(\gamma),$$ 
	we have that $\lambda_1+\frac{1}{3\beta N}<\gamma$.

	Suppose that $E_N(\epsilon)\cap\Omega_N^0(\epsilon_0)$ holds.
	Since $\frac{\frac{1}{4}-\epsilon-\frac{b+2}{(1+b)^2}}{b}<\frac{\frac{1}{4}-\epsilon}{b+1}$, the $\tau$ satisfies the conditions for $\zeta$ and $\zeta'$ in Theorem \ref{thm:rigidity_concise_version}.
	According to Theorem \ref{thm:rigidity_concise_version}, when $N$ is large enough, we have
	\begin{align}\label{eq15}
		\begin{cases}
			\vert \lambda_i-\gamma_i\vert \le N^{-\frac{1}{4}+\epsilon+\tau b}\quad&\text{if }i\in[\kappa'N^{1-\tau(b+1)},\frac{N}{2}]\\
			\vert \lambda_i-\gamma_i\vert \le N^{-\frac{1}{4}+\epsilon+\tau a}\quad&\text{if }i\in[\frac{N}{2},N-\kappa'N^{1-\tau(a+1)}]
		\end{cases}
	\end{align}
	where $\kappa'>0$ is defined in Theorem \ref{thm:rigidity_concise_version}.

	To prove $\gamma<\lambda_1+N^{-1+\tau_0}$ we need
	$$\vert \frac{1}{N}\sum_{i=1}^N\frac{1}{\lambda_1+N^{-1+\tau_0}-\lambda_i}-\int \frac{d\mu_{fc}(t)}{L_+-t}\vert \le I+II+III$$
	where
	$$I=\Big\vert \frac{1}{N}\sum\limits_{i=N^{\tau_1}}^{N-\kappa'N^{1-\tau(a+1)}}
	\frac{1}{\lambda_1+N^{-1+\tau_0}-\lambda_i}-\frac{1}{N}\sum\limits_{i=N^{\tau_1}}^{N-\kappa'N^{1-\tau(a+1)}}
	\frac{1}{L_+-\gamma_i}\Big\vert $$
	$$II=\Big\vert \frac{1}{N}\sum\limits_{i=N^{\tau_1}}^{N-\kappa'N^{1-\tau(a+1)}}
	\frac{1}{L_+-\gamma_i}-\int \frac{d\mu_{fc}(t)}{L_+-t}\Big\vert $$ 
	$$III=\Big\vert \frac{1}{N}\sum_{i<N^{\tau_1}}\frac{1}{\lambda_1+N^{-1+\tau_0}-\lambda_i}+\frac{1}{N}\sum_{i>N-\kappa'N^{1-\tau(a+1)}}\frac{1}{\lambda_1+N^{-1+\tau_0}-\lambda_i}\Big\vert $$
	
 {\bf Estimation of I. }If $N$ is large enough, then for any $i\in[N^{\tau_1},N/2]\subset[\kappa'N^{1-\tau(b+1)},\frac{N}{2}]$,
		\begin{multline}\label{eq17}
			\Big\vert \frac{1}{\lambda_1+N^{-1+\tau_0}-\lambda_i}-\frac{1}{L_+-\gamma_i}\Big\vert \vert L_+-\gamma_i\vert \le\dfrac{\vert \lambda_1-L_+\vert +\vert \gamma_i-\lambda_i\vert +N^{-1+\tau_0}}{\vert L_+-\gamma_i\vert -\vert L_+-\lambda_1\vert -\vert \lambda_i-\gamma_i\vert -N^{-1+\tau_0}}\\
			\le\dfrac{s_0N^{-\frac{1}{1+b}}+N^{-\frac{1}{4}+\epsilon+\tau b}+N^{-1+\tau_0}}{C_*^{-1}(i/N)^{\frac{1}{1+b}}-s_0N^{-\frac{1}{1+b}}-N^{-\frac{1}{4}+\epsilon+\tau b}-N^{-1+\tau_0}}=\dfrac{s_0+N^{\frac{1}{1+b}-\frac{1}{4}+\epsilon+\tau b}+N^{\frac{1}{1+b}-1+\tau_0}}{C_*^{-1}\cdot i^{\frac{1}{1+b}}-s_0-N^{\frac{1}{1+b}-\frac{1}{4}+\epsilon+\tau b}-N^{\frac{1}{1+b}-1+\tau_0}}\\
			\le\dfrac{s_0+N^{\frac{1}{1+b}-\frac{1}{4}+\epsilon+\tau b}+N^{\frac{1}{1+b}-1+\tau_0}}{C_*^{-1}\cdot N^{\frac{\tau_1}{1+b}}-s_0-N^{\frac{1}{1+b}-\frac{1}{4}+\epsilon+\tau b}-N^{\frac{1}{1+b}-1+\tau_0}}
			\le2s_0C_*N^{-\frac{\tau_1}{1+b}}.
		\end{multline} 
		Here we used Lemma \ref{lemma:distance_between_gamma_i_and_edge},  \eqref{eq15} and the definition of $\Omega_N^0(\epsilon_0)$ in the second inequality and we used the conditions on $\tau$ and $\tau_0$   in the last inequality. (In particular, the condition $\tau<\frac{\frac{1}{4}-\epsilon-\frac{b+2}{(b+1)^2}}{b}$ implies $\frac{1}{1+b}-\frac{1}{4}+\epsilon+\tau b<0$.) The constant $C_*>0$ is defined in Lemma \ref{lemma:distance_between_gamma_i_and_edge}.
		
		By a similar argument we can prove \eqref{eq17} for $i\in[N/2,N-\kappa'N^{1-\tau(a+1)}]$. So we have that if $N$ is large enough, then
		
		\begin{multline*}
			I
			\le2s_0C_*N^{-\frac{\tau_1}{1+b}}\cdot\frac{1}{N}\sum\limits_{i=N^{\tau_1}}^{N-\kappa'N^{1-\tau(a+1)}}
			\frac{1}{L_+-\gamma_i}\\
			\le 2s_0C_*N^{-\frac{\tau_1}{1+b}}\cdot\sum\limits_{i=N^{\tau_1}}^{N-\kappa'N^{1-\tau(a+1)}}\int_{\hat\gamma_{i-1}}^{\hat\gamma_{i-2}}\frac{1}{L_+-t}d\mu_{fc}(t)\le 4\beta_c\cdot s_0C_*N^{-\frac{\tau_1}{1+b}}
		\end{multline*}
		where we used \eqref{eq18} in the second inequality and used the definition of $\beta_c$ in the last inequality. 
		
{\bf  Estimation of II. } If $N$ is large enough, then by\eqref{eq18}, \eqref{eq12} and Lemma \ref{lemma:distance_between_gamma_i_and_edge},
		\begin{align*}
			II\le& \int_{L_-}^{\hat\gamma_{N-\kappa'N^{1-\tau(a+1)}-2}}\frac{d\mu_{fc}(t)}{L_+-t}+\int_{\hat\gamma_{N^{\tau_1}}}^{L_+}\frac{d\mu_{fc}(t)}{L_+-t}
			\le C\vert \hat\gamma_{N-\kappa'N^{1-\tau(a+1)}-2}-L_-\vert +\int_{\hat\gamma_{N^{\tau_1}}}^{L_+}C_0(L_+-t)^{b-1}dt\\
			=&C\vert \hat\gamma_{N-\kappa'N^{1-\tau(a+1)}-2}-L_-\vert +\frac{C_0}{b}\vert L_+-\hat\gamma_{N^{\tau_1}}\vert ^b
			\le C\vert \gamma_{N-\kappa'N^{1-\tau(a+1)}-2}-L_-\vert +\frac{C_0}{b}\vert L_+-\gamma_{N^{\tau_1}+1}\vert ^b\\
			\le&C\Big(\frac{\kappa'N^{1-\tau(a+1)}+2}{N}\Big)^{\frac{1}{1+a}}+C\Big(\frac{N^{\tau_1}+1}{N}\Big)^{\frac{b}{1+b}}
		\end{align*}
		where $C_0$  is defined in \eqref{eq12} and  $C>0$ is a constant.

{\bf  Estimation of III. } Since $\lambda_1\ge\cdots\ge\lambda_N$, we have for large enough $N$:
		$$0<\frac{1}{\lambda_1+N^{-1+\tau_0}-\lambda_i}\le N^{1-\tau_0}\quad(1\le i\le N/2)$$
\begin{multline*}
0<\frac{1}{\lambda_1+N^{-1+\tau_0}-\lambda_i}\le\frac{1}{\lambda_1+N^{-1+\tau_0}-\lambda_{\lfloor N/2\rfloor}}\\
\le \frac{1}{\vert L_+-\gamma_{\lfloor N/2\rfloor}\vert -\vert \lambda_{\lfloor N/2\rfloor}-\gamma_{\lfloor N/2\rfloor}\vert -\vert L_+-\lambda_1\vert -N^{-1+\tau_0}}\le\frac{2}{\vert L_+-\gamma_{\lfloor N/2\rfloor}\vert } \quad(N/2\le i\le N)
\end{multline*}	
		so
		$$III\le N^{\tau_1-\tau_0}+\kappa'N^{-\tau(a+1)}\cdot \frac{2}{\vert L_+-\gamma_{\lfloor N/2\rfloor}\vert } .$$
	
	Since $R'(\lambda_1+N^{-1+\tau_0})=2\beta-\frac{1}{N}\sum_{i=1}^N\frac{1}{\lambda_1+N^{-1+\tau_0}-\lambda_i}$, we know from the estimations of $I$, $II$ and $III$ that if $N$ is large enough, then
	\begin{multline}
		R'(\lambda_1+N^{-1+\tau_0})\ge2\beta-\int\frac{d\mu_{fc}(t)}{L_+-t}-I-II-III\\
		\ge 2\beta-2\beta_c-4\beta_cs_0C_*N^{-\frac{\tau_1}{1+b}}-\Big(C\Big(\frac{\kappa'N^{1-\tau(a+1)}+2}{N}\Big)^{\frac{1}{1+a}}+C\Big(\frac{N^{\tau_1}+1}{N}\Big)^{\frac{b}{1+b}}\Big)\\
		-\Big(N^{\tau_1-\tau_0}+\kappa'N^{-\tau(a+1)}\cdot \frac{2}{\vert L_+-\gamma_{\lfloor N/2\rfloor}\vert }\Big)>0=R'(\gamma)
	\end{multline}
	therefore $\gamma<\lambda_1+N^{-1+\tau_0}$ because $R'$ is increasing on $(\lambda_1,+\infty)$.  
\end{proof}
\begin{lemma}\label{lemma:analogue_of_lemma_6.2}
	Suppose the assumptions in Lemma \ref{lemma:position_of_gamma} hold. Suppose   $\tau_2$, $\tau_3$ and $\tau'$ are constants satisfying:
	\begin{itemize}
		
		\item $\max(1-\tau(1+b),1-\tau(1+a))<\tau_2<1$
		\item $
		\max\Big(2+(1+b)(-\frac{1}{4}+\epsilon+\tau b),1-\tau(b+1),1-\frac{a+1}{a}(\frac{1}{4}-\epsilon-\frac{1}{b+1})\Big)
		<\tau_3<\frac{b}{b+1}$
		\item $\frac{1-\tau_3}{a+1}<\tau'<\min(\frac{\frac{1}{4}-\epsilon-\frac{1}{b+1}}{a},\frac{\frac{1}{4}-\epsilon}{1+a}) $
	\end{itemize}
Then we have the following conclusions.
\begin{itemize}
	\item 
	If $N$ is large enough, then
	\begin{align*}
		\Big\vert \frac{1}{N}\sum_{i=1}^N\log(\gamma-\lambda_i)-\int\log(L_+-t)d\mu_{fc}(t)-2\beta_c(\gamma-L_+)\Big\vert \le W_1\Phi_N
	\end{align*}
	where $W_1>0$ is a constant and
	$$\Phi_N=N^{-2+2\tau_0-2\frac{\tau_3-1}{1+b}}+N^{\frac{-2\tau_3}{b+1}}+N^{-\frac{1}{4}+\epsilon+\tau b+\frac{1-\tau_3}{b+1}}+N^{-\frac{1}{4}+\epsilon+\tau'a}+N^{\frac{b(\tau_3-1)}{b+1}}\Big(N^{-\frac{1}{1+b}}+N^{-1+\tau_0}\Big)+N^{\tau_3-1}\log N.$$
\item	If
	$N$ is large enough, then
	$$ N^{l-l\tau_0-1} \le \frac{R^{(l)}(\gamma)(-1)^l}{(l-1)!}\le W_2^lN^{-1+\tau_2+l},\quad l=2,3,\ldots $$
	where $W_2>0$ is a constant.
\end{itemize}
\end{lemma}

\begin{remark}
\begin{itemize}
	\item From the definition of $\epsilon$ and  $\tau$ (see Lemma \ref{lemma:position_of_gamma}), we see that the $\tau_2$ and $\tau_3$ satisfying the conditions exist. Since $\tau<\frac{\frac{1}{4}-\epsilon-\frac{b+2}{(b+1)^2}}{b}$, we have from the definition of $\tau_3$: 
	\begin{align}\label{eq111}
	\tau_3>1-\tau(b+1)>\frac{3}{4}+\epsilon
	\end{align}
\eqref{eq111} and the definition of $\tau_3$ yields  $\frac{1-\tau_3}{a+1}<\min(\frac{\frac{1}{4}-\epsilon-\frac{1}{b+1}}{a},\frac{\frac{1}{4}-\epsilon}{1+a})$, thus $\tau'$ is well defined.
	\item By \eqref{eq111} and the definitions of $\tau_0$ and $\tau_3$, we have
\begin{align}\label{eq114}
\tau_3>1-\tau(1+b)>\frac{3}{4}>(\tau_0-1)(b+1)+\frac{3}{2}
\end{align}
 so
\begin{align}\label{eq112}
\frac{1}{b+1}-2+2\tau_0-2\frac{\tau_3-1}{1+b}<0.
\end{align}
By $\tau_0<\frac{b}{b+1}$ and the definition of $\tau_3$ we have $\tau_3<\frac{b}{b+1}<2-\frac{b+1}{b}\tau_0$, so
\begin{align}\label{eq113} 
\frac{1}{b+1}+\frac{b(\tau_3-1)}{b+1}-1+\tau_0<0
\end{align}	
By \eqref{eq112}, \eqref{eq113} and the definition of $\tau_3$ we see that
	\begin{align}\label{eq44}
		\lim_{N\to\infty}N^{\frac{1}{1+b}}\cdot\Phi_N=0.
	\end{align}
\end{itemize}	
\end{remark} 
\begin{proof}

	By Lemma \ref{lemma:distance_between_gamma_i_and_edge}, Theorem \ref{thm:rigidity_concise_version},  Lemma \ref{lemma:position_of_gamma} and the definition of $\Omega_N^0(\epsilon_0)$, if $N$ is large enough and $i\in[1+N^{\tau_3},N-N^{\tau_3}]$, then
	\begin{align}\label{eq20}
		\Big\vert \frac{\gamma-L_+}{L_+-\gamma_i}\Big\vert \le\Big\vert \frac{\gamma-\lambda_1}{L_+-\gamma_i}\Big\vert +\Big\vert \frac{\lambda_1-L_+}{L_+-\gamma_i}\Big\vert \le \frac{N^{-1+\tau_0}+s_0N^{-\frac{1}{1+b}}}{C_*^{-1}N^{\frac{\tau_3-1}{b+1}}}\stackrel{N\to\infty}{\longrightarrow}0\quad\text{ (by \eqref{eq114})}
	\end{align}
	\begin{align}\label{eq21} 
		\Big\vert \frac{\gamma_i-\lambda_i}{L_+-\gamma_i}\Big\vert \le\frac{N^{-\frac{1}{4}+\epsilon+\tau b}}{C_*^{-1}N^{\frac{\tau_3-1}{b+1}}}\mathds1_{i\le N/2}+\frac{N^{-\frac{1}{4}+\epsilon+\tau'a}}{\vert L_+-\gamma_{N/2}\vert }\mathds1_{i\ge N/2}\stackrel{N\to\infty}{\longrightarrow}0\quad\text{ (by definitions of $\tau_3$ and $\tau'$)}
	\end{align}
	and thus
	\begin{align}\label{eq22} 
		\log(\gamma-\lambda_i)-\log (L_+-\gamma_i)=\log(1+\frac{\gamma-L_+}{L_+-\gamma_i}+\frac{\gamma_i-\lambda_i}{L_+-\gamma_i})=\frac{\gamma-L_+}{L_+-\gamma_i}+B_1+B_2
	\end{align}
	where $\vert B_1\vert \le2\Big\vert \frac{\gamma_i-\lambda_i}{L_+-\gamma_i}\Big\vert $ and $\vert B_2\vert \le \Big\vert \frac{\gamma-L_+}{L_+-\gamma_i}\Big\vert ^2$. 
	
	By \eqref{eq18}, Lemma \ref{lemma:properties_of_mu_fc} and Lemma \ref{lemma:distance_between_gamma_i_and_edge}, if $N$ is large enough, then
	\begin{multline}\label{eq23} 
		\Big\vert \frac{1}{N}\sum\limits_{i=N^{\tau_3}+1}^{N-N^{\tau_3}}\frac{1}{L_+-\gamma_i}-2\beta_c\Big\vert \le \Big\vert \int_{\hat\gamma_{N-N^{\tau_3}+1}}^{\hat\gamma_{N^{\tau_3}+2}}\frac{d\mu_{fc}(t)}{L_+-t}-2\beta_c\Big\vert =\int_{L_-}^{\hat\gamma_{N-N^{\tau_3}+1}}\frac{d\mu_{fc}(t)}{L_+-t}+\int_{\hat\gamma_{N^{\tau_3}+2}}^{L_+}\frac{d\mu_{fc}(t)}{L_+-t}\\
		\le W_3\int_{L_-}^{\hat\gamma_{N-N^{\tau_3}+1}}(t-L_-)^adt+C_0\int_{\hat\gamma_{N^{\tau_3}+2}}^{L_+}(L_+-t)^{b-1}dt\le W_3\vert \gamma_{N-N^{\tau_3}+1}-L_-\vert ^{a+1}+\frac{C_0}{b}\vert L_+-\hat\gamma_{N^{\tau_3}+2}\vert ^b\\
		\le W_3\cdot\frac{N^{\tau_3}}{N}+ \frac{C_0}{b}\vert L_+-\gamma_{N^{\tau_3}+2}\vert ^b\le W_3\frac{N^{\tau_3}}{N}+ \frac{C_0}{b}\cdot C_*^b\Big(\frac{N^{\tau_3}+3}{N}\Big)^{\frac{b}{1+b}}\\
		\le W_3N^{\frac{b(\tau_3-1)}{b+1}}
	\end{multline}
	where $W_3>0$ is a constant, $C_0>0$ is defined in Lemma       \ref{lemma:properties_of_mu_fc} and $C_*$ is defined in Lemma         \ref{lemma:distance_between_gamma_i_and_edge}.
	
	According to Lemma \ref{lemma:properties_of_mu_fc}, Lemma \ref{lemma:distance_between_gamma_i_and_edge}, the definitions of $\gamma_i$ and  $\hat\gamma_i$ and the fact that
$$\int_{\hat\gamma_{i-1}}^{\hat\gamma_{i-2}}\log(L_+-t)d\mu_{fc}(t)\le \frac{1}{N}\log(L_+-\gamma_i)\le \int_{\hat\gamma_{i+1}}^{\hat\gamma_i}\log(L_+-t)d\mu_{fc}(t)\quad(2\le i\le N-1)$$
	we know that if $N$ is large enough, then
	\begin{multline}\label{eq24}
		\Big\vert \frac{1}{N}\sum\limits_{i=N^{\tau_3}+1}^{N-N^{\tau_3}}\log(L_+-\gamma_i)-\int\log(L_+-t)d\mu_{fc}(t)\Big\vert \\
		\le\int_{L_-}^{\hat\gamma_{N-N^{\tau_3}-2}}\vert \log(L_+-t)\vert d\mu_{fc}(t)+\int_{\hat\gamma_{N^{\tau_3}+2}}^{L_+}\vert \log(L_+-t)\vert d\mu_{fc}(t)\\
		\le W_4\int_{L_-}^{\gamma_{N-N^{\tau_3}-2}}(t-L_-)^adt+C_0\int_{ \gamma_{N^{\tau_3}+3}}^{L_+}(L_+-t)^b\vert \log(L_+-t)\vert dt\\
		\le W_4N^{\tau_3-1}+C_0\frac{\vert L_+-\gamma_{N^{\tau_3}+3}\vert ^{b+1}}{b+1}\Big\vert \log\big(L_+-\gamma_{N^{\tau_3}+3}\big)-\frac{1}{b+1}\Big\vert \\
		\le W_4N^{\tau_3-1}\log N
	\end{multline}
	where $C_0>0$ is defined in Lemma \ref{lemma:properties_of_mu_fc} and $W_4>0$ is a constant.	
	
	By Lemma \ref{lemma:position_of_gamma} and the definition of $\Omega_N^0(\epsilon_0)$, if $N$ is large enough, then
	\begin{align}\label{eq25} 
		\Big\vert \frac{1}{N}\sum_{i\le 1+N^{\tau_3}}\log(\gamma-\lambda_i)+\frac{1}{N}\sum_{i\ge N-N^{\tau_3}}\log(\gamma-\lambda_i)\Big\vert \le 4N^{\tau_3-1}\log N.
	\end{align}
	According to \eqref{eq20}, \eqref{eq21}, \eqref{eq22}, \eqref{eq23}, \eqref{eq24}, \eqref{eq25}, if $N$ is large enough, then 
	\begin{multline*}
		\Big\vert \frac{1}{N}\sum_{i=1}^N\log(\gamma-\lambda_i)-\int\log(L_+-t)d\mu_{fc}(t)-2\beta_c(\gamma-L_+)\Big\vert \\
		\le\Big\vert \frac{1}{N}\sum\limits_{i=1+N^{\tau_3}}^{N-N^{\tau_3}}\log(\gamma-\lambda_i)-\frac{1}{N}\sum\limits_{i=1+N^{\tau_3}}^{N-N^{\tau_3}}\log(L_+-\gamma_i)-\frac{\gamma-L_+}{N}\sum\limits_{i=1+N^{\tau_3}}^{N-N^{\tau_3}}\frac{1}{L_+-\gamma_i}\Big\vert \\
		+\Big\vert \frac{\gamma-L_+}{N}\sum\limits_{i=1+N^{\tau_3}}^{N-N^{\tau_3}}\frac{1}{L_+-\gamma_i}-2\beta_c(\gamma-L_+)\Big\vert +\Big\vert \frac{1}{N}\sum\limits_{i=1+N^{\tau_3}}^{N-N^{\tau_3}}\log(L_+-\gamma_i)-\int\log(L_+-t)d\mu_{fc}(t)\Big\vert \\
		+\Big\vert \frac{1}{N}\sum_{i\le 1+N^{\tau_3}}\log(\gamma-\lambda_i)+\frac{1}{N}\sum_{i\ge N-N^{\tau_3}}\log(\gamma-\lambda_i)\Big\vert \\
		\le\Big(\frac{N^{-1+\tau_0}+s_0N^{-\frac{1}{1+b}}}{C_*^{-1}N^{\frac{\tau_3-1}{b+1}}}\Big)^2+2\frac{N^{-\frac{1}{4}+\epsilon+\tau b}}{C_*^{-1}N^{\frac{\tau_3-1}{b+1}}}+2\frac{N^{-\frac{1}{4}+\epsilon+\tau'a}}{\vert L_+-\gamma_{N/2}\vert }+W_3N^{\frac{b(\tau_3-1)}{b+1}}\vert \gamma-L_+\vert +W_4N^{\tau_3-1}\log N+4N^{\tau_3-1}\log N.
	\end{multline*}
	The above inequality together with $$\vert \gamma-L_+\vert \le\vert \gamma-\lambda_1\vert +\vert \lambda_1-L_+\vert <N^{-1+\tau_0}+s_0N^{-\frac{1}{b+1}}$$ and the fact that $\vert L_+-\gamma_{N/2}\vert $ is bounded below by a constant completes the proof of the first conclusion of the lemma.

	For the second conclusion, notice that
	$$R^{(l)}(z)=\frac{1}{N}\sum_{i=1}^N\frac{(-1)^l(l-1)!}{(z-\lambda_i)^l},\quad l=2,3,\ldots,$$
	so for large enough $N$ we have by Lemma \ref{lemma:position_of_gamma} that
	\begin{align}\label{eq29} 
		\frac{R^{(l)}(\gamma)(-1)^l}{(l-1)!}\ge\frac{1}{N}\frac{1}{(\gamma-\lambda_1)^l}\ge N^{l-l\tau_0-1},\quad l=2,3,\ldots
	\end{align}

	From the conditions on $\tau_2$ and $\tau$ we see
	\begin{align}\label{eq26}
		1>\tau_2>1-\tau(1+b)>1+(b+1)(-\frac{1}{4}+\epsilon+\tau b).
	\end{align}
	By Lemma \ref{lemma:position_of_gamma}, if $N$ is large enough, then 
	\begin{align}\label{eq27}
		\frac{1}{N}\Big(\sum_{i=1}^{N^{\tau_2}}+\sum_{i=N-N^{\tau_2}}^N\Big)\frac{1}{(\gamma-\lambda_i)^l}\le 3N^{-1+\tau_2}\cdot (3\beta N)^l=3(3\beta)^l N^{-1+\tau_2+l},\quad l=2,3,\ldots
	\end{align}
	
	By \eqref{eq26} and Lemma \ref{lemma:distance_between_gamma_i_and_edge} if $N$ is large enough and $i\in[N^{\tau_2},N-N^{\tau_2}]$, then $\vert L_+-\gamma_i\vert \ge C_*^{-1}N^{\frac{\tau_2-1}{1+b}}$, $$\vert L_+-\gamma\vert \le\vert L_+-\lambda_1\vert +\vert \gamma-\lambda_1\vert \le s_0N^{-\frac{1}{b+1}}+N^{-1+\tau_0}\le (1+s_0)N^{-\frac{1}{b+1}}\quad\text{(since $\tau_0<\frac{b}{b+1}$)}$$
	\begin{align}
		\text{and}\quad\vert \lambda_i-\gamma_i\vert \le
		\begin{cases}
			N^{-\frac{1}{4}+\epsilon+\tau b}\quad\text{if }N^{\tau_2}\le i\le N/2\\
			N^{-\frac{1}{4}+\epsilon+\tau a}\quad\text{if }  N/2\le i\le N-N^{\tau_2}
		\end{cases}
	\end{align}
The above estimations and the fact that $\vert L_+-\gamma_j\vert $ is of order 1 for $j\ge N/2$ imply that if $N$ is large enough then
	\begin{align*}
		\vert \gamma-\lambda_i\vert \ge \vert L_+-\gamma_i\vert -\vert \gamma_i-\lambda_i\vert -\vert L_+-\gamma\vert 
		\ge\frac{1}{2}\vert L_+-\gamma_i\vert \ge\frac{1}{2C_*}\Big(\frac{i}{N}\Big)^{\frac{1}{1+b}},\quad\text{if }N^{\tau_2}\le i\le N-N^{\tau_2}
	\end{align*}
	and thus
	\begin{multline}\label{eq28}
		\frac{1}{N}\sum_{N^{\tau_2}<i<N-N^{\tau_2}}\frac{1}{(\gamma-\lambda_i)^l}\le\frac{1}{N}\sum_{N^{\tau_2}<i<N-N^{\tau_2}}\frac{(2C_*)^lN^{\frac{l}{1+b}}}{i^{l/(b+1)}}
		\le (2C_*)^lN^{\frac{l}{1+b}-1}\int_{\frac{1}{2}N^{\tau_2}}^Nx^{-\frac{l}{1+b}}dx\\
		\le \begin{cases}
			(2C_*)^l\log N\quad&\text{if $b$ is an integer and }l=b+1\\\frac{b+1}{b+1-l}(2C_*)^l\quad&\text{if }l<b+1\\
			\frac{b+1}{l-(b+1)}(4C_*)^lN^{(1-\tau_2)(\frac{l}{b+1}-1)}\quad&\text{if }l>b+1
		\end{cases}
	\end{multline}
	where $C_*$ is defined in Lemma \ref{lemma:distance_between_gamma_i_and_edge}.  In \eqref{eq28}, the coefficient $\pm\frac{b+1}{b+1-l}$ for the case $l\ne b+1$ is bounded by:
	\begin{align*} 
		\vert \frac{b+1}{b+1-l}\vert \le 
		\begin{cases}
			b+1\quad&\text{if }b\in\mathbb Z\\
			\frac{b+1}{\text{dist}(b,\mathbb Z)}\quad&\text{if }b\not\in\mathbb Z
		\end{cases}
	\end{align*}
	and this bound is independent of $l$.
	So by the fact $-1+\tau_2+l>(1-\tau_2)(\frac{l}{b+1}-1)$ and \eqref{eq29}, \eqref{eq27} and \eqref{eq28} we complete the proof.
\end{proof}

\begin{definition}\label{definition:steepest_descent_curve}
	\begin{itemize}
\item Set $S=\{x+\i y\in{\mathbb C }\backslash(-\infty,\lambda_1]\vert \Im R(x+\i y)=0\}$.
\item Set $S^+=\{x+\i y\in S\vert y>0\}$, $S^-=\{x+\i y\in S\vert y<0\}$.
\item For $y$ satisfying $0<\vert y\vert <\frac{\pi}{2\beta}$, let $h(y)$ be the unique real number such that $h(y)+\i y\in S$. Set $h(0)=\gamma$.
	\end{itemize}
\end{definition} 
\begin{lemma}\label{lemma:steepest_descent_curve}
\begin{itemize}
	\item $h(y)$ is well defined. In other words, for any $y$ satisfying $0<\vert y\vert <\frac{\pi}{2\beta}$, there is a  unique real number $x$ such that $x+\i y\in S$.
	\item $S=\{h(y)+\i y\vert -\frac{\pi}{2\beta}<y<\frac{\pi}{2\beta}\}$ and $h(y)\in C^1((-\frac{\pi}{2\beta},\frac{\pi}{2\beta}))$.
	\item $h(y)\le\gamma$ and the identity holds only when $y=0$.
	\item If $\frac{1}{4}<c_0<\frac{\pi}{2}$, then $h(y)$ is strictly decreasing on $[\frac{c_0}{\beta},\frac{\pi}{2\beta})$.
\end{itemize}
\end{lemma}
\begin{proof}
See Appendix \ref{appendix:steepest_descent_curve}.
\end{proof}
\begin{remark}
Since $h(y)$ is $C^1$, we can define the integral of continuous functions along $S$.
\end{remark}

\begin{lemma}\label{lemma:analogue_of_Lemma_6.3}
	Suppose the conditions in Lemma \ref{lemma:analogue_of_lemma_6.2} are satisfied. Set
	\begin{align}\label{eq42}
		K=-\i e^{-\frac{N}{2}R(\gamma)}\int_{\gamma-\i \infty}^{\gamma+\i \infty}e^{\frac{N}{2}R(z)}dz.
	\end{align}
	There exist  constants  $N_0>0$ and $W_0>0$ such that if $N>N_0$ then
	$$N^{-10}\le K\le W_0.$$
\end{lemma}

\begin{proof}
	
	From Lemma \ref{lemma:integral_formula} we know $K>0$. By the same argument as (6.31) of \cite{Baik+Lee},  
	\begin{align*}
		K=&-\i \int_{\gamma-\i \infty}^{\gamma+\i \infty}\exp\Big(\frac{N}{2}\Big[R(z)-R(\gamma)\Big]\Big)dz=\int_{-\infty}^{\infty}\exp\Big(\frac{N}{2}\Big[R(\gamma+\i t)-R(\gamma)\Big]\Big)dt\\
		=&\int_{-\infty}^{\infty}\exp\Big(\i\beta Nt-\frac{1}{2}\sum_{i=1}^N\log\Big(1+\frac{\i t}{\gamma-\lambda_i}\Big)\Big)dt\\
		\le&\int_{-\infty}^{\infty}\exp\Big(-\frac{1}{4}\sum_{i=1}^N\log\Big(1+\frac{t^2}{(\gamma-\lambda_i)^2}\Big)\Big)dt
	\end{align*}
	where we take the absolute value of the integrand to obtain the last inequality. Since 
	$$0<\gamma-\lambda_i\le\gamma-\lambda_N\le\vert \gamma-\lambda_1\vert +\vert \lambda_1-L_+\vert +\vert L_+-L_-\vert +\vert L_--\lambda_N\vert \le1+2s_0+L_+-L_-, $$
	we have for $N>4$:
	$$K\le \int_{-\infty}^{\infty}\exp\Big(-\frac{N}{4}\log\Big(1+\frac{t^2}{(1+2s_0+L_+-L_-)^2}\Big)\Big)dt\le\int_{{\mathbb R }}\Big(1+\frac{t^2}{(1+2s_0+L_+-L_-)^2}\Big)^{-1}dt<\infty$$
	and therefore the right hand side of the conclusion is proved.
	
	To prove the left hand side of the conclusion we need Lemma \ref{lemma:steepest_descent_curve}. We first claim that 
	\begin{align}\label{eq33}
		\int_{\gamma-\i \infty}^{\gamma+\i \infty}e^{\frac{N}{2}R(z)}dz=\int_Se^{\frac{N}{2}R(z)}dz
	\end{align}
	where the direction on $S$ is from $-\infty-\frac{\pi}{2\beta}\i$ to $-\infty+\frac{\pi}{2\beta}\i$. In fact, suppose $r>0$ such that $\vert z-\lambda_i\vert >r/2$ for all $\vert z\vert =r$, then for $C_r:=\{z\in{\mathbb C }\vert \vert z\vert =r, \Re z\le\gamma\}$ we have
	$$\Re R(z)\le 2\beta\gamma-\log(r/2),\quad\forall z\in C_r$$
	and thus
	$$\Big\vert \int_{C_r}\big\vert e^{\frac{N}{2}R(z)}\big\vert dz\Big\vert \le2\pi(\sqrt2 e^{\beta\gamma})^Nr^{1-\frac{N}{2}}\to0\quad\text{as }r\to\infty.$$
	Moreover, by the last conclusion of Lemma \ref{lemma:steepest_descent_curve} we have $\int_0^{\pi/(2\beta)}\vert e^{\frac{N}{2}R(h(y)+\i y)}\vert \sqrt{1+(h'(y))^2}dy<\infty$. So \eqref{eq33} is true.
	
	Notice that $R(z)=R(\bar z)$ for $z\in S$, so by \eqref{eq33}
	\begin{multline}\label{eq40}
		K=-\i\int_{S}e^{\frac{N}{2}(R(z)-R(\gamma))}dz=-\i\Big(\int_{S^+}e^{\frac{N}{2}(R(z)-R(\gamma))}(dx+\i dy)+\int_{S^-}e^{\frac{N}{2}(R(z)-R(\gamma))}(dx+\i dy)\Big)\\
		=2\int_{S^+}\exp\Big(\frac{N}{2}(R(z)-R(\gamma))\Big)dy
	\end{multline}
	
	Now we define
	$$Q_N=\big\{z\in{\mathbb C }\big\vert \vert z-\gamma\vert <N^{-9}\big\}.$$
	By Lemma \ref{lemma:position_of_gamma}, $R(z)$ is analytic on a neighborhood of $\bar Q_N$, so by $R'(\gamma)=0$ we have for large enough $N$:
	\begin{align}\label{eq38}
		R(z)-R(\gamma)=\sum_{j=2}^\infty\frac{R^{(j)}(\gamma)}{j!}(z-\gamma)^j,\quad\text{if } z\in Q_N .
	\end{align}
The next lemma shows that $S$ does not leave $\gamma$ too fast when the $y$ coordinate is small enough.
\begin{lemma}\label{lemma:regularity_of_S}
	When $N$ is large enough, we have
	\begin{align}\label{eq39}
	\{z\in S^+\vert \Im z\in(0,N^{-10})\}\subset Q_N.
\end{align} 
\end{lemma}
\begin{proof}
See Appendix \ref{appendix}.
\end{proof}
	By Lemma \ref{lemma:analogue_of_lemma_6.2}, \eqref{eq38} and the definition of $Q_N$, if $N$ is large enough, then:
	\begin{align} 
		\vert R(z)-R(\gamma)\vert \le \sum_{j=2}^\infty\frac{\vert R^{(j)}(\gamma)\vert }{j!}\vert z-\gamma\vert ^j\le 2W_2^2N^{-16},\quad\forall z\in Q_N
	\end{align}
	which together with \eqref{eq40} and \eqref{eq39} imply:
	\begin{align*}
		K=2\int_{S^+}\exp\Big(\frac{N}{2}(R(z)-R(\gamma))\Big)dy\ge2\int_0^{N^{-10}}\exp\Big(\frac{N}{2}\cdot(-2W_2^2N^{-16})\Big)dy>N^{-10}.
	\end{align*}
	So the proof is complete.
\end{proof}
\begin{proof}[Proof of Theorem \ref{thm:low_temperature}]
	Recall the definition of $K$ in \eqref{eq42}. According to Lemma \ref{lemma:integral_formula},
	$$\int_{S_{N-1}}e^{\beta\langle\sigma,(W+\lambda V)\sigma\rangle}d\omega_N(\sigma)=\frac{\sqrt N \beta}{\sqrt\pi(2\beta e)^{N/2}}\cdot e^{\frac{N}{2}R(\gamma)}\cdot K\cdot(1+O(N^{-1})).$$ 
	Now we choose the constants $s_0,\epsilon,\tau,\tau_0,,\tau_1,\tau_2,\tau_3,\tau'$ in the same way as in \eqref{eq110}, Lemma \ref{lemma:position_of_gamma} and Lemma \ref{lemma:analogue_of_lemma_6.2}. Suppose $E_N(\epsilon)\cap\Omega_N^0(\epsilon_0)$ holds. By Lemma \ref{lemma:analogue_of_Lemma_6.3}, if $N$ is large enough, then
	\begin{align*}
		F_N=\frac{1}{N}\log\Big(\int_{S_{N-1}}e^{\beta\langle\sigma,(W+\lambda V)\sigma\rangle}d\omega_N(\sigma)\Big)=\frac{R(\gamma)}{2}-\frac{1}{2}\log(2\beta e)+e_N
	\end{align*}
	where 
	\begin{align}\label{eq45} 
		\vert e_N\vert \le C\log N/N
	\end{align}
	for some constant $C>0$. By Lemma \ref{lemma:position_of_gamma} and Lemma \ref{lemma:analogue_of_lemma_6.2}, if $N$ is large enough and $E_N(\epsilon)\cap\Omega_N^0(\epsilon_0)$ holds, then
	\begin{align*}
		&\Big\vert F_N+\frac{1}{2}\log(2e\beta)-\beta\lambda_1+\frac{1}{2}\int\log(L_+-t)d\mu_{fc}(t)+\beta_c(\lambda_1-L_+)\Big\vert \le\vert e_N\vert +(\beta+\beta_c)\vert \gamma-\lambda_1\vert +W_1\Phi_N\\\le&\vert e_N\vert +(\beta+\beta_c)N^{-1+\tau_0}+W_1\Phi_N
	\end{align*}
	and thus
	\begin{align}\label{eq43}
		\Big\vert I_N-N^{\frac{1}{b+1}}(\lambda_1-L_+)\Big\vert 
		\le\frac{N^{\frac{1}{b+1}}}{\vert \beta-\beta_c\vert }\Big[\vert e_N\vert +(\beta+\beta_c)N^{-1+\tau_0}+W_1\Phi_N\Big]
	\end{align}
	where $$I_N=N^{\frac{1}{b+1}}\Big[\frac{F_N+\frac{1}{2}\log(2e\beta)+\frac{1}{2}\int\log(L_+-t)d\mu_{fc}(t)-\beta L_+}{\beta-\beta_c}\Big].$$
	Fix $s<0$. Choose $\epsilon_0\in(0,\vert s\vert /10)$. There exists $\delta_s>0$ such that if $\vert s'-s\vert \le\delta_s$ then 
	\begin{align}\label{eq46}
		\vert \exp(-\frac{C_\mu(-s')^{b+1}}{b+1})-\exp(-\frac{C_\mu(-s)^{b+1}}{b+1})\vert <\epsilon_0.
	\end{align}
Write $\mathcal E_N=N^{\frac{1}{b+1}}(\lambda_1-L_+)-I_N$.	Notice that
	\begin{multline}\label{eq47}
		{\mathbb P}(I_N\le s)={\mathbb P}(\{I_N\le s\}\cap(E_N(\epsilon)\cap\Omega_N^0(\epsilon_0))^c)+{\mathbb P}(\{I_N\le s\}\cap(E_N(\epsilon)\cap\Omega_N^0(\epsilon_0)))\\
		={\mathbb P}(\{I_N\le s\}\cap(E_N(\epsilon)\cap\Omega_N^0(\epsilon_0))^c)+{\mathbb P}(\{N^{\frac{1}{b+1}}(\lambda_1-L_+)\le s+\mathcal E_N\}\cap(E_N(\epsilon)\cap\Omega_N^0(\epsilon_0)))\\
		={\mathbb P}(N^{\frac{1}{b+1}}(\lambda_1-L_+)\le s+\mathcal E_N)-{\mathbb P}(\{N^{\frac{1}{b+1}}(\lambda_1-L_+)\le s+\mathcal E_N\}\cap(E_N(\epsilon)\cap\Omega_N^0(\epsilon_0))^c)\\
		+{\mathbb P}(\{I_N\le s\}\cap(E_N(\epsilon)\cap\Omega_N^0(\epsilon_0))^c).
	\end{multline}
If $N$ is large enough, then by \eqref{eq44}, \eqref{eq45}, \eqref{eq43} and the definition of $\tau_0$, we have $\mathcal E_N\in(-\delta_s,\delta_s)$ and then
	\begin{multline}\label{eq48}
		{\mathbb P}(N^{\frac{1}{b+1}}(\lambda_1-L_+)\le s+\mathcal E_N)\in\Big({\mathbb P}(N^{\frac{1}{b+1}}(\lambda_1-L_+)\le s-\delta_s),{\mathbb P}(N^{\frac{1}{b+1}}(\lambda_1-L_+)\le s+\delta_s)\Big)
	\end{multline}
	By \eqref{eq46}, \eqref{eq47}, \eqref{eq48}, \eqref{eq115} and \eqref{eq110} we have:
	$$\vert {\mathbb P}(I_N\le s)-{\mathbb P}(N^{\frac{1}{b+1}}(\lambda_1-L_+)\le s)\vert \le5\epsilon_0\quad\text{when $N$ is large enough}.$$
	Since $\epsilon_0$ can be arbitrarily small, we have by Theorem \ref{thm:fluctuation_of_lambda_1},
	$$\lim_{N\to\infty}{\mathbb P}(I_N\le s)=\lim_{N\to\infty}{\mathbb P}(N^{\frac{1}{b+1}}(\lambda_1-L_+)\le s)=1-{\mathbb P}(N^{\frac{1}{b+1}}(L_+-\lambda_1)\le -s)=\exp\big(-\frac{C_\mu(-s)^{b+1}}{b+1}\big).$$
	It is easy to check that the above identity is also true when $s=0$.
\end{proof} 
 
\section{SSK model in high temperature: proof of Theorem \ref{thm:high_temperature}}\label{sec:high_temperature}
In this section we use the method introduced in \cite{Baik+Lee} and Theorem \ref{thm:CLT} to prove Theorem \ref{thm:high_temperature}, but we follow a different way to control $\vert \gamma-\hat\gamma\vert $. In \cite{Baik+Lee}, the tool used to control $\vert \gamma-\hat\gamma\vert $ is the rigidity of eigenvalues, but we will use the local law to control $\vert \gamma-\hat\gamma\vert $ because the rigidity we have here is not strong enough to provide a proper estimation of $\vert \gamma-\hat\gamma\vert $.

Throughout this section we assume that the conditions of Theorem \ref{thm:high_temperature} are satisfied.
\begin{definition}
\begin{itemize}
	\item 
	Set
	$\hat R(z)=2\beta z-\int\log(z-t)d\mu_{fc}(t)$ analytically defined on ${\mathbb C }\backslash(-\infty,L_+]$ such that $\Im\log(z-t)\in(-\pi,\pi)$ for all $t\in\text{supp}(\mu_{fc})$.
\item Suppose $\epsilon\in(\frac{1}{b+1},\frac{1}{12})$ and
$$\Omega_1(\epsilon)=\Big\{\vert \lambda_1-L_+\vert <\frac{\min(\hat\gamma-L_+,1)}{20}\quad\text{and}\quad\vert \lambda_N-L_-\vert <\frac{\min(\hat\gamma-L_+,1)}{20} \Big\}\cap \Omega_V(\epsilon)\cap\tilde \Omega(\epsilon).$$
\end{itemize}
\end{definition}
\begin{remark}
	$\hat R(z)$ is an analogue of the $R(z)$ defined in Definition \ref{definition:R(z)}. Obviously the $\hat\gamma$ defined in Theorem \ref{thm:high_temperature} is the unique point on $(L_+,+\infty)$ such that $\hat R'(\hat \gamma)=0$. According to Theorem \ref{thm:fluctuation_of_lambda_1} and Proposition \ref{proposition:extreme_eigenvalue} we have
	\begin{align}\label{eq72}
		{\mathbb P}(\Omega_1(\epsilon))\to1\quad\text{as }N\to\infty.
	\end{align}
\end{remark}  
\begin{lemma}\label{lemma:distance_between_gamma_and_hat_gamma} 
There exists a constant $N_0>0$ such that if $N>N_0$ and $\Omega_1(\epsilon)$ holds, then
$$\vert \gamma-\hat\gamma\vert \le N^{3\epsilon-\frac{1}{2}},\quad\vert  R(\hat\gamma)-R(\gamma)\vert \le \frac{8}{(\hat\gamma-L_+)^2}N^{6\epsilon-1}$$
where $\gamma$ was defined in Definition \ref{definition:R(z)} and $\hat\gamma$ was  defined in Theorem \ref{thm:high_temperature}.
\end{lemma}
\begin{proof}
Let $\mathcal L$ be the boundary of the rectangle with vertices $$L_++\frac{\min(\hat\gamma-L_+,1)}{3}\pm\frac{\min(\hat\gamma-L_+,1)}{3}\cdot\i\quad\text{and}\quad L_--\frac{\min(\hat\gamma-L_+,1)}{3}\pm\frac{\min(\hat\gamma-L_+,1)}{3}\cdot\i$$
with counterclockwise orientation.  By Lemma \ref{lemma:all_eigenvalues_in_2+lambda+r} and Lemma \ref{lemma:lambdav_i-z-m_fc_large_on_Gamma}, if $N$ is large enough, then
$$\mathcal L\cap\{z:\Im z\ge N^{-\frac{1}{2}-\epsilon}\}\subset \mathcal D_\epsilon'.$$
 
Notice that if $x>(\hat\gamma+L_+)/2$, $N$ is large enough and $\Omega_1(\epsilon)$ holds then 
\begin{itemize}
	\item 
\begin{align} \label{eq65}
	\vert R^{(l)}(x)\vert =\frac{1}{N}\vert \sum_{i=1}^N\frac{(l-1)!}{(x-\lambda_i)^l}\vert \le \frac{(l-1)!}{((\hat\gamma-L_+)/4)^l}\quad (l=2,3,\ldots)
\end{align}
\item
\begin{align}\label{eq64}
	\frac{1}{x-t}=\frac{1}{2\pi\i}\oint_{\mathcal L}\frac{1}{(x-\xi)(\xi-t)}d\xi\quad\text{for any $ t$  enclosed by $\mathcal L$}
\end{align}
\item
\begin{multline}\label{eq66}
\vert R'(x)-\hat R'(x)\vert =\vert \int\frac{1}{x-t}d\mu_{fc}(t)-\frac{1}{N}\sum\frac{1}{x-\lambda_i}\vert =\frac{1}{2\pi}\Big\vert \oint_{\mathcal L}\frac{m_N(\xi)-m_{fc}(\xi)}{x-\xi}d\xi\Big\vert \quad\text{(by \eqref{eq64})}\\
\le\frac{1}{2\pi}\int_{\mathcal L\cap{\vert \Im \xi\vert \le N^{-\frac{1}{2}-\epsilon}}}\frac{\vert m_N(\xi)-m_{fc}(\xi)\vert }{\vert x-\xi\vert }d\xi+\frac{1}{2\pi}\int_{\mathcal L\cap{\vert \Im \xi\vert > N^{-\frac{1}{2}-\epsilon}}}\frac{\vert m_N(\xi)-m_{fc}(\xi)\vert }{\vert x-\xi\vert }d\xi\\
\le\frac{100}{\pi(\min(\hat\gamma-L_+,1))^2}N^{-\frac{1}{2}-\epsilon}+\frac{6\vert \mathcal L\vert }{\pi\cdot \min(\hat\gamma-L_+,1)}N^{2\epsilon-\frac{1}{2}}\quad\text{(by  Lemma \ref{lemma:hatm_fc_close_tom_fc} and Definition \ref{definition:D_epsilon_andD'_epsilon})}\\
\le C N^{2\epsilon-\frac{1}{2}}
\end{multline}
for some constant $C>0$.

\end{itemize}

By mean value theorem and the fact that $\hat R'(\hat\gamma)=0$,
\begin{multline}\label{eq67}
	\hat R'(\hat\gamma+ N^{3\epsilon-\frac{1}{2}})=\hat R''(\hat\gamma+t_1N^{3\epsilon-\frac{1}{2}})\cdot N^{3\epsilon-\frac{1}{2}}=\hat R''(\hat\gamma)\cdot N^{3\epsilon-\frac{1}{2}}+\hat R'''(\hat\gamma+t_1t_2 N^{3\epsilon-\frac{1}{2}})\cdot t_1 N^{6\epsilon-1}
\end{multline}
\begin{multline}\label{eq68}
	\hat R'(\hat\gamma- N^{3\epsilon-\frac{1}{2}})=-\hat R''(\hat\gamma-t_1'N^{3\epsilon-\frac{1}{2}})\cdot N^{3\epsilon-\frac{1}{2}}=-\hat R''(\hat\gamma)\cdot N^{3\epsilon-\frac{1}{2}}+\hat R'''(\hat\gamma-t_1't_2' N^{3\epsilon-\frac{1}{2}})\cdot t_1' N^{6\epsilon-1}
\end{multline}
where $t_1$, $t_2$, $t_1'$, $t_2'$  are all in $[0,1]$. According to \eqref{eq65}, \eqref{eq66}, \eqref{eq67}, \eqref{eq68} and the fact that $\hat R''(\hat\gamma)>0$, we have that if $N$ is large enough and $\Omega_1(\epsilon)$ holds, then
$$R'(\hat\gamma+N^{3\epsilon-\frac{1}{2}})\ge\hat R'(\hat\gamma+N^{3\epsilon-\frac{1}{2}})-CN^{2\epsilon-\frac{1}{2}}\ge \hat R''(\hat\gamma)\cdot N^{3\epsilon-\frac{1}{2}}-\frac{128}{(\hat\gamma-L_+)^3}\cdot t_1 N^{6\epsilon-1}-CN^{2\epsilon-\frac{1}{2}}>0$$
$$R'(\hat\gamma-N^{3\epsilon-\frac{1}{2}})\le\hat R'(\hat\gamma-N^{3\epsilon-\frac{1}{2}})+CN^{2\epsilon-\frac{1}{2}}\le -\hat R''(\hat\gamma)\cdot N^{3\epsilon-\frac{1}{2}}+\frac{128}{(\hat\gamma-L_+)^3}\cdot t_1' N^{6\epsilon-1}+CN^{2\epsilon-\frac{1}{2}}<0$$
thus
$$\vert \gamma-\hat\gamma\vert \le N^{3\epsilon-\frac{1}{2}}$$
because $R'(\gamma)=0$ and $R'$ is increasing on $(\lambda_1,+\infty)$. 

For the second conclusion, according to  Taylor's formula and the fact that $R'(\gamma)=0$, we have:
$$R(\hat\gamma)-R(\gamma)=\frac{1}{2}R''(\gamma+s(\hat\gamma-\gamma))(\hat\gamma-\gamma)^2$$
for some $s\in[0,1]$. This together with \eqref{eq65} and the first conclusion yields the second conclusion. 
\end{proof}

\begin{lemma}\label{lemma:analogue_of_lemma_5.4}
Suppose $c_3\in(0,1/10)$. There exists  constants $c_4>0$ and $N_0>0$ such that if $N>N_0$ and $\Omega_1(\epsilon)$ holds, then
	$$\int_{\gamma-\i\infty}^{\gamma+\i\infty}e^{\frac{N}{2}R(z)}dz=\i e^{\frac{N}{2}R(\gamma)}\sqrt{\frac{4\pi}{NR''(\gamma)}}(1+w_N)$$
	where $\vert w_N\vert \le c_4N^{4c_3-\frac{1}{2}}$.
\end{lemma}
\begin{proof}
	Suppose $\Omega_1(\epsilon)$ holds. Notice that
	\begin{multline*}
		\int_{\gamma-\i\infty}^{\gamma+\i\infty}e^{\frac{N}{2}R(z)}dz=\frac{\i}{\sqrt N}\int_{-\infty}^{+\infty}\exp\Big(\frac{N}{2}R(\gamma+\frac{\i t}{\sqrt N})\Big)dt\\
		=\frac{\i e^{\frac{N}{2}R(\gamma)}}{\sqrt N}\int_{-\infty}^{+\infty}\exp\Big(\frac{N}{2}\Big(R(\gamma+\frac{\i t}{\sqrt N})-R(\gamma)\Big)\Big)dt.
	\end{multline*}
	Using the Taylor's formula (for complex analytic functions), if $N$ is large enough and $\vert t\vert \le N^{c_3}$ then
	\begin{align*}
		R(\gamma+\frac{\i t}{\sqrt N})-R(\gamma)=\frac{R''(\gamma)}{2}(\frac{\i t}{\sqrt N})^2+\frac{R'''(\gamma)}{6}(\frac{\i t}{\sqrt N})^3+r_N(t) 
	\end{align*}
	and the remaining term $r_N(t)$ satisfies:
	\begin{align}\label{eq74}
		\vert r_N(t)\vert =\Big\vert (\frac{\i t}{\sqrt N})^4\frac{1}{2\pi\i}\oint_ {\vert w-\hat\gamma\vert =(\hat\gamma-L_+)/2}\frac{R(w)}{(w-\gamma)^4(w-\gamma-\frac{\i t}{\sqrt N})}dw\Big\vert \le C_1t^4/N^2\le C_1N^{4c_3-2}
	\end{align}
	for some $t$-independent constant $C_1>0$. By \eqref{eq65}, if $N$ is large enough and $\vert t\vert \le N^{c_3}$ then
	\begin{align} \label{eq75}
		\vert \frac{N}{2}\frac{R'''(\gamma)}{6}(\frac{\i t}{\sqrt N})^3\vert \le C_2N^{3c_3-\frac{1}{2}}
	\end{align}
	for some $t$-independent constant $C_2>0$, therefore we have 
	\begin{multline} \label{eq70}
		\Big\vert 	\int_{-N^{c_3}}^{N^{c_3}}\exp\Big(\frac{N}{2}\Big(R(\gamma+\frac{\i t}{\sqrt N})-R(\gamma)\Big)\Big)dt-\int_{-N^{c_3}}^{N^{c_3}}\exp\Big(-\frac{t^2}{4}R''(\gamma)\Big)dt\Big\vert \\
		=\Big\vert \int_{-N^{c_3}}^{N^{c_3}}\Big[\exp\Big(\frac{N}{2}\Big(\frac{R''(\gamma)}{2}(\frac{\i t}{\sqrt N})^2+\frac{R'''(\gamma)}{6}(\frac{\i t}{\sqrt N})^3+r_N(t)\Big)\Big)-\exp\Big(-\frac{t^2}{4}R''(\gamma)\Big)\Big]dt\Big\vert \\
		\le\int_{-N^{c_3}}^{N^{c_3}}\exp\Big(-\frac{t^2}{4}R''(\gamma)\Big)\cdot2\cdot\Big\vert \frac{N}{2}\Big(\frac{R'''( \gamma)}{6}(\frac{\i t}{\sqrt N})^3+r_N(t)\Big)\Big\vert  dt\\
		\le 5C_2N^{4c_3-\frac{1}{2}}
	\end{multline}
	where we used \eqref{eq74}, \eqref{eq75} and the fact that $R''(\gamma)>0$ in the last inequality.
	
	Since $\Omega_1(\epsilon)$ holds, Lemma \ref{lemma:distance_between_gamma_and_hat_gamma} yields $ (\hat\gamma-L_+)/2\le\vert \gamma-\lambda_i\vert \le 2(\hat\gamma-L_+)+(L_+-L_-)$, so 
	\begin{multline} \label{eq69}
		\int_{N^{c_3}}^\infty \Big\vert \exp\Big(\frac{N}{2}\Big(R(\gamma+\frac{\i t}{\sqrt N})-R(\gamma)\Big)\Big)\Big\vert dt=\int_{N^{c_3}}^\infty\exp\Big(\frac{-1}{2}\sum_{i=1}^N\log\sqrt{1+\frac{t^2}{N(\gamma-\lambda_i)^2}}\Big)dt\\
		\le \int_{N^{c_3}}^\infty\exp\Big(\frac{-N}{2}\log\sqrt{1+\frac{C_3t^2}{N}}\Big)dt= \int_{N^{c_3}}^\infty\exp\Big(\frac{-N}{4}\log(1+\frac{C_3t^2}{N})\Big)dt
	\end{multline}
	for some constant $C_3>0$. Plugging
	\begin{align*}
		\log(1+\frac{C_3t^2}{N})\ge\begin{cases}
			\frac{C_3t^2}{2N}\ge \frac{C_3N^{2c_3}}{2N}\quad\text{if }t\in[N^{c_3},\sqrt{N/C_3}]\quad\text{(since $\log(1+x)\ge\frac{x}{2}$ on $[0,1]$)}\\
			\log2\ge\frac{C_3N^{2c_3}}{2N}\quad\text{if }t\in[\sqrt{N/C_3},N] \text{ and }N>N_0
		\end{cases}
	\end{align*}
	into \eqref{eq69}, we know if $N$ is large enough, then
	\begin{multline} \label{eq71}
		\int_{N^{c_3}}^\infty \Big\vert \exp\Big(\frac{N}{2}\Big(R(\gamma+\frac{\i t}{\sqrt N})-R(\gamma)\Big)\Big)\Big\vert dt\le \int_{N^{c_3}}^N\exp\Big(-\frac{N}{4}\frac{C_3N^{2c_3}}{2N}\Big)dt+\int_N^\infty \exp(-\frac{N}{4}\log(\frac{C_3t^2}{N}))dt\le \frac{1}{N}.
	\end{multline}
	Similarly we have $	\int_{-\infty}^{-N^{c_3}} \Big\vert \exp\Big(\frac{N}{2}\Big(R(\gamma+\frac{\i t}{\sqrt N})-R(\gamma)\Big)\Big)\Big\vert dt\le1/N$ when $N$ is large enough. This together with  \eqref{eq70} and \eqref{eq71} imply that if $N$ is large enough, then:
	\begin{multline} 
		\Big\vert 	\int_{\mathbb R }\exp\Big(\frac{N}{2}\Big(R(\gamma+\frac{\i t}{\sqrt N})-R(\gamma)\Big)\Big)dt-\int_{\mathbb R }\exp\Big(-\frac{t^2}{4}R''(\gamma)\Big)dt\Big\vert \\
		\le 5C_2N^{4c_3-\frac{1}{2}}+\frac{2}{N}+\int_{[-N^{c_3},N^{c_3}]^c}\exp\Big(-\frac{t^2}{4}R''(\gamma)\Big)dt\le 5C_2N^{4c_3-\frac{1}{2}}+\frac{2}{N}+\frac{1}{N}\le6C_2N^{4c_3-\frac{1}{2}} 
	\end{multline}
Since $\int_{\mathbb R }\exp\Big(-\frac{t^2}{4}R''(\gamma)\Big)dt=\sqrt{4\pi/R''(\gamma)}$, we complete the proof by \eqref{eq65}.
\end{proof}

\begin{proof}[Proof of Theorem \ref{thm:high_temperature}]
	According to Lemma \ref{lemma:integral_formula} and Lemma \ref{lemma:analogue_of_lemma_5.4}, if $N$ is large enough and $\Omega_1(\epsilon)$ holds, then
	$$\int_{S_{N-1}}e^{\beta\langle\sigma,(W+\lambda V)\sigma\rangle}d\omega_N(\sigma)=\frac{\sqrt N \beta}{\sqrt\pi(2\beta e)^{N/2}}\cdot e^{\frac{N}{2}R(\gamma)}\sqrt{\frac{4\pi}{NR''(\gamma)}}(1+u_N)$$ 
	where $\vert u_N\vert \le N^{-1/3}$ and thus by Lemma \ref{lemma:distance_between_gamma_and_hat_gamma}
	\begin{align*}
		F_N=\frac{1}{N}\log\int_{S_{N-1}}e^{\beta\langle\sigma,(W+\lambda V)\sigma\rangle}d\omega_N(\sigma)=-\frac{1}{2}\log(2\beta e)+\beta\hat\gamma-\frac{1}{2N}\sum_{i=1}^N\log(\hat\gamma-\lambda_i)+t_N
	\end{align*}
	where 
	\begin{align}\label{eq73}
		\vert t_N\vert \le CN^{6\epsilon-1}
	\end{align}
	for some constant $C>0$. Therefore, if $N$ is large enough and $\Omega_1(\epsilon)$ holds, then
	\begin{multline*}
		-2\sqrt N\Big(F_N+\frac{1}{2}\log(2\beta e)-\beta\hat\gamma\Big)-\sqrt N\int\log(\hat\gamma-t)d\mu_{fc}(t)\\
		=\frac{1}{\sqrt N}\sum_{i=1}^N\log(\hat\gamma-\lambda_i)-\sqrt N\int\log(\hat\gamma-t)d\mu_{fc}(t)-2\sqrt N\cdot t_N.
	\end{multline*}
According to Theorem \ref{thm:CLT}, \eqref{eq73}, \eqref{eq72} and the assumption that $\epsilon<1/12$, we complete the proof.
	
\end{proof}

\appendix
\section{Analysis on the curve of steepest-descent: proof of Lemma \ref{lemma:steepest_descent_curve}}\label{appendix:steepest_descent_curve}
Now we study the curve of steepest-descent.  In this section,
\begin{itemize}
	\item 	$N$ is fixed;
	\item we do not consider  randomness. In other words, we can imagine that the sample point in the probability space is fixed.
\end{itemize}

\begin{lemma}\label{lemma:existence+uniqueness_of_h(y)}
	If $\Im R(x+\i y)=0$ and $y>0$, then $y\in(0,\frac{\pi}{2\beta})$. On the other hand, for any $y\in(0,\frac{\pi}{2\beta})$ there is a unique $x\in{\mathbb R }$ such that $\Im R(x+\i y)=0$.
\end{lemma}
\begin{proof}
	By definition we have that if $x\in{\mathbb R }$ and $y>0$ then
	\begin{align}\label{eq30}
		\Im R(x+\i y)=2\beta y-\frac{1}{N}\sum_{i=1}^N\arccos\frac{x-\lambda_i}{\sqrt{(x-\lambda_i)^2+y^2}}.
	\end{align}
	So if $\Im R(x+\i y)=0$ then $2\beta y=\frac{1}{N}\sum_{i=1}^N\arccos\frac{x-\lambda_i}{\sqrt{(x-\lambda_i)^2+y^2}}<\pi$, thus $y<\pi/(2\beta)$.
	
	On the other hand, suppose $y\in(0,\pi/(2\beta))$. Let
	$$f_y(x)=\Im R(x+\i y)=2\beta y-\frac{1}{N}\sum_{i=1}^N\arccos\frac{x-\lambda_i}{\sqrt{(x-\lambda_i)^2+y^2}}.$$
	Then $\lim\limits_{x\to-\infty}f_y(x)=2\beta y-\pi <0$ and $\lim\limits_{x\to+\infty}f_y(x)=2\beta y>0$. By continuity there exists $x\in{\mathbb R }$ such that $f_y(x)=0$ so $\Im R(x+\i y)=0$. Moreover, if $x_1<x_2$ such that $f_y(x_1)=f_y(x_2)=0$ then there is $x_3\in(x_1,x_2)$ with $f_y'(x_3)=0$. But $f_y'(x)=\frac{y}{N}\sum\limits_{i=1}^N\frac{1}{(x-\lambda_i)^2+y^2}>0$, so such $x_3$ cannot exist. In summary, there is a unique $x\in{\mathbb R }$ such that $\Im R(x+\i y)=0$.  
\end{proof}

\begin{lemma}\label{lemma:curve_to_the_left_of_gamma_a}
	$h(y)<\gamma$ for all $y\in(0,\frac{\pi}{2\beta})$
\end{lemma}
\begin{proof}
	Suppose $y_0\in(0,\frac{\pi}{2\beta})$. If $h(y_0)=\gamma$, then $\Im R(\gamma+\i y_0)=0$. Since $\Im R(\gamma)=0$, there must be $y_1\in(0,y_0)$ such that
	$$\frac{\partial}{\partial y}\Big\vert _{y=y_1}\Im R(\gamma+\i y)=0.$$
	But
	$$\frac{\partial}{\partial y}\Big\vert _{y=y_1}\Im R(\gamma+\i y)=2\beta-\frac{1}{N}\sum_{i=1}^N\frac{\gamma-\lambda_i}{(\gamma-\lambda_i)^2+y_1^2}>2\beta-\frac{1}{N}\sum_{i=1}^N\frac{1}{\gamma-\lambda_i}=0.$$
	The definition of $\gamma$ is used in the last identity. This means $h(y)\ne\gamma$ for all $y\in(0,\frac{\pi}{2\beta})$. Since 
	$$\lim_{y\to \frac{\pi}{2\beta}-}h(y)=-\infty,$$
	we know by continuity that  $h(y)<\gamma$ for all $y\in(0,\frac{\pi}{2\beta})$.
\end{proof}
\begin{lemma}\label{lemma:monotone}
	If   $y\in[\frac{1}{4\beta},\frac{\pi}{2\beta})$ then
	\begin{align}\label{eq32}
		h'(y)\le\frac{1}{2}-2\beta y.
	\end{align}
\end{lemma}
\begin{proof}Suppose $y\in(0,\frac{\pi}{2\beta})$.
	According to \eqref{eq30} and the implicit function theorem,
	\begin{align}\label{eq31}
		-h'(y)=\frac{\big(\frac{\partial\Im R(x+\i y)}{\partial y}\big)}{\big(\frac{\partial\Im R(x+\i y)}{\partial x}\big)}\Bigg\vert _{x=h(y)}=\frac{2\beta-\frac{1}{N}\sum_{i=1}^N\frac{h(y)-\lambda_i}{(h(y)-\lambda_i)^2+y^2}}{y\cdot\frac{1}{N}\sum_{i=1}^N\frac{1}{(h(y)-\lambda_i)^2+y^2}}.
	\end{align}
	Since $$\Big\vert \frac{h(y)-\lambda_i}{(h(y)-\lambda_i)^2+y^2}\Big\vert \le\frac{1}{2y},\quad \frac{1}{(h(y)-\lambda_i)^2+y^2}\le\frac{1}{y^2},$$
	we see that on the right hand side of \eqref{eq31}, the numerator is larger than or equal to $2\beta-\frac{1}{2y}$ and the denominator is less than or equal to $1/y$. So the lemma is proved.
\end{proof}
\begin{corollary}\label{coro:derivative_of_inverse_function}
	If $\frac{1}{4}<c_0<\frac{\pi}{2}$, then $h$ is a bijection from $[\frac{c_0}{\beta},\frac{\pi}{2\beta})$ to $(-\infty,h(\frac{c_0}{\beta})]$. The inverse function satisfies:
	$$\frac{2}{1-4c_0}\le(h^{-1})'(x)<0,\quad\forall x\in (-\infty,h(\frac{c_0}{\beta})].$$
\end{corollary}
\begin{proof}
	By \eqref{eq32} we know $h$ is strictly decreasing on $[\frac{c_0}{\beta},\frac{\pi}{2\beta})$, so is bijective on this interval. Then using $(h^{-1})'=(h')^{-1}$ we complete the proof.
\end{proof}

\begin{lemma}\label{lemma:question}
	We have
	$$\lim\limits_{y\to0+}h(y)=\gamma,\quad \lim\limits_{y\to0+}\frac{h(y)-\gamma}{y}=0\quad\text{and}\quad\lim\limits_{y\to0+}h'(y)=0.$$
\end{lemma}

\begin{proof}
We put this proof at the end of this section.
\end{proof}
Now we are ready to prove Lemma \ref{lemma:steepest_descent_curve}.
\begin{proof}[Proof of Lemma \ref{lemma:steepest_descent_curve}] 
	The first, third and last conclusions of Lemma \ref{lemma:steepest_descent_curve} come from Lemma \ref{lemma:existence+uniqueness_of_h(y)}, Lemma \ref{lemma:curve_to_the_left_of_gamma_a} and Lemma \ref{lemma:monotone} respectively. Notice that $R(\bar z)=\overline{R(z)}$ and that $\gamma$ is the only number in $(\lambda_1,\infty)$ where $R'=0$. Moreover, according to Lemma \ref{lemma:steepest_descent_curve}, the $y$-coordinate of any point on $S$ is in $(-\frac{\pi}{2\beta},\frac{\pi}{2\beta})$. So we have $S=\{h(y)+\i y\vert -\frac{\pi}{2\beta}<y<\frac{\pi}{2\beta}\}$. Finally, by Lemma \ref{lemma:question} we have   $h(y)\in C^1((-\frac{\pi}{2\beta},\frac{\pi}{2\beta}))$.
\end{proof}
\begin{proof}[Proof of Lemma \ref{lemma:question}]
	We notice that $N$ is a fixed number in this lemma. If $\lambda_1=\cdots=\lambda_N$ then the lemma is trivial. Now we assume that $$\lambda_1=\cdots=\lambda_M>\lambda_{M+1}\ge\cdots\ge \lambda_N$$ 
	for some $M\in[1,N-1]$.
	\begin{lemma}\label{lemma:arccos}
		\begin{itemize}
			\item If $0<t<1$ then there exists $t_1\in[0,t]$ such that
			$$\arccos\sqrt{1-t^2}=t+\frac{t^3}{6}\cdot\frac{1+2t_1^2}{(1-t_1^2)^{5/2}}.$$
			\item There exists $w_0>0$ such that $\frac{\arccos\sqrt{1-t^2}}{t}\ge w_0$ for all $t\in[0,1]$.
		\end{itemize}  
	\end{lemma}
	\begin{proof}
		The first conclusion is from Taylor's formula. The other  conclusion is trivial.
	\end{proof}
	Now we use Lemma \ref{lemma:arccos} to prove Lemma \ref{lemma:question}.
	\begin{enumerate}
		\item According to \eqref{eq30},
		\begin{align}\label{eq76}
			2\beta y=\frac{1}{N}\sum_{i=1}^N\arccos\frac{h(y)-\lambda_i}{\sqrt{(h(y)-\lambda_i)^2+y^2}}\quad\forall y\in(0,\frac{\pi}{2\beta})
		\end{align}
		so we have 
		$$h(y)> \lambda_1\quad\forall y\in(0,\frac{\pi}{4N\beta}]$$
		otherwise the right hand side of \eqref{eq76} is larger than $\frac{1}{N}\arccos\frac{h(y)-\lambda_1}{\sqrt{(h(y)-\lambda_1)^2+y^2}}\ge\frac{\pi}{2N}\ge2\beta y$.
		 
		For $y\in(0,\frac{\pi}{4N\beta}]$, define
		$$f(y):=2\beta y-\frac{1}{N}\sum_{i=M+1}^N\arccos\frac{h(y)-\lambda_i}{\sqrt{(h(y)-\lambda_i)^2+y^2}}$$
		By Lemma  \ref{lemma:arccos} 
		there exists $\delta_0\in(0,\frac{\pi}{4N\beta})$ such that if $y\in(0,\delta_0)$ then
		\begin{multline}\label{eq78}
			\vert f(y)\vert =\Big\vert 2\beta y-\frac{1}{N}\sum_{i=M+1}^N\arccos\sqrt{1-\Big(\frac{y}{\sqrt{(h(y)-\lambda_i)^2+y^2}}\Big)^2}\Big\vert \\
			\le 2\beta y+\frac{1}{N}\sum_{i=M+1}^N\frac{2y}{\sqrt{(h(y)-\lambda_i)^2+y^2}}\le 2\beta y+\frac{2y}{\lambda_1-\lambda_{M+1}}=w_1y
		\end{multline}
		where $w_1:=2\beta+\frac{2}{\lambda_1-\lambda_{M+1}}$.   By Lemma  \ref{lemma:arccos}, \eqref{eq78}, \eqref{eq76} and the definition of $f(y)$, if $y\in(0,\delta_0)$ then
		\begin{multline*}
			w_1y\ge \vert f(y)\vert \ge\frac{M}{N}\arccos\frac{h(y)-\lambda_1}{\sqrt{(h(y)-\lambda_1)^2+y^2}}=\frac{M}{N}\arccos\sqrt{1-\Big(\frac{y}{\sqrt{(h(y)-\lambda_1)^2+y^2}}\Big)^2}\\
			\ge \frac{M}{N}\frac{w_0y}{\sqrt{(h(y)-\lambda_1)^2+y^2}}
		\end{multline*}
		which implies
		$$\sqrt{(h(y)-\lambda_1)^2+y^2}\ge \frac{Mw_0}{Nw_1}.$$
		So there exists $\delta_1\in(0,\delta_0)$ such that
		\begin{align}\label{eq109}
			h(y)-\lambda_1\ge \frac{Mw_0}{2Nw_1}\quad\forall y\in(0,\delta_1)
		\end{align}
		According to \eqref{eq30} and \eqref{eq31}, if $y\in(0,\delta_1)$ then
		\begin{multline}\label{eq79}
			-h'(y)\cdot \frac{1}{N}\sum_{i=1}^N\frac{1}{(h(y)-\lambda_i)^2+y^2} =\frac{1}{y}\Big(2\beta-\frac{1}{N}\sum_{i=1}^N\frac{h(y)-\lambda_i}{(h(y)-\lambda_i)^2+y^2}\Big) \\
			=\frac{1}{y}\Big(\frac{1}{yN}\sum_{i=1}^N\arccos\frac{h(y)-\lambda_i}{\sqrt{(h(y)-\lambda_i)^2+y^2}}-\frac{1}{N}\sum_{i=1}^N\frac{h(y)-\lambda_i}{(h(y)-\lambda_i)^2+y^2}\Big)=\frac{1}{y^2N}\sum_{i=1}^N A_i
		\end{multline}
		where
		$$A_i=\arccos\sqrt{1-\Big(\frac{y}{\sqrt{(h(y)-\lambda_i)^2+y^2}}\Big)^2}-\frac{y(h(y)-\lambda_i)}{(h(y)-\lambda_i)^2+y^2}. $$
		By Lemma  \ref{lemma:arccos} and  \eqref{eq109}, there exist constants $w_2>0$ and $\delta_2\in(0,\delta_1)$ such that if $y\in(0,\delta_2)$ then $\vert A_i\vert \le w_2 y^3$ and thus by  \eqref{eq79} and  Lemma \ref{lemma:curve_to_the_left_of_gamma_a}:
		\begin{align}\label{eq81}
			\vert h'(y)\vert \le \vert \frac{1}{N}\sum_{i=1}^N\frac{1}{(h(y)-\lambda_i)^2+y^2}\vert ^{-1} w_2 y\le y\cdot w_2\big((\gamma-\lambda_N)^2+\big(\frac{\pi}{2\beta}\big)^2\big).
		\end{align}
		This tells us the boundedness of $h'(y)$ on $(0,\delta_2)$. So by the completeness of ${\mathbb R }$ we know $\lim\limits_{y\to0+}h(y)$ exists. Now multiplying both sides of the first identity in \eqref{eq79} by $y$, letting $y\to0+$, using the 
		boundedness of $h'(y)$ on $(0,\delta_2)$, we have:
		$$2\beta-\frac{1}{N}\sum_{i=1}^N\frac{1}{\big(\lim\limits_{y\to0+}h(y)\big)-\lambda_i}=0.$$
		This together with the definition of $\gamma$ completes the proof of the first conclusion.
		\item 
		Plugging $2\beta=\frac{1}{N}\sum\frac{1}{\gamma-\lambda_i}$ into the first identity of \eqref{eq79}, we have for $y\in(0,\pi/(2\beta))$:
		\begin{multline}\label{eq82}
			-h'(y)\cdot \frac{1}{N}\sum_{i=1}^N\frac{1}{(h(y)-\lambda_i)^2+y^2} =\frac{1}{y}\Big(\frac{1}{N}\sum\limits_{i=1}^N\frac{1}{\gamma-\lambda_i}-\frac{1}{N}\sum_{i=1}^N\frac{h(y)-\lambda_i}{(h(y)-\lambda_i)^2+y^2}\Big)\\
			=\frac{y}{N}\sum_{i=1}^N\frac{1}{(\gamma-\lambda_i)((h(y)-\lambda_i)^2+y^2)}+\frac{h(y)-\gamma}{y}\cdot\frac{1}{N}\sum_{i=1}^N\frac{h(y)-\lambda_i}{(\gamma-\lambda_i)((h(y)-\lambda_i)^2+y^2)}
		\end{multline}
		Taking $y\to0+$ on both sides of \eqref{eq82}, using \eqref{eq81} and the first conclusion of this lemma, we have 
		$$\Big(\lim\limits_{y\to0+}\frac{h(y)-\gamma}{y}\Big)\cdot\frac{1}{N}\sum_{i=1}^N\frac{1}{(\gamma-\lambda_i)^2}=0$$
		which completes the proof of the second conclusion.
		\item
		Now we use the first two conclusions to prove the third conclusion. According to \eqref{eq31} and the fact that $\lim_{y\to0+}\frac{1}{N}\sum_{i=1}^N\frac{1}{(h(y)-\lambda_i)^2+y^2}=\frac{1}{N}\sum_{i=1}^N\frac{1}{(\gamma-\lambda_i)^2}>0$,  we know from the formula $\frac{\partial \Im R}{\partial y}=\Re R'$ that it suffices to show that
		\begin{align}\label{eq50}
			\lim_{y\to0+}\frac{1}{y}\cdot\Re\big (R'(h(y)+\i y)\big)=0
		\end{align}
		Notice that $R'(\gamma)=0$. So
		\begin{align}\label{eq49} 
			\frac{1}{y}\cdot  R'(h(y)+\i y)= \frac{ R'(h(y)+\i y)-R'(\gamma)}{h(y)+\i y-\gamma}\cdot\frac{h(y)+\i y-\gamma}{y}.
		\end{align}
		According to the first two conclusions of this lemma and the mean-value theorem, \eqref{eq49} must converges to $\i\cdot R''(\gamma)=\frac{\i}{N}\sum_{i=1}^N\frac{1}{(\gamma-\lambda_i)^2}$ as $y\to0+$, so \eqref{eq50} is true.
	\end{enumerate}
\end{proof}

\section {Proofs of auxiliary lemmas}\label{appendix}
\begin{proof}[Proof of Lemma \ref{lemma:lambdav_i-z-m_fc_large_on_Gamma}]
If $z$ is on the upper or lower edge of $\Gamma$ then 
$\vert \Im(\lambda v_i-z-m_{fc}(z))\vert \ge\vert \Im z\vert =d$. So we only need to prove the lemma for $z$ on the left and right edges. Now let $z$ be on the right edge of $\Gamma$. The case that $z$ is on the left edge can be proved by the same method. Notice that
$$\frac{d}{dx}(x+m_{fc}(x))>0\quad\text{on} (L_+,\infty),$$
so by Lemma \ref{lemma:properties_of_mu_fc} $$C:=(L_++d)+m_{fc}(L_++d)-\lambda>L_++m_{fc}(L_+)-\lambda=0.$$
Since $v_i\in[-1,1]$, we have that
$$\min_i\vert L_++d+m_{fc}(L_++d)-\lambda v_i\vert \ge C.$$
By continuity there is $y_0>0$ such that if $z=L_++d+\i y$ with $y\in[-y_0,y_0]$ then
$$\min_i\vert z+m_{fc}(z)-\lambda v_i\vert \ge C/2.$$
If $z=L_++d+\i y$ with $y\not\in[-y_0,y_0]$ then
$$\min_i\vert z+m_{fc}(z)-\lambda v_i\vert \ge \vert \Im (z+m_{fc}(z))\vert \ge \vert \Im z\vert \ge y_0.$$
Taking $C_d=\min(\frac{C}{2}, y_0)$ we complete the proof of the first conclusion. Since $d$ can be arbitrarily small, the second conclusion follows from the first conclusion.
\end{proof}

\begin{proof}[Proof of Lemma \ref{lemma:properties_of_good_event}]
	The first conclusion is from \eqref{eqn22},   Lemma \ref{lemma:lambdav_i-z-m_fc_large_on_Gamma} and the facts that $\varsigma\ge\frac{1}{1+b}$, $\varpi\le\frac{1}{2}+\varsigma$ and $d<1$. The second conclusion is from Proposition \ref{proposition:extreme_eigenvalue} and Theorem \ref{thm:local_law_for_resolvent_entries}. For the third conclusion, notice that if $N$ is large enough and $B_N\cap\Omega_V(\varsigma)$ holds, then the following statements hold for each $\xi\in\Gamma_+$
	\begin{multline*}
		\vert G_{ii}(\xi)-\hat g_i(\xi)\vert \le \vert G_{ii}(\xi)-\frac{1}{\lambda v_i-\xi-m_N(\xi)}\vert +\frac{\vert m_N-\hat m_{fc}\vert }{\vert \lambda v_i-\xi-m_N\vert \vert \lambda v_i-\xi-\hat m_{fc}\vert }\\
		\le  N^{\varsigma'-\frac{1}{2}}\cdot\vert \Im \xi \vert ^{-3}+N^{2\varsigma-\frac{1}{2}}\cdot\vert \Im \xi \vert ^{-2}\quad\text{(by definitions of $B_N$ and $\tilde\Omega(\varsigma)$)}
	\end{multline*}
	\begin{multline*}
		\vert G_{ii}(\xi)\vert \ge\vert \hat g_i(\xi)\vert -\vert G_{ii}(\xi)-\hat g_i(\xi)\vert \\
		\ge 2W'\vert \Im \xi\vert -N^{\varsigma'-\frac{1}{2}}\cdot\vert \Im \xi \vert ^{-3}-N^{2\varsigma-\frac{1}{2}}\cdot\frac{1}{\vert \Im \xi\vert ^2}\quad\text{(since $\frac{1}{\vert \hat g_i(\xi)\vert }\le\lambda+\vert \xi\vert +\vert \hat m_{fc}(\xi)\vert \le\lambda+\vert \xi\vert +\frac{1}{\vert \Im \xi\vert }$)}\\ 
		\ge W'\vert \Im \xi\vert \quad\text{(by definitions  of $\varpi$, $\varsigma$ and $\varsigma'$)}
	\end{multline*}
	\begin{multline*}
		\vert G_{ii}(\xi)-g_i(\xi)\vert \le\vert G_{ii}(\xi)-\frac{1}{\lambda v_i-\xi-m_N}\vert +\frac{\vert m_N-m_{fc}\vert }{\vert \lambda v_i-\xi-m_N\vert \vert \lambda v_i-\xi-m_{fc}\vert }\\
		\le N^{\varsigma'-\frac{1}{2}}\cdot\vert \Im \xi \vert ^{-3}+2N^{2\varsigma-\frac{1}{2}}\cdot\vert \Im \xi\vert ^{-2}
	\end{multline*}
	
		\begin{multline*}
			\vert \hat m_{fc}(\xi)-\frac{1}{N}\tr G^{(i)}(\xi)\vert \le \vert \hat m_{fc}-m_N\vert +\frac{1}{N}\vert G_{ii}\vert +\Big\vert \frac{1}{N}\sum_p^{(i)}\frac{(G_{ip})^2}{G_{ii}}\Big\vert \quad\text{(by \eqref{eqn3})}\\\le N^{2\varsigma-\frac{1}{2}}+\frac{1}{N\vert \Im \xi\vert }+\frac{1}{N\vert G_{ii}\vert }\vert (G^2)_{ii}-(G_{ii})^2\vert 
			\le N^{2\varsigma-\frac{1}{2}}+\frac{1}{N\vert \Im \xi\vert }+\frac{2}{W'N\vert \Im\xi\vert ^3}
		\end{multline*}
\end{proof}
\begin{proof}[Proof of Lemma \ref{lemma:hat_m_fc-m_fc}]
	By Lemma \ref{lemma:self_consistent_equation} and the definitions of $g_i$ and $\hat g_i$,
	\begin{multline}\label{eqn47}
		\sqrt N(\hat m_{fc}-m_{fc})=\sqrt N\Big(\frac{1}{N}\sum\hat g_i(\xi)-\int\frac{d\mu(t)}{\lambda t-\xi-m_{fc}(\xi)}\Big)
		=\frac{1}{\sqrt N}\sum\Big(\hat g_i(\xi)-\E[g_i(\xi)]\Big)\\
		=\frac{1}{\sqrt N}\sum\Big( g_i(\xi)-\E[g_i(\xi)]\Big)+\frac{1}{\sqrt N}\sum (\hat m_{fc}-m_{fc})g_i(\xi)\hat g_i(\xi)\\
		=\frac{1}{\sqrt N}\sum\Big( g_i(\xi)-\E[g_i(\xi)]\Big)+\frac{\hat m_{fc}-m_{fc}}{\sqrt N}(\sum g_i^2+\sum (\hat m_{fc}-m_{fc})\hat g_ig_i^2)\\
		=\frac{1}{\sqrt N}\sum\Big( g_i(\xi)-\E[g_i(\xi)]\Big)+\frac{(\hat m_{fc}-m_{fc})^2}{\sqrt N}\sum \hat g_ig_i^2
		+\sqrt N(\hat m_{fc}-m_{fc})(\frac{1}{N}\sum (g_i^2-\E [g_i^2]))\\+\sqrt N(\hat m_{fc}-m_{fc})\int\frac{d\mu(t)}{(\lambda t-\xi-m_{fc})^2}\quad\text{(since $\E [g_i^2]=\int\frac{d\mu(t)}{(\lambda t-\xi-m_{fc})^2}$)}
	\end{multline}
	Moving the last term on the right hand side of \eqref{eqn47} to the left hand side, multiplying both sides by $1+m_{fc}'(z)$, using
	\eqref{eqn48}, we complete the proof.
\end{proof}

\begin{proof}[Proof of Lemma \ref{coro:convergence_as_process}]

Let
$$Y_N(\xi)=N^{-\frac{1}{2}-a_1}\sum_i(g_i^2(\xi)-\E[g_i^2(\xi)]).$$
By the Cramer-Wold Theorem, for any $\xi_1,\ldots,\xi_k\in\Gamma$, if $N\to\infty$ then 
\begin{align}\label{eqn51}
	(Y_N(\xi_1),\ldots,Y_N(\xi_k))\to {\bf0}\quad\text{in distribution}
\end{align}
where ${\bf0}$ is the zero vector in ${\mathbb R }^k$. For any $\xi_1,\xi_2\in\Gamma$, by Lemma \ref{lemma:lambdav_i-z-m_fc_large_on_Gamma} and the definition of $\Gamma$, we have $\vert g_i(\xi_1)\vert \le\frac{1}{C_d}$, $\vert g_i(\xi_2)\vert \le\frac{1}{C_d}$ and
\begin{align}\label{eqn50}
	\vert g_i^2(\xi_1)-g_i^2(\xi_2)\vert \le \frac{2}{C_d^3}[\vert \xi_1-\xi_2\vert +\vert m_{fc}(\xi_1)-m_{fc}(\xi_2)\vert ] \le \frac{2}{C_d^3}(1+\frac{1}{d^2})\vert \xi_1-\xi_2\vert 
\end{align}
where $C_d$ is defined in Lemma \ref{lemma:lambdav_i-z-m_fc_large_on_Gamma}. 
So
\begin{multline*}
	\E[\vert Y_N(\xi_1)-Y_N(\xi_2)\vert ^2]\\
	=\frac{1}{N^{1+2a_1}}\E\Big[\sum_{i,j=1}^N\Big(g_i^2(\xi_1)-g_i^2(\xi_2)-\E[g_i^2(\xi_1)]+\E[g_i^2(\xi_2)]\Big)\Big(\overline{g_j^2(\xi_1)-g_j^2(\xi_2)-\E[g_j^2(\xi_1)]+\E[g_j^2(\xi_2)]}\Big)\Big]\\
	=\frac{1}{N^{1+2a_1}}\E\Big[\sum_{i=1}^N\Big\vert g_i^2(\xi_1)-g_i^2(\xi_2)-\E[g_i^2(\xi_1)]+\E[g_i^2(\xi_2)]\Big\vert ^2\Big]\quad\text{(by independence of $g_1,\ldots,g_N$)}\\
	\le \frac{1}{N^{2a_1}}\Big\vert \frac{4}{C_d^3}(1+\frac{1}{d^2})\vert \xi_1-\xi_2\vert \Big\vert ^2\quad\text{(by \eqref{eqn50})}
\end{multline*}

So  
$${\mathbb P}(\vert Y_N(\xi_1)-Y_N(\xi_2)\vert \ge s)\le\frac{1}{s^2}\E[\vert Y_N(\xi_1)-Y_N(\xi_2)\vert ^2]\le \frac{1}{s^2}\frac{1}{N^{2a_1}}\Big(\frac{4}{C_d^3}(1+\frac{1}{d^2})\Big)^2 \vert \xi_1-\xi_2\vert ^2,\quad\forall s>0$$
According to Theorem 12.3 of \cite{Billingsley},
$$\{Y_N(\xi)\vert \xi\in\Gamma\}\quad N=1,2,\ldots$$
is a tight sequence of random functions on $\Gamma$. This together with \eqref{eqn51} and Theorem 8.1 of \cite{Billingsley} imply that 
$$\{Y_N(\xi)\vert \xi\in\Gamma\} $$
(as random elements in the space of continuous functions on $\Gamma$)  converges in distribution to $0$ as $N\to\infty$. By Portmanteau's Theorem (see, for example, Theorem 2.1 of \cite{Billingsley})  the proof of the first conclusion is complete. The second conclusion can be proved in the same way.
\end{proof}

\begin{proof}[Proof of Lemma \ref{lemma:some_computation}]
	By \eqref{eqn31} we have
	\begin{multline}\label{eqn32}
		\E_N\Big[G_{ij}(\xi)\frac{\partial e^{\i tX_N}}{\partial W_{ij}}\Big]
		=\frac{\i t}{\sqrt N}\frac{-2}{1+\delta_{ij}}\E_N\Big[e^{\i tX_N}G_{ij}(\xi) \int_{\Gamma_+}f(\xi')G_{ij}'(\xi')d\xi'\Big]\cdot\mathds1_{\Omega_V(\varsigma)}\\
		=\frac{\i t}{\sqrt N}\frac{-2}{1+\delta_{ij}}\int_{\Gamma_+}f(\xi')\E_N\Big[e^{\i tX_N}G_{ij}(\xi)G_{ij}'(\xi')\Big]d\xi'\cdot\mathds1_{\Omega_V(\varsigma)}
	\end{multline}
	Putting \eqref{eqn32} into \eqref{eqn30}, we have :
	\begin{multline}\label{eqn28}
		(\xi-\lambda v_i)\E_N[e^{\i tX_N}(G_{ii}(\xi)-\E_N G_{ii}(\xi))]\\
		=\frac{-1}{N}\sum_j\Big(\E_N[e^{\i tX_N}(G_{ii}(\xi)G_{jj}(\xi)+(G_{ij}(\xi))^2)]-\E_N[e^{\i tX_N}]\E_N[(G_{ii}(\xi)G_{jj}(\xi)+(G_{ij}(\xi))^2)]\Big)\\
		-\frac{2\i t}{N^{3/2}}\sum_{j}\int_{\Gamma_+}f(\xi')\E_N\Big[e^{\i tX_N}G_{ij}(\xi) G_{ij}'(\xi')\Big]d\xi'\mathds1_{\Omega_V(\varsigma)}
		+\mathcal E_1^{(i)}(\xi)+\mathcal E_2^{(i)}(\xi)\\
		=-\frac{1}{N}\E_N[e^{\i tX_N}(G_{ii}(\xi)\tr G(\xi)-\E_N[G_{ii}(\xi)\tr G(\xi)])]+\mathcal E_1^{(i)}(\xi)+\mathcal E_2^{(i)}(\xi)\\
		-\frac{1}{N}\E_N[e^{\i tX_N}(G_{ii}'(\xi)-\E_NG_{ii}'(\xi))]-\frac{2\i t}{N^{3/2}} \int_{\Gamma_+}f(\xi')\E_N\Big[e^{\i tX_N}(G(\xi)G'(\xi'))_{ii}\Big]d\xi'\mathds1_{\Omega_V(\varsigma)}
	\end{multline}
	
	Notice 
	\begin{multline}\label{eqn29}
		-\frac{1}{N}\E_N[e^{\i tX_N}(G_{ii}\tr G-\E_N[G_{ii}\tr G])]=-m_{fc}\E_N[e^{\i tX_N}(G_{ii}-\E_NG_{ii})]-\frac{g_i}{N}\E_N[e^{\i tX_N}(\tr G-\E_N\tr G)]\\
		+\E_N[e^{\i tX_N}(G_{ii}-\E_NG_{ii})(m_{fc}-\frac{1}{N}\tr G)]
		+\frac{1}{N}\E_N[e^{\i tX_N}(\tr G-\E_N\tr G)](g_i-\E_NG_{ii})\\
		+\frac{1}{N}\E_N[e^{\i tX_N}]\E_N[G_{ii}(\tr G-\E_N\tr G)]
	\end{multline} 
	Plugging \eqref{eqn29} into \eqref{eqn28}, moving the term $-m_{fc}(\xi)\E_N[e^{\i tX_N}(G_{ii}(\xi)-\E_NG_{ii}(\xi))]$ to the left hand side, multiplying both sides by $-g_i(\xi)$ and taking $\sum_i$, using the definition of $\mathcal E_3(\xi)$, we have 
	\begin{multline}\label{eqn33}
		\E_N[e^{\i t X_N}(\tr G(\xi)-\E_N\tr G(\xi))]=
		-\sum_ig_i(\xi)\E_N[e^{\i tX_N}(G_{ii}(\xi)-\E_NG_{ii}(\xi))(m_{fc}(\xi)-\frac{1}{N}\tr G(\xi))]\\
		+\frac{1}{N}\sum_ig_i^2(\xi)\E_N[e^{\i t X_N}(\tr G(\xi)-\E_N\tr G(\xi))]-\frac{1}{N}\sum_ig_i(\xi)(g_i(\xi)-\E_NG_{ii}(\xi))\E_N[e^{\i tX_N}(\tr G(\xi)-\E_N\tr G(\xi))]\\
		-\frac{1}{N}\E_N[e^{\i tX_N}]\sum_i g_i(\xi)\E_N[G_{ii}(\xi)(\tr G(\xi)-\E_N\tr G(\xi))]+\mathcal E_3(\xi).
	\end{multline}
	Moving the term $\frac{1}{N}\sum_ig_i^2(\xi)\E_N[e^{\i t X_N}(\tr G(\xi)-\E_N\tr G(\xi))]$ to the left hand side:
	\begin{multline}
		(1-\frac{1}{N}\sum_ig_i^2(\xi))	\E_N[e^{\i t X_N}(\tr G(\xi)-\E_N\tr G(\xi))]\\
		=
		-\sum_ig_i(\xi)\E_N[e^{\i tX_N}(G_{ii}(\xi)-\E_NG_{ii}(\xi))(m_{fc}(\xi)-\frac{1}{N}\tr G(\xi))]\\
		-\frac{1}{N}\sum_ig_i(\xi)(g_i(\xi)-\E_NG_{ii}(\xi))\E_N[e^{\i tX_N}(\tr G(\xi)-\E_N\tr G(\xi))]\\
		-\frac{1}{N}\E_N[e^{\i tX_N}]\sum_i g_i(\xi)\E_N[G_{ii}(\xi)(\tr G(\xi)-\E_N\tr G(\xi))]+\mathcal E_3(\xi)\\
		=	-\sum_ig_i(\xi)\E_N[e^{\i tX_N}(G_{ii}(\xi)-g_i(\xi))(m_{fc}(\xi)-\frac{1}{N}\tr G(\xi))]\\
		-\E_N[e^{\i tX_N}]\sum_i g_i(\xi)\E_N[(G_{ii}(\xi)-g_i(\xi))(\frac{1}{N}\tr G(\xi)-  m_{fc}(\xi))]+\mathcal E_3(\xi)
	\end{multline}
	where we used the fact that $g_i$ and $\E_NG_{ii}$ are $\sigma(V)$-measurable.
\end{proof}
\begin{proof}[Proof of Lemma \ref{lemma:11111}]
	First we notice that the conditions in Definition \ref{definition:constants} imply: 
	\begin{align}\label{eq95}
		5\varpi+2\varsigma-\frac{1}{2}<0.
	\end{align}
	According to \eqref{eq94},
	for any $\xi\in{\mathbb C }\backslash{\mathbb R }$:
	\begin{multline}\label{eq97}
		\vert \hat m_{fc}'-(1+\hat m_{fc}')(\frac{1}{N}\sum g_i^2)\vert =\frac{1}{N}\Big\vert \sum\frac{1+\hat m_{fc}'}{(\lambda v_i-\xi-\hat m_{fc})^2}-\frac{1}{N}\sum\frac{1+\hat m_{fc}'}{(\lambda v_i-\xi- m_{fc})^2}\Big\vert \\
		\le \frac{\vert 1+\hat m_{fc}'\vert }{N}\sum\Big\vert \frac{(\lambda v_i-\xi-m_{fc}(\xi))+(\lambda v_i-\xi-\hat m_{fc}(\xi))}{(\lambda v_i-\xi-m_{fc}(\xi))^2(\lambda v_i-\xi-\hat m_{fc}(\xi))^2}\Big\vert \vert m_{fc}-\hat m_{fc}\vert \le \vert 1+\hat m_{fc}'\vert \cdot\frac{2}{\vert \Im \xi\vert ^3}\vert m_{fc}-\hat m_{fc}\vert .
	\end{multline}
	If $\xi\in\Gamma_+$, then $\vert \Im \xi\vert \le d<1$ and
	\begin{align}\label{eq96}
		\vert 1+\hat m_{fc}'\vert \le 1+\vert \Im\xi\vert ^{-2}\le2\vert \Im\xi\vert ^{-2}.
	\end{align}
	Now suppose $N$ is large enough and $\Omega_V(\varsigma)$ holds.  By Lemma \ref{lemma:properties_of_good_event} we have $\Gamma_+\subset\mathcal D_\varsigma'$. According to Lemma \ref{lemma:hatm_fc_close_tom_fc}, \eqref{eq97} and \eqref{eq96}, if $\xi\in\Gamma_+$, then 
	$$\vert \hat m_{fc}'-(1+\hat m_{fc}')(\frac{1}{N}\sum g_i^2)\vert \le  \frac{4}{\vert \Im \xi\vert ^5} N^{2\varsigma-\frac{1}{2}}\le 4N^{5\varpi+2\varsigma-\frac{1}{2}}=o(1)\quad\text{(by \eqref{eq95})}$$
	and therefore
	$$\vert 1+\hat m_{fc}'\vert \Big\vert 1-\frac{1}{N}\sum_ig_i^2(\xi)\Big\vert =\Big\vert 1+\Big(\hat m_{fc}'-(1+\hat m_{fc}')(\frac{1}{N}\sum g_i^2)\Big)\Big\vert \ge\frac{2}{3}.$$
	The last inequality together with \eqref{eq96} completes the proof.
\end{proof}
\begin{proof}[Proof of Lemma \ref{lemma:computation}]
According to \eqref{eqn15}, 
\begin{align}\label{eq100}
	G_{ii}(\xi)=\frac{-1}{\xi+Q_i(\xi)}\quad\text{and}\quad Q_i-\hat m_{fc}+\lambda v_i=\frac{1}{\hat g_i}-\frac{1}{G_{ii}}.
\end{align} 
Using $\frac{-1}{a}=\frac{1}{b}+\frac{a+b}{b^2}+\frac{(a+b)^2}{b^3}+\frac{-1}{a}\frac{(a+b)^3}{b^3}$ with $a=\xi+Q_i$ and $b=-\xi-\hat m_{fc}+\lambda v_i$ we have
\begin{multline*}
	\frac{-1}{\xi+Q_i}=\frac{1}{-\xi-\hat m_{fc}+\lambda v_i}+\frac{Q_i-\hat m_{fc}+\lambda v_i}{(-\xi-\hat m_{fc}+\lambda v_i)^2}+\frac{(Q_i-\hat m_{fc}+\lambda v_i)^2}{(-\xi-\hat m_{fc}+\lambda v_i)^3}+\frac{-1}{\xi+Q_i}\frac{(Q_i-\hat m_{fc}+\lambda v_i)^3}{(-\xi-\hat m_{fc}+\lambda v_i)^3}\\
	=\hat g_i(\xi)+\hat g_i^2(\xi)(Q_i-\hat m_{fc}+\lambda v_i)+\hat g_i^3(\xi)(Q_i-\hat m_{fc}+\lambda v_i)^2-\frac{\hat g_i^3(\xi)}{\xi+Q_i}(Q_i-\hat m_{fc}+\lambda v_i)^3
\end{multline*}
and therefore
\begin{multline}\label{eqn16}
	\tr G(\xi)-N\hat m_{fc}(\xi)=\big(\sum_{i=1}^NG_{ii}(\xi)\big)-N\hat m_{fc}(\xi)=\Big(\sum_{i=1}^N\frac{-1}{\xi+Q_i}\Big)-N\hat m_{fc}(\xi)\\
	=\sum_{i=1}^N\Big(\hat g_i(\xi)+\hat g_i^2(\xi)(Q_i-\hat m_{fc}+\lambda v_i)+\hat g_i^3(\xi)(Q_i-\hat m_{fc}+\lambda v_i)^2-\frac{\hat g_i^3(\xi)}{\xi+Q_i}(Q_i-\hat m_{fc}+\lambda v_i)^3\Big)-N\hat m_{fc}(\xi)\\
	=\sum_{i=1}^N\Big(\hat g_i^2(\xi)(Q_i-\hat m_{fc}+\lambda v_i)+\hat g_i^3(\xi)(Q_i-\hat m_{fc}+\lambda v_i)^2-\frac{\hat g_i^3(\xi)}{\xi+Q_i}(Q_i-\hat m_{fc}+\lambda v_i)^3\Big)\quad\text{(by Lemma \ref{lemma:self_consistent_equation})}\\
	=\sum_{i=1}^N\Big(\hat g_i^2(\xi)(Q_i-\hat m_{fc}+\lambda v_i)+\hat g_i^3(\xi)(Q_i-\hat m_{fc}+\lambda v_i)^2+\hat g_i^3(\xi)G_{ii}(\xi)\Big(\frac{1}{\hat g_i(\xi)}-\frac{1}{G_{ii}(\xi)}\Big)^3\Big)\quad\text{(by \eqref{eq100})}\\
	=\sum_{i=1}^N\hat g_i^2(\xi)(Q_i-\hat m_{fc}+\lambda v_i)+\sum_{i=1}^N\hat g_i^3(\xi)(Q_i-\hat m_{fc}+\lambda v_i)^2+\sum_{i=1}^N\frac{(G_{ii}(\xi)-\hat g_i(\xi))^3}{(G_{ii}(\xi))^2}
\end{multline} 
Notice that $W_{ij}$ is independent of the sigma algebra generated by $V$ and $\{G_{pq}^{(i)}\vert p,q\ne i\}$. So if $j_1,\ldots,j_k\in\{1,\ldots,N\}$ and $p,q,r,t\in\{1,\ldots,N\}\backslash\{i\}$, then
\begin{align}\label{eq101}
	\E_N[W_{ij_1}\cdots W_{ij_k}G_{pq}^{(i)}]=\E[W_{ij_1}\cdots W_{ij_k}]\E_N[G_{pq}^{(i)}],\quad \E_N[W_{ij_1}\cdots W_{ij_k}G_{pq}^{(i)}G_{rt}^{(i)}]=\E[W_{ij_1}\cdots W_{ij_k}]\E_N[G_{pq}^{(i)}G_{rt}^{(i)}]
\end{align}
Notice that $\hat g_i$ and $\hat m_{fc}$ are $\sigma(V)$-measurable. By \eqref{eq101} we have 
\begin{multline}\label{eqn19}
	\E_N[\sum_{i=1}^N\hat g_i^2(\xi)(Q_i-\hat m_{fc}+\lambda v_i)]=\sum_{i=1}^N\hat g_i^2(\xi)\Big[\E_N[\sum_{p,q}^{(i)}W_{ip}G_{pq}^{(i)}W_{qi}]-\hat m_{fc}(\xi)\Big]\\
	=\sum_{i=1}^N\hat g_i^2(\xi)\Big[\frac{1}{N}\E_N[\sum_{p}^{(i)}G_{pp}^{(i)}]-\hat m_{fc}(\xi)\Big]
	=\sum_{i=1}^N\hat g_i^2(\xi)\Big[\frac{1}{N}\E_N[\sum_{p}^{(i)}(G_{pp}-\frac{(G_{ip})^2}{G_{ii}})]-\hat m_{fc}(\xi)\Big]\quad\text{(by \eqref{eqn3})}\\
	=\sum_{i=1}^N\hat g_i^2(\xi)\E_N[\frac{1}{N}\tr G(\xi)-\hat m_{fc}(\xi)]-\frac{1}{N}\sum_{i=1}^N\hat g_i^2(\xi)\E_N[G_{ii}(\xi)+\frac{1}{G_{ii}(\xi)}\sum_p^{(i)}(G_{ip}(\xi))^2].
\end{multline}
Similarly, using \eqref{eq101},
\begin{multline}\label{eqn17}
	\E_N\Big[\sum_{i=1}^N\hat g_i^3(\xi)(Q_i-\hat m_{fc}+\lambda v_i)^2\Big]=\sum_{i=1}^N\hat g_i^3(\xi)\E_N\Big[\Big(-W_{ii}-\hat m_{fc}+\sum_{p,q}^{(i)}W_{ip}G_{pq}^{(i)}W_{qi}\Big)^2\Big]\\
	=\sum_{i=1}^N\hat g_i^3(\xi)\Big(\frac{2}{N}+\hat m_{fc}^2-\frac{2}{N}\hat m_{fc}\E_N[\tr G^{(i)}]+\E_N\Big[\Big(\sum_{p,q}^{(i)}W_{ip}G_{pq}^{(i)}W_{qi}\Big)^2\Big]\Big)
\end{multline}
Considering the cases $p=q=r=t$, $p=q\ne r=t$, $p=t\ne q=r$ and $p=r\ne q=t$  we have
\begin{multline}\label{eqn18}
	\E_N\Big[\Big(\sum_{p,q}^{(i)}W_{ip}G_{pq}^{(i)}W_{qi}\Big)^2\Big]=\sum_{p,q,r,t}^{(i)}\E_N\Big[W_{ip}W_{qi}W_{ir}W_{ti}G_{pq}^{(i)}G_{rt}^{(i)}\Big]\\
	=\sum_{p}^{(i)}\E[W_{ip}^4]\E_N\Big[(G_{pp}^{(i)})^2\Big]+\frac{1}{N^2}\sum_{p\ne r}^{(i)}\E_N[G_{pp}^{(i)}G_{rr}^{(i)}]+\frac{2}{N^2}\sum_{p\ne q}^{(i)}\E_N[(G_{pq}^{(i)})^2]\quad\text{(by \eqref{eq101})}\\
	=\sum_{p}^{(i)}\E[W_{ip}^4]\E_N\Big[(G_{pp}^{(i)})^2\Big]+\frac{1}{N^2}\Big(\E_N[(\tr G^{(i)})^2]-\sum_p^{(i)}\E_N[(G_{pp}^{(i)})^2]\Big)+\frac{2}{N^2}\Big(\E_N[\tr( G^{(i)}G^{(i)})]-\sum_p^{(i)}\E_N[(G_{pp}^{(i)})^2]\Big)\\
	=\sum_{p}^{(i)}\E[W_{ip}^4]\E_N\Big[(G_{pp}^{(i)})^2\Big]+\frac{1}{N^2} \E_N\Big[(\tr G^{(i)})^2+2(\tr G^{(i)})'-3\sum_p^{(i)}(G_{pp}^{(i)})^2\Big] 
\end{multline}
Plug \eqref{eqn19}, \eqref{eqn17} and \eqref{eqn18}  into \eqref{eqn16}, then we have:
\begin{multline*}
	\E_N[\tr G(\xi)-N\hat m_{fc}(\xi)]=\sum_{i=1}^N\hat g_i^2(\xi)\E_N[\frac{1}{N}\tr G(\xi)-\hat m_{fc}(\xi)]-\frac{1}{N}\sum_{i=1}^N\hat g_i^2(\xi)\E_N[G_{ii}(\xi)+\frac{1}{G_{ii}(\xi)}\sum_p^{(i)}(G_{ip}(\xi))^2]\\
	+\sum_{i=1}^N\hat g_i^3(\xi)\Big(\frac{2}{N}+\frac{2}{N^2}\E_N[(\tr G^{(i)})']+\sum_p^{(i)}\E_N[(G_{pp}^{(i)})^2](\E[W_{ip}^4]-\frac{3}{N^2})+\E_N[(\hat m_{fc}-\frac{1}{N}\tr G^{(i)})^2]\Big)
	+\E_N[\sum_{i=1}^N\frac{(G_{ii}(\xi)-\hat g_i(\xi))^3}{(G_{ii}(\xi))^2}]
\end{multline*}
Moving the first term on the RHS to the LHS and using \eqref{eqn21}, we complete the proof.
\end{proof}
\begin{proof}[Proof of Lemma \ref{lemma:regularity_of_S}]
	Comparing the imaginary part of both sides of \eqref{eq38}:
$$0=R''(\gamma)\Re(z-\gamma)\Im z+\frac{R'''(\gamma)}{2}(\Re(z-\gamma))^2\Im z-\frac{R'''(\gamma)}{6}(\Im z)^3+\sum_{j=4}^\infty\frac{R^{(j)}(\gamma)}{j!}\Im((z-\gamma)^j),\quad\forall z\in S^+\cap Q_N.$$
In the above equation $\Im z\cdot R''(\gamma)\ne0$, so we can divide both sides by $\Im z\cdot R''(\gamma)$ and have:
\begin{align}\label{eq34}
	X-\frac{\alpha}{2}X^2+\frac{\alpha}{6}Y^2+H(X,Y)=0,\quad\forall z\in S^+\cap Q_N
\end{align}
where
\begin{itemize}
	\item $X=\Re(z-\gamma)$, $Y=\Im z$;
	\item $\alpha=-\frac{R'''(\gamma)}{R''(\gamma)}>0;$
	\item $H(x,y)=\sum_{j=4}^\infty\frac{R^{(j)}(\gamma)}{j!R''(\gamma)}\frac{\Im((x+\i y)^j)}{y}$.
\end{itemize}
According to (6.46) of \cite{Baik+Lee}, we have
\begin{align}\label{lemma:(6.46)_in_Baik+Lee}
	\vert \Im ((x+\i y)^j)\vert \le j\cdot\vert x+\i y\vert ^{j-1}\cdot\vert y\vert \quad\quad\forall x,y\in{\mathbb R }, j\in\{1,2,\ldots\}.
\end{align} 
By Lemma \ref{lemma:analogue_of_lemma_6.2} and  \eqref{lemma:(6.46)_in_Baik+Lee}, if $N$ is large enough then
\begin{align}\label{eq35}
	\frac{2}{W_2^2}N^{1-3\tau_0-\tau_2}\le\alpha\le2W_2^3N^{1+2\tau_0+\tau_2}
\end{align}
and
\begin{multline*}
	\vert H(X,Y)\vert \le \sum_{j=4}^\infty W_2^jN^{-2+2\tau_0+\tau_2+j}\vert X+\i Y\vert ^{j-1}\le2W_2^4N^{2\tau_0+\tau_2+2}\vert X+\i Y\vert ^3\\
	\le2W_2^4N^{2\tau_0+\tau_2-7}(X^2+Y^2),\quad\forall z=X+\gamma+\i Y\in S^+\cap Q_N
\end{multline*}
where $W_2>0$ is defined in Lemma \ref{lemma:analogue_of_lemma_6.2}. We used the condition $\vert X+\i Y\vert <N^{-9}$ in the last inequality. So if  $z=X+\gamma+\i Y\in S^+\cap Q_N$ and $N$ is large enough, then we can write
\begin{align}\label{eq37}
	H(X,Y)=H^{(1)}(X,Y)+H^{(2)}(X,Y)
\end{align}
where 
\begin{align}\label{eq36}
	\vert H^{(1)}(X,Y)\vert \le 2W_2^4N^{2\tau_0+\tau_2-7}X^2\quad\text{and}\quad\vert H^{(2)}(X,Y)\vert \le 2W_2^4N^{2\tau_0+\tau_2-7}Y^2.
\end{align}
By \eqref{eq34}, \eqref{eq37} and \eqref{eq36}, if  $z=X+\gamma+\i Y\in S^+\cap Q_N$ and $N$ is large enough, then
$$X(1-\frac{\alpha}{2}X+\frac{H^{(1)}(X,Y)}{X})+\frac{\alpha}{6}Y^2(1+\frac{6H^{(2)}(X,Y)}{\alpha Y^2})=0$$
where (by \eqref{eq35}, \eqref{eq36} and the definition of $Q_N$)
$$\vert -\frac{\alpha}{2}X+\frac{H^{(1)}(X,Y)}{X}\vert \le2W_2^3N^{-5}<\frac{1}{2}\quad\text{and}\quad\vert \frac{6H^{(2)}(X,Y)}{\alpha Y^2}\vert \le6W_2^6N^{-1}<\frac{1}{2}$$
thus  we have
$$\frac{-X}{Y^2}\in[\frac{\alpha}{18},\frac{\alpha}{2}]\subset [\frac{1}{9W_2^2}N^{1-3\tau_0-\tau_2},W_2^3N^{1+2\tau_0+\tau_2}]\quad\text{(by \eqref{eq35})}$$
and 
$$\vert X\vert \le Y^2\cdot W_2^3N^{1+2\tau_0+\tau_2}\le Y$$
which implies \eqref{eq39}.
\end{proof}

\section*{Funding}
The authors were partially supported by National Research Foundation of Korea under grant number NRF-2019R1A5A1028324.


\begin{thebibliography}{99}  
	
\bibitem{Baik+Lee} 	Jinho Baik and Ji Oon Lee, Fluctuations of the free energy of the  spherical Sherrington-Kirkpatrick model, {\it J. Stat. Phys.} {\bf 165}, 2016.

\bibitem{Baik+Lee2_____} 	Jinho Baik and Ji Oon Lee, Fluctuations of the free energy of the spherical Sherrington–Kirkpatrick
model with ferromagnetic interaction, {\it Ann. Henri Poincar\'e} {\bf 18}, 2017.

\bibitem{Baik+Lee3_____} 	Jinho Baik and Ji Oon Lee, Free energy of bipartite spherical Sherrington–Kirkpatrick model, {\it Ann. Inst. H. Poincaré Probab. Statist. } {\bf 56}, 2020.

\bibitem{Baik+Lee+Wu_____} 	J. Baik, J. O. Lee and H. Wu, Ferromagnetic to paramagnetic transition in spherical spin glass, {\it Journal of Statistical Physics, } {\bf 173}, 2018.

	
\bibitem{Benaych-Georges+Knowles} Florent Benaych-Georges and Antti Knowles, Lectures on the local semicircle law for Wigner matrices, In {\it Advanced Topics in Random Matrices}, Panoramas et Synth\`eses {\bf 53}, Soci\'et\'e Math\'ematique de France, 2016.
	
\bibitem{Biane} P. Biane, On the free convolution with a semi-circular distribution, {\it Indiana Univ. Math. J.} {\bf 46}, 1997.
	


\bibitem{Billingsley} P.  Billingsley, {\it Convergence of Probability Measures}, John Wiley \& Sons, Inc., New
York-London-Sydney, 1968.



\bibitem{Crisanti+Sommers} A. Crisanti and H.-J. Sommers, The spherical p-spin interaction spin glass model: the statics, {\it Zeitschrift f\"ur Physik B Condensed Matter} {\bf 87}, 1992.

\bibitem{Edwards+Anderson} S. F. Edwards and P. W. Anderson, Theory of spin glasses, {\it J. Phys. F} {\bf 5(5)}, 1975.

\bibitem{Erdos+Yau} L. Erd{\H o}s and H.-T. Yau, {\it A Dynamical Approachto Random Matrix Theory}, American Mathematical Society, Providence, 2017 

\bibitem{Guerra} F. Guerra, Broken replica symmetry bounds in the mean field spin glass model, {\it Commun. Math. Phys.} {\bf 233(1)}, (2003)


\bibitem{Ji+Lee} Hong Chang Ji and Ji Oon Lee, Central limit theorem for linear spectral statistics of deformed Wigner matrices, {\it Random Matrices: Theory Appl.} {\bf 9}, 2020.

\bibitem{Kosterlitz+Thouless+Jones_____} J. Kosterlitz, D. Thouless, and R. C. Jones, Spherical model of a spin-glass, {\it Phys. Rev. Lett.} {\bf 36}, 1976.
	

\bibitem{Landon} B. Landon, Free energy fluctuations of the 2-spin spherical SK model
at critical temperature, Preprint, arXiv:2010.06691v1, 2020.

\bibitem{Landon+Sosoe_____} B. Landon and P. Sosoe, Fluctuations of the overlap at low temperature in the 2-spin spherical
SK model, Preprint, arXiv:1905.03317, 2019.

\bibitem{Landon+Sosoe2_____} B. Landon and P. Sosoe, Fluctuations of the 2-spin SSK model with magnetic field, Preprint, arXiv:2009.12514, 2020.
	
\bibitem{Lee+Schnelli} Ji Oon Lee and Kevin Schnelli, Extremal eigenvalues and eigenvectors of deformed Wigner matrices, {\it Probab. Theory Relat. Fields} {\bf 164}, 2016.

\bibitem{Lee+Schnelli2} Ji Oon Lee and Kevin Schnelli, Local deformed semicircle law and complete delocalization for Wigner matrices with
random potential, {\it Journal of Mathematical Physics} {\bf 54}, 2013.

\bibitem{Lee+Schnelli3} Ji Oon Lee and Kevin Schnelli, Local law and Tracy–Widom limit for sparse random
matrices, {\it Probab. Theory Relat. Fields} {\bf 171}, 2018.

\bibitem{Lee+Schnelli_edge_universality} Ji Oon Lee and Kevin Schnelli, Edge universality for deformed Wigner matrices, {\it Reviews in Mathematical Physics} {\bf 27}, 2015.

\bibitem{Lee+Schnelli+Stetler+Yau} Ji Oon Lee, Kevin Schnelli, Ben Stetler and Horng-Tzer Yau, Bulk universality for deformed Wigner matrices, {\it Annals of Probability} {\bf 44}, 2016.



\bibitem{Li+Schnelli+Xu} Yiting Li, Kevin Schnelli and Yuanyuan Xu, Central limit theorem for mesoscopic eigenvalue statistics of deformed
Wigner matrices and sample covariance matrices, {\it Ann. Inst. H. Poincaré Probab. Statist. } {\bf 57}, 2021.

\bibitem{Nguyen+Sosoe_____} V. L. Nguyen and P. Sosoe, Central limit theorem near the critical temperature for the overlap
in the 2-spin spherical SK model, {\it J. Math. Phys.} {\bf 60}, 2019.


\bibitem{Pastur} L. Pastur, On the spectrum of random matrices, {\it Theor. Math. Phys.} {\bf 10}, 1972.

\bibitem{Parisi} G. Parisi, A sequence of approximated solutions to the {SK} model for spin glasses, {\it J. Phys. A} {\bf 13(4)}, 1980.

\bibitem{Sherrington+Kirkpatrick} D. Sherrington and S. Kirkpatrick, Solvable model of a spin-glass, {\it Phys. Rev. Lett.} {\bf 35(26)}, 1975.

\bibitem{Talagrand_SK} M. Talagrand, The Parisi formula, {\it Ann. Math} {\bf 163(1)}, 2006.

\bibitem{Talagrand} M. Talagrand, Free energy of the spherical mean field model, {\it Probab. Theory Related Fields} {\bf 134}, 2006.




\end{thebibliography}
\end{document}